\documentclass[11pt,a4paper]{article}
\usepackage{amsmath,amssymb,latexsym,graphicx,mathrsfs,amsthm}
\usepackage{secdot}
\sectiondot{subsection}

\usepackage{enumitem}

\usepackage{xcolor}
\usepackage{url}
\usepackage{hyperref}
\definecolor{ForestGreen}{rgb}{0.1,0.6,0.05}
\definecolor{EgyptBlue}{rgb}{0.063,0.1,0.6}
\definecolor{RipeOlive}{HTML}{556B2F}
\hypersetup{
	colorlinks=true,
	linkcolor=EgyptBlue,         
	citecolor=ForestGreen,
	urlcolor=RipeOlive
}

\usepackage[hyperpageref]{backref}
\usepackage{epstopdf}
\newtheorem{theorem}{Theorem}
\newtheorem{proposition}[theorem]{Proposition}
\newtheorem{lemma}[theorem]{Lemma}

\theoremstyle{definition}

\newtheorem{remark}[theorem]{Remark}

\numberwithin{equation}{section}
\numberwithin{theorem}{section}
\numberwithin{equation}{section}
\numberwithin{theorem}{section}

\newenvironment{proof*}[1]{\begin{trivlist}\item[\hskip%
		\labelsep{{\bf Proof of \/{\rm\bf #1.}}\quad}]\rm}%
	{\hfill\qed\rm\end{trivlist}}

\newcommand{\W}{W_0^{1,p}}
\newcommand{\intO}{\int_\Omega}
\newcommand{\C}{C^1_0(\overline{\Omega})}
\newcommand{\E}{E_{\alpha,\beta}}
\newcommand{\En}{E_{\alpha_n,\beta_n}} 
\newcommand{\J}{J_{\alpha,\beta}}
\newcommand{\N}{\mathcal{N}_{\alpha,\beta}}

\usepackage[utf8]{inputenc}
\usepackage[left=2.5cm,right=2.5cm,top=2cm,bottom=2cm]{geometry}
\title{Remarks on minimizers for $(p,q)$-Laplace equations with two parameters
	\footnote{AMS Subject Classifications: 35J62, 35J20, 35P30}}
\author{ 
	\normalsize Vladimir Bobkov\\ 
	{\small  Department of Mathematics and NTIS, Faculty of Applied Sciences, University of West Bohemia}\\ 
	{\small Univerzitn\'i 8, 306 14 Plze\v{n}, Czech Republic}\\
	{\small e-mail: bobkov@kma.zcu.cz}\\[0.5em] 
	\normalsize Mieko Tanaka\\
	{\small Department of  Mathematics, 
		Tokyo University of Science}\\
	{\small Kagurazaka 1-3, Shinjyuku-ku, Tokyo 162-8601, Japan}\\
	{\small e-mail: miekotanaka@rs.tus.ac.jp} 
}

\date{}

\begin{document}
\maketitle 
	\begin{abstract} 
	We study in detail the existence, nonexistence and behavior of global minimizers, ground states and corresponding energy levels of the $(p,q)$-Laplace equation $-\Delta_p u -\Delta_q u = \alpha |u|^{p-2}u + \beta |u|^{q-2}u$ in a bounded domain $\Omega \subset \mathbb{R}^N$ under zero Dirichlet boundary condition, where $p > q > 1$ and $\alpha, \beta \in \mathbb{R}$. 
	A curve on the $(\alpha,\beta)$-plane which allocates a set of the existence of ground states and  the multiplicity of positive solutions is constructed.
	Additionally, we show that eigenfunctions of the $p$- and $q$-Laplacians under zero Dirichlet boundary condition are linearly independent.
	\par
	\smallskip
	\noindent {\bf  Keywords}:\  $p$-Laplacian,\ $(p,q)$-Laplacian,\ nonlinear eigenvalue problem, global minimizer,\ ground states,\ Nehari manifold, \ fibered functional, \ improved Poincare inequality.
	\end{abstract}


\section{Introduction}\label{sec:intro}
Consider the following generalized eigenvalue problem:
\begin{equation*}\label{eq:D}
\tag{$GEV;\alpha,\beta$}
\left\{
\begin{aligned}
-&\Delta_p u -\Delta_q u=\alpha |u|^{p-2}u+\beta |u|^{q-2}u 
&&{\rm in}\ \Omega, \\
&u=0 &&{\rm on }\ \partial \Omega,
\end{aligned}
\right.
\end{equation*}
where $1<q<p<\infty$ and $\Delta_r$ with $r=\{p,q\}$ stands for the $r$-Laplace operator formally defined by $\Delta_r u :={\rm div}\,(|\nabla u|^{r-2}\nabla u)$. 
Clearly, the assumption $q<p$ is imposed without loss of generality.
Parameters $\alpha, \beta$ are real numbers, and $\Omega\subset \mathbb{R}^N$ ($N \geq 1$) is a bounded domain with $C^2$-boundary. 

We say that $u\in\W := W_0^{1,p}(\Omega)$ is a (weak) solution of \eqref{eq:D} if the following equality is satisfied for all test functions $\varphi\in \W$:
\begin{equation*}
\intO |\nabla u|^{p-2}\nabla u\nabla\varphi \,dx 
+ \intO |\nabla u|^{q-2}\nabla u\nabla\varphi \,dx 
= \alpha \intO |u|^{p-2}u \varphi \, dx + \beta \intO |u|^{q-2}u \varphi\,dx.
\end{equation*}
It is easy to see that weak solutions of \eqref{eq:D} correspond to critical points of the $C^1$ energy functional $\E : \W \to \mathbb{R}$ defined by
\begin{equation*}\label{def:E} 
\E(u) = \frac{1}{p}\, H_\alpha (u)+\frac{1}{q}\,G_\beta(u), 
\end{equation*}
where
\begin{equation*}\label{def:HG}
H_\alpha (u) :=\|\nabla u\|_p^p -\alpha\|u\|_p^p 
\quad {\rm and}\quad 
G_\beta(u) :=\|\nabla u\|_q^q -\beta\|u\|_q^q.
\end{equation*}
Hereinafter, $\| \cdot \|_r$ denotes the norm of $L^r(\Omega)$, and $W_0^{1,r}$ is endowed with the norm $\| \nabla (\cdot) \|_r$, $r > 1$.

Let $\lambda_1(r)$ and $\varphi_r \in W_0^{1,r} \setminus \{0\}$ be the first eigenvalue and a first eigenfunction of the $r$-Laplacian in $\Omega$ under zero Dirichlet boundary condition, respectively; i.e., they weakly satisfy the problem
\begin{equation}\label{eq:eigen}
\left\{
\begin{aligned}
-&\Delta_r u = \lambda |u|^{r-2}u  &&{\rm in}\ \Omega, \\
&u=0 &&{\rm on }\ \partial \Omega.
\end{aligned}
\right.
\end{equation}
Note that $\lambda_1(r)$ is simple and isolated, cf.\ \cite{anane1987}, and it can be defined as
\begin{equation}\label{eq:lambdar}
\lambda_1(r) := \inf\left\{\frac{\|\nabla u\|_r^r}{\|u\|_r^r}:~ u \in W_0^{1,r} \setminus \{0\}
\right\}.
\end{equation}
Therefore, $\varphi_r$ is unique modulo scaling; moreover, it has a constant sign in $\Omega$, and hence we will always assume, for definiteness, that  $\varphi_r \geq 0$ and $\|\varphi_r\|_r = 1$. 
The spectrum of the $r$-Laplacian will be denoted as $\sigma(-\Delta_r)$, and the set of all eigenfunctions associated to some $\mu \in \mathbb{R}$ will be denoted as $ES(r;\mu)$. For instance, $ES(r;\lambda_1(r)) \equiv \mathbb{R}\varphi_r$. 
The simplicity of the first eigenvalue and the definition \eqref{eq:lambdar} directly imply the following facts which will be often used in our arguments. 
\begin{lemma}\label{lem:eigenvalue}
	Let $u \in \W \setminus \{0\}$. 
	Then the following assertions are satisfied:
	\begin{enumerate}[label={\rm(\roman*)}]
		\item\label{lem:eigenvalue:1} if $\alpha \leq \lambda_1(p)$, then $H_\alpha(u) \geq 0$. Moreover, $H_\alpha(u) = 0$ if and only if $\alpha = \lambda_1(p)$ and $u \in \mathbb{R}\varphi_p$.
		\item\label{lem:eigenvalue:2} if $\beta \leq \lambda_1(q)$, then $G_\beta(u) \geq 0$. Moreover, $G_\beta(u) = 0$ if and only if $\beta = \lambda_1(q)$ and $u \in \mathbb{R}\varphi_q$.
	\end{enumerate}
\end{lemma}

Boundary value problems of the type \eqref{eq:D} containing several heterogeneous operators naturally arise in a wide range of mathematical modeling issues since such hybrid operators enable to describe simultaneously various aspects of real processes, and these problems have been being actively studied nowadays, see, for instance, \cite{cahnhill,zakharov,BenciDerrick,chueshov,sun}.
In particular, investigation of problems with the sum of the $p$- and $q$-Laplace operators also attracts considerable attention, see, e.g., \cite{CherIl,figueiredo,yin,alves} and references below, where the cases of various nonlinearities and boundary conditions were considered; we also refer the reader to the recent survey \cite{marcomasconi2017}.

Problem \eqref{eq:D}, while being a formal combination of eigenvalue problems \eqref{eq:eigen} for the $p$- and $q$-Laplacians, possesses its own structure of the solution set which appears to be significantly different from the pure eigenvalue cases or similar problems with nonhomogeneous nonlinearities, see \cite{marcomasconi2017}.
For instance, based on the results of \cite{T-2014} and \cite{MT}, it was proved in \cite{BobkovTanaka2015} that \eqref{eq:D} has at least one positive solution whenever
$$
(\alpha, \beta) \in \big( (-\infty, \lambda_1(p)) \times (\lambda_1(q), \infty) \big) 
\cup 
\big( (\lambda_1(p), \infty) \times (-\infty, \lambda_1(q)) \big).
$$
Moreover, there was constructed a ``threshold'' curve $\mathcal{C}$ on the $(\alpha,\beta)$-plane, which separates sets of the existence and nonexistence of positive solutions of \eqref{eq:D}. 
The shape of $\mathcal{C}$ is different, depending on whether the following conjecture is valid or not:
\begin{itemize}
	\item[{\bf (LI)}]
	The first eigenfunctions $\varphi_p$ and $\varphi_q$ are linearly independent.
\end{itemize}
Although it is natural to anticipate that {\bf (LI)} holds true, it was shown in \cite[Appendix C]{BobkovTanaka2015} that it can be violated when weighted eigenvalue problems are considered. 
On the other hand, the existence and nonexistence of sign-changing solutions of \eqref{eq:D} were studied in \cite{marano2013,T-Uniq,T-2014,papa2,BobkovTanaka2016}. 

The present article is devoted to the detailed investigation of some energy aspects of problem \eqref{eq:D}. 
Namely, we study questions of the existence and behavior (with respect to the parameters $\alpha$ and $\beta$)
of \textit{global minimums} of $\E$ on $\W$ and on $\N$, where $\N$ is the Nehari manifold associated to \eqref{eq:D}. Corresponding minimizers, whenever they exist, will be referred as \textit{global minimizers} and \textit{ground states} of $\E$, respectively. 
Although a partial information in this direction is contained in the available literature, the complete picture has not been completely understood. 
Except for a partial result in the case 
\begin{equation*}\label{eq:p=2q}
p=2q
\quad \text{and} \quad 
(\alpha,\beta) = \left(\frac{\|\nabla \varphi_p\|_p^p}{\|\varphi_p\|_p^p}, \frac{\|\nabla \varphi_p\|_q^q}{\|\varphi_p\|_q^q}\right),
\end{equation*}
we fully characterize the existence and behavior of global minimizers and ground states of $\E$ for all $(\alpha, \beta) \in \mathbb{R}^2$, together with the corresponding energy levels. 
It appears that the geometry of the energy functional (and hence the existence of its critical points) at $\left(\frac{\|\nabla \varphi_p\|_p^p}{\|\varphi_p\|_p^p}, \frac{\|\nabla \varphi_p\|_q^q}{\|\varphi_p\|_q^q}\right)$ crucially depends on the choice of $p < 2q$, $p=2q$ or $p > 2q$. In this respect, the situation is reminiscent of the Fredholm alternative for the $p$-Laplacian, where the difference between $p<2$, $p=2$, and $p>2$ is vital, see, e.g., \cite{takac,drabek,takac2} and references therein. 
Special attention is paid also to other borderline cases. In particular, a curve $\mathcal{C}_*$ on the $(\alpha, \beta)$-plane which separates sets where the least energy on $\N$ is finite or not is constructed. Furthermore, we show that $C_*$ allocates a set of $(\alpha,\beta)$ where \eqref{eq:D} possesses at least two positive solutions. (Note that \cite{BobkovTanaka2015} contains no multiplicity results.) 
The obtained information provides the existence of positive solutions of \eqref{eq:D} for some sets of parameters which were not covered in \cite{BobkovTanaka2015}, and it gives better understanding of the properties of the solution set of \eqref{eq:D}, as well as the geometry of the corresponding energy functional. 
Finally, we show the validity of {\bf (LI)} conjecture.

The article is organized as follows. 
In Section \ref{sec:mainresults}, we formulate main results concerning global minimizers and ground states of $\E$.
Section \ref{sec:preliminaries} contains preliminary results necessary for our arguments.
In Section \ref{sec:proofs:global_minimizers}, we prove the main results for global minimizers of $\E$.
Section \ref{sec:proofs:ground_states} is devoted to the proof of the main results for ground states of $\E$.
Finally, Appendix \ref{sec:appendix} contains the proof of {\bf (LI)} conjecture.


\section{Statements of main results}\label{sec:mainresults}

We start by defining two critical values which will play an essential role in our results:
\begin{equation}\label{def:values} 
\alpha_*:=\frac{\|\nabla \varphi_q\|_p^p}{\|\varphi_q\|_p^p} 
\quad \text{and} \quad 
\beta_*:=\frac{\|\nabla \varphi_p\|_q^q}{\|\varphi_p\|_q^q}. 
\end{equation}
The following lemma states the validity of {\bf (LI)} conjecture, as well as the consequent properties of $\alpha_*$ and $\beta_*$; see Appendix \ref{sec:appendix} for the proof. 
\begin{lemma}\label{lem:LID}
	$\varphi_p$ and $\varphi_q$ are linearly independent, and hence $\alpha_* > \lambda_1(p)$ and $\beta_* > \lambda_1(q)$.
\end{lemma}

\subsection{Global minimizers}
Define an extended function $m\colon\mathbb{R}^2\to \mathbb{R}\cup\{-\infty\}$ as the \textit{global minimum} of $\E$ on $\W$:
\begin{equation}\label{def:mini}
m(\alpha,\beta):=\inf\{\E(u)\,:\, u \in \W\} 
\quad
{\rm for}\ (\alpha,\beta)\in\mathbb{R}^2. 
\end{equation}
Let us collect the basic properties of $m$.
\begin{proposition}\label{prop:GM1} 
	The following assertions are satisfied (see Fig.\ \ref{fig1}):
	\begin{enumerate}[label={\rm(\roman*)}]
	\item\label{prop:GM1:1} if $\alpha \le \lambda_1(p)$ and $\beta \le \lambda_1(q)$, then $m(\alpha,\beta)=0$ and $0$ is the unique global minimizer of $\E$; 
	\item\label{prop:GM1:2} if $\alpha < \lambda_1(p)$ and $\beta > \lambda_1(q)$, then $m(\alpha,\beta)<0$ and $\E$ has a nontrivial global minimizer;
	\item\label{prop:GM1:3} if $\alpha > \lambda_1(p)$ and $\beta \in \mathbb{R}$, then $m(\alpha,\beta)=-\infty$, that is, $\E$ has no global minimizers.
	\end{enumerate}
\end{proposition} 

\begin{remark}\label{rem:GM} 
	If $\alpha \le 0$ and $\beta > \lambda_1(q)$, then, using a D\'iaz-Sa\'a type inequality \cite{faria}, it can be proved in much the same way as \cite[Theorem 1.1]{T-Uniq} that \eqref{eq:D} has a \textit{unique} positive solution (and hence $\E$ has exactly two global minimizers since $\E$ is even). 
	Whether the same uniqueness holds true for $0 < \alpha < \lambda_1(p)$ and $\beta > \lambda_1(q)$ remains an open question. 
\end{remark}

Let us study the behavior of global minimizers when $(\alpha,\beta)$ approaches the boundary of $(-\infty,\lambda_1(p)) \times (\lambda_1(q),\infty)$.
\begin{proposition}\label{prop:GM2}  
	Let $\{\alpha_n\}_{n \in \mathbb{N}}$ and $\{\beta_n\}_{n \in \mathbb{N}}$ be such that $\alpha_n<\lambda_1(p)$ and $\beta_n>\lambda_1(q)$ for $n \in \mathbb{N}$. 
	Let $\alpha, \beta \in \mathbb{R}$ be such that $\lim\limits_{n\to \infty}\alpha_n=\alpha$ and $\lim\limits_{n\to \infty}\beta_n=\beta$. 
	Let $u_n$ be a global minimizer of $\En$ for $n \in \mathbb{N}$.
	Then the following assertions are satisfied:
	\begin{enumerate}[label={\rm(\roman*)}]
		\item\label{prop:GM2:1} if $\alpha=\lambda_1(p)$ and $\beta>\beta_*$, then $\lim\limits_{n \to \infty} \En(u_n) = -\infty$, $\lim\limits_{n \to \infty} \|u_n\|_p = \infty$, and $|u_n|/\|u_n\|_p$ converges to $\varphi_p/\|\varphi_p\|_p$ strongly in $\W$ as $n\to\infty$; 
		\item\label{prop:GM2:2} if $\alpha=\lambda_1(p)$ and $\lambda_1(q)<\beta<\beta_*$, then $\limsup\limits_{n\to\infty} \En(u_n)<0$, $\{u_n\}_{n \in \mathbb{N}}$ is bounded in $\W$, and any subsequence of $\{u_n\}_{n \in \mathbb{N}}$ has a subsequence strongly convergent in $\W$ to a global minimizer of $\E$ as $n\to\infty$; 
		\item\label{prop:GM2:3} if $\beta=\lambda_1(q)$, then $\lim\limits_{n\to \infty} \En(u_n) = 0$, $\lim\limits_{n \to \infty} \|\nabla u_n\|_p = 0$, and $|u_n|/\|\nabla u_n\|_q$ converges to  $\varphi_q/\|\nabla \varphi_q\|_q$ strongly in $W_0^{1,q}$ as $n\to\infty$; 
		\item\label{prop:GM2:4} if $\alpha=\lambda_1(p)$, $\beta=\beta_*$ and $p>2q$, 		then $\limsup\limits_{n\to\infty} \En(u_n)<0$, $\{u_n\}_{n \in \mathbb{N}}$ is bounded in $\W$, and any subsequence of $\{u_n\}_{n \in \mathbb{N}}$ has a subsequence strongly convergent in $\W$ to a global minimizer of $\E$ as $n\to\infty$; 
		\item\label{prop:GM2:5} if $\alpha=\lambda_1(p)$, $\beta=\beta_*$ and $p<2q$, then $\lim\limits_{n \to \infty} \En(u_n) = -\infty$, $\lim\limits_{n \to \infty} \|u_n\|_p = \infty$, and $|u_n|/\|u_n\|_p$ converges to $\varphi_p/\|\varphi_p\|_p$ strongly in $\W$ as $n\to\infty$. 
	\end{enumerate} 
\end{proposition} 

\begin{figure}[!h]
	\begin{center}
		\includegraphics[width=0.34\linewidth]{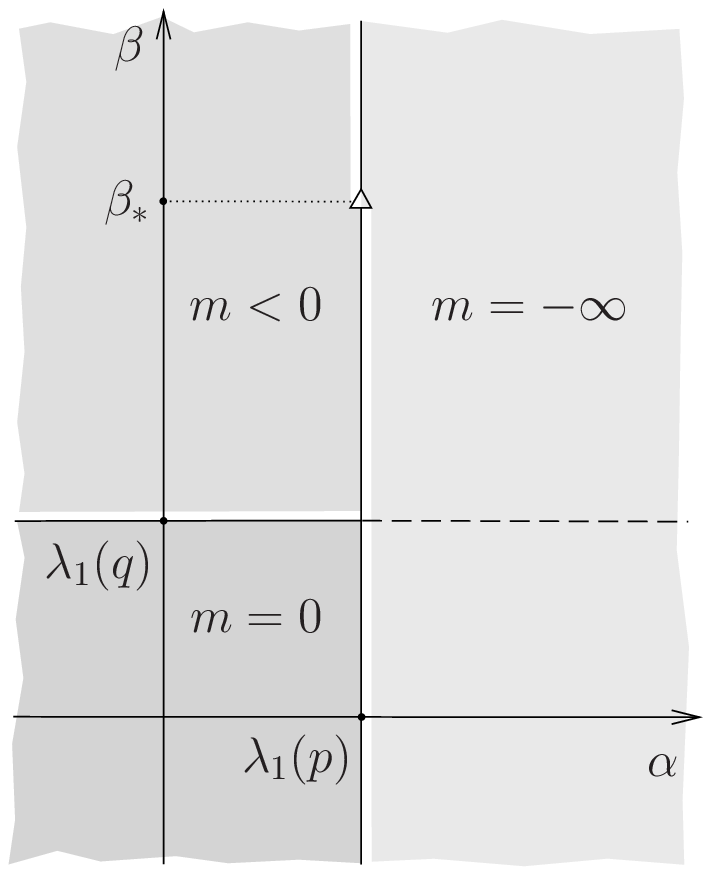}
		\caption{The global minimum $m$ of $\E$ on $\W$.}
		\label{fig1}
	\end{center}
\end{figure}

Proposition \ref{prop:GM2} allows to complement Proposition \ref{prop:GM1} with the remaining case $\alpha = \lambda_1(p)$ and $\beta > \lambda_1(q)$.
\begin{proposition}\label{prop:GM3}  
	Let $\alpha=\lambda_1(p)$ and $\beta > \lambda_1(q)$. Then $m(\alpha, \beta) < 0$. Moreover, the following assertions are satisfied:
	\begin{enumerate}[label={\rm(\roman*)}]
		\item\label{prop:GM3:1} if $\beta>\beta_*$,	then $m(\alpha,\beta)=-\infty$; 
		\item\label{prop:GM3:2} if $\lambda_1(q)<\beta<\beta_*$, then $m(\alpha,\beta)>-\infty$ and $\E$ has a global minimizer; 
		\item\label{prop:GM3:3} if $\beta=\beta_*$, then $m(\alpha,\beta)>-\infty$ if and only if $p \geq 2q$. Moreover, if $p > 2q$, then $\E$ has a global minimizer.
	\end{enumerate}
\end{proposition}

We conclude this subsection by a continuity result for $m$.
\begin{proposition}\label{prop:conti-minimum-value} 
The global minimum value $m$ defined as an extended function by \eqref{def:mini} is continuous on $\mathbb{R}^2\setminus \{\lambda_1(p)\}\times(-\infty,\beta_*]$ and discontinuous on $\{\lambda_1(p)\}\times(-\infty,\beta_*)$.
\end{proposition}


\subsection{Ground states} 
Define the Nehari manifold associated to $\E$ at $(\alpha,\beta) \in \mathbb{R}^2$ by 
\begin{equation*}\label{def:Nehari} 
\N:=\{ v\in\W\setminus\{0\}\,:\,
\langle \E^\prime(v),v \rangle=H_\alpha(v)+G_\beta(v)=0\,\}. 
\end{equation*}
Evidently, any nontrivial critical point of $\E$ belongs to $\N$. 
Define an extended function $d\colon\mathbb{R}^2\to \mathbb{R}\cup\{\pm\infty\}$ as the \textit{least energy} on $\N$, namely, 
\begin{equation*}\label{def:least-enery} 
d(\alpha,\beta):=\inf\{\E(u)\,:\, u \in \N\} 
\quad
{\rm for}\ (\alpha,\beta)\in\mathbb{R}^2,
\end{equation*}
and set $d(\alpha,\beta)=\infty$ whenever $\N = \emptyset$. 
With a slight abuse of notation, we say that $u$ is a \textit{ground state} of $\E$ if 
$$
u \in\N 
\quad  {\rm and} \quad 
\E(u) = d(\alpha, \beta).
$$

\begin{lemma}\label{lem:nonempty-Nehari} 
	$\N = \emptyset$ and hence $d(\alpha,\beta)=\infty$
	if and only if $(\alpha,\beta) \in (-\infty,\lambda_1(p)] \times (-\infty,\lambda_1(q)]$.
\end{lemma}

\begin{remark}\label{rem:positivity}
	Note that any {\it nontrivial} global minimizer of $\E$ is a ground state of $\E$. 
	On the other hand, it is shown in \cite[Lemma~2]{BobkovTanaka2015} that any ground state $u$ with $H_\alpha(u) \cdot G_\beta(u) \neq 0$ is a (nontrivial) critical point of $\E$. Therefore, the existence of a ground state $u$ with $\E(u)\neq0$ ensures that $u$ is a solution of \eqref{eq:D}. Moreover, by considering $|u|$ if necessary, we conclude that $u$ is a nonnegative solution. 
	Furthermore, the regularity result up to the boundary (\cite[Theorem~1]{Lieberman} and \cite[p.~320]{L}) and the strong maximum principle (cf.\ \cite[Theorem~5.4.1]{PS}) guarantee that $u\in \C$, $u>0$ in $\Omega$ and $\partial u/\partial\nu <0$ on $\partial\Omega$, where $\nu$ denotes a unit outward normal vector to $\partial \Omega$. 
\end{remark}

We start with some general elementary properties of ground states of $\E$.
\begin{proposition}\label{prop:property-gs} 
	Let $(\alpha, \beta) \in \mathbb{R}^2$ and let $u$ be a ground state of $\E$. 
	Then the following assertions are satisfied:
	\begin{enumerate}[label={\rm(\roman*)}]
		\item\label{prop:property-gs:1} if $\E(u)<0$, then $u$ is a local minimum point of $\E$; 
		\item\label{prop:property-gs:2} if $\E(u)>0$, then $u$ is not an extrema point of $\E$. 
	\end{enumerate}
\end{proposition}

Let us now consider the existence of ground states of $\E$. For this end, we show that for some $(\alpha, \beta) \in \mathbb{R}^2$ the least energy on $\N$ coincides with a mountain pass level of $\E$; see, e.g., \cite{jeanjean,bellazzini} for related problems. 
First, we define two mountain pass critical values for $\alpha>\lambda_1(p)$ as follows: 
\begin{align*} 
c(\alpha,\beta)&:=\inf_{\gamma\in\Gamma(\alpha,\beta)} 
\max_{t\in[0,1]} \E (\gamma(t)), 
\quad 
c^+(\alpha,\beta) :=\inf_{\gamma\in\Gamma^+(\alpha,\beta)} 
\max_{t\in[0,1]} E^+_{\alpha,\beta} (\gamma(t)),
\end{align*}
where 
\begin{align*}
\Gamma(\alpha,\beta)&:=\{\gamma\in C([0,1],\W):~ \gamma(0)=0,~ \E(\gamma(1))<0\}, \\
\Gamma^+(\alpha,\beta)&:=\{\gamma\in C([0,1],\W):~ \gamma(0)=0,~ E^+_{\alpha,\beta}(\gamma(1))<0\}, 
\end{align*}
and the functional $E^+_{\alpha,\beta}: \W \to \mathbb{R}$ is defined by
\begin{equation*}\label{def:I} 
E^+_{\alpha,\beta}(u) := \frac{1}{p}\|\nabla u\|_p^p + \frac{1}{q}\|\nabla u\|_q^q 
- \frac{\alpha}{p}\|u_+\|_p^p - \frac{\beta}{q}\|u_+\|_q^q.
\end{equation*}
Here $u_+$ denotes the positive part of $u$, that is, $u_+ := \max\{u,0\}$. 

\begin{theorem}\label{thm:MP-Nehari} 
	Let $\alpha>\lambda_1(p)$ and $\beta<\lambda_1(q)$. Then 
	$$
	c^+(\alpha,\beta) = c(\alpha,\beta) = d(\alpha,\beta) > 0
	$$ 
	and $d(\alpha,\beta)$ is attained by a positive solution of \eqref{eq:D}. 
\end{theorem}

Let us complement Theorem \ref{thm:MP-Nehari} with the case $\beta = \lambda_1(q)$.
Recall that $\alpha_*$ is defined by \eqref{def:values}.
\begin{theorem}\label{prop:MP-Nehari-resonant} 
	Let $\beta=\lambda_1(q)$. Then the following assertions are satisfied:
	\begin{enumerate}[label={\rm(\roman*)}]
		\item\label{prop:MP-Nehari-resonant:1} if $\alpha\le \lambda_1(p)$, then $d(\alpha,\beta)=\infty$; 
		\item\label{prop:MP-Nehari-resonant:2} if $\lambda_1(p)<\alpha<\alpha_*$, then $d(\alpha,\beta)>0$ and it is attained by a positive solution of \eqref{eq:D}; 
		\item\label{prop:MP-Nehari-resonant:3} if $\alpha=\alpha_*$, then $d(\alpha,\beta)=0$ and it is attained only by $t\varphi_q$ for any $t \neq 0$;
		\item\label{prop:MP-Nehari-resonant:4} if $\alpha>\alpha_*$, then $d(\alpha,\beta)=0$ and it is not attained. 
	\end{enumerate}
\end{theorem}

\begin{remark} 
	Let $\beta=\lambda_1(q)$. Note that the existence result \cite[Theorem~2.2 (ii)]{BobkovTanaka2015} does not directly imply that \eqref{eq:D} has a positive solution for $\lambda_1(p) < \alpha < \alpha_*$. On the other hand, if $\alpha>\alpha_*$, then it was shown in \cite[Proposition~4 (ii)]{BobkovTanaka2015} that \eqref{eq:D} has no positive solutions. In the remaining case $\alpha=\alpha_*$, although $d(\alpha,\beta)=0$ is attained by $t\varphi_q$ for any $t \not = 0$, it is obvious that $t\varphi_q$ is not a solution of \eqref{eq:D}, since $\varphi_q$ does not satisfy $-\Delta_p u = \alpha_* |u|^{p-2}u$, see Lemma~\ref{lem:LID}. 
\end{remark}

Now, we study the behavior of ground states when $(\alpha,\beta)$ approaches the boundary of $(\lambda_1(p),\infty) \times (-\infty,\lambda_1(q))$. 
\begin{proposition}\label{prop:behavior-gs} 
	Let $\{\alpha_n\}_{n \in \mathbb{N}}$ and $\{\beta_n\}_{n \in \mathbb{N}}$ be such that $\alpha_n>\lambda_1(p)$ and $\beta_n<\lambda_1(q)$ for $n \in \mathbb{N}$, or $\lambda_1(p) < \alpha_n < \alpha_*$ and $\beta_n \leq \lambda_1(q)$ for $n \in \mathbb{N}$.
	Let $\alpha, \beta \in \mathbb{R}$ be such that $\lim\limits_{n \to \infty} \alpha_n = \alpha$ and $\lim\limits_{n \to \infty} \beta_n = \beta$. 
	Let $u_n$ be a ground state of $\En$ for $n \in \mathbb{N}$. 
	Then the following assertions are satisfied:
	\begin{enumerate}[label={\rm(\roman*)}]
		\item\label{prop:behavior-gs:1} 
		if $\alpha=\lambda_1(p)$, then $\lim\limits_{n\to\infty}E_{\alpha_n,\beta}(u_n)=\infty$, 
		$\lim\limits_{n\to\infty}\|u_n\|_p=\infty$ and $|u_n|/\|u_n\|_p$ converges to $\varphi_p/\|\varphi_p\|_p$ strongly in $\W$ as $n \to \infty$; 
		\item\label{prop:behavior-gs:2} if $\lambda_1(p)<\alpha<\alpha_*$ and $\beta=\lambda_1(q)$, then $\{u_n\}_{n \in \mathbb{N}}$ is bounded in $\W$ and it has a subsequence strongly convergent in $\W$ to a ground state of $\E$;
		\item\label{prop:behavior-gs:3} if $\alpha\ge \alpha_*$ and $\beta=\lambda_1(q)$, then $\lim\limits_{n \to \infty} \|\nabla u_n\|_p = 0$ and $|u_n|/\|\nabla u_n\|_q$ converges to $\varphi_q/\|\nabla \varphi_q\|_q$ weakly in $\W$ and strongly in $W_0^{1,q}$ as $n\to\infty$;
	\end{enumerate}
\end{proposition}

\begin{figure}[!h]
	\begin{center}
		\includegraphics[width=0.52\linewidth]{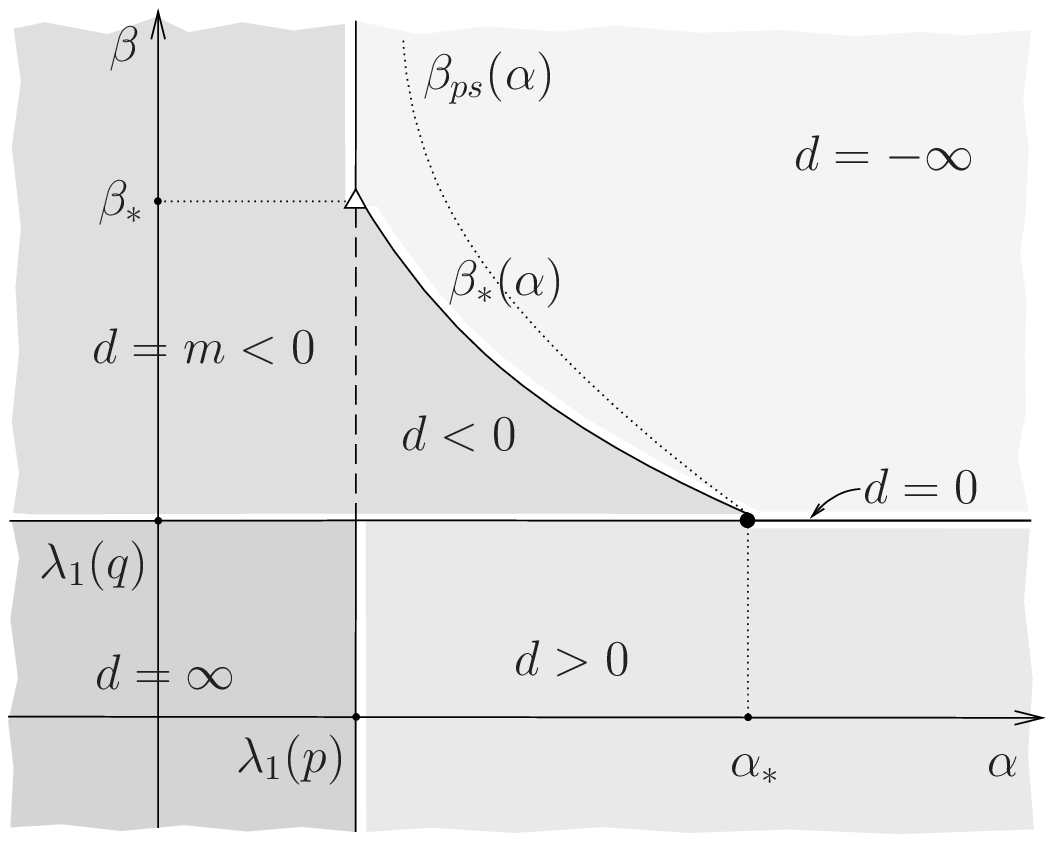}
		\caption{The least energy $d$ on $\N$.}
		\label{fig2}
	\end{center}
\end{figure}

In order to handle the existence of ground states of $\E$ in the case $\alpha \geq \lambda_1(p)$ and $\beta \geq \lambda_1(q)$, we define the following family of critical points:
\begin{equation*}\label{def:alpha_*}
\beta_*(\alpha) := \inf\left\{ \frac{\|\nabla u\|_q^q}{\|u\|_q^q}\,:\,
u\in\W\setminus \{0\} \text{ and } H_\alpha(u)\le 0\right\},
\end{equation*}
and we set $\beta_*(\alpha) = \infty$ whenever $\alpha < \lambda_1(p)$.
Let us collect the main properties of $\beta_*(\alpha)$.
\begin{proposition}\label{prop:property-curve} 
	The following assertions are satisfied (see Fig.\ \ref{fig2}):
	\begin{enumerate}[label={\rm(\roman*)}]
		\item\label{prop:property-curve:1} $\lambda_1(q) \le \beta_*(\alpha) < \infty$ for all $\alpha\ge \lambda_1(p)$; 
		\item\label{prop:property-curve:2} $\beta_*(\lambda_1(p))=\beta_*$ and $\beta_*(\alpha)=\lambda_1(q)$ for all $\alpha\ge \alpha_*$; 
		\item\label{prop:property-curve:4} $\beta_*(\alpha) > \lambda_1(q)$ for all $\alpha < \alpha_*$; 
		\item\label{prop:property-curve:3} $\beta_*(\alpha)$ is attained for all  $\alpha\ge \lambda_1(p)$; 
		\item\label{prop:property-curve:5} $\beta_*(\alpha)$ is continuous for all $\alpha > \lambda_1(p)$ and right-continuous at $\alpha = \lambda_1(p)$;
		\item\label{prop:property-curve:6} $\beta_*(\alpha)$ is (strictly) decreasing for all $\lambda_1(p) \leq \alpha \leq \alpha_*$.
	\end{enumerate} 
\end{proposition}

Let us study the existence and nonexistence of ground states of $\E$ in domains bounded by $\beta_*(\alpha)$ and two lines $\{\lambda_1(p)\}\times\mathbb{R}$ and $\mathbb{R}\times\{\lambda_1(q)\}$. 
\begin{theorem}\label{thm:negative-gs}
	Let $\alpha \ge \lambda_1(p)$. The following assertions are satisfied:
	\begin{enumerate}[label={\rm(\roman*)}]
		\item\label{thm:negative-gs:1} If $\lambda_1(q)<\beta<\beta_*(\alpha)$, then $d(\alpha,\beta)<0$ and it is attained by a positive solution of \eqref{eq:D};
		\item\label{thm:negative-gs:2} if $\beta>\beta_*(\alpha)$, then $d(\alpha,\beta)=-\infty$.
	\end{enumerate} 
\end{theorem}

\begin{remark}
According to Theorem~\ref{thm:negative-gs}, we see that the curve $\mathcal{C}_*$ defined by
$$
\mathcal{C}_*:=\{(\alpha,\beta_*(\alpha)) \in \mathbb{R}^2:~\alpha\ge \lambda_1(p) \}
$$
separates the set $[\lambda_1(p), \infty) \times (\lambda_1(q), \infty)$ with respect to the existence and nonexistence of ground states of $\E$. 
This implies that $\mathcal{C}_*$ lies below or on the curve $\mathcal{C}$ constructed in \cite{BobkovTanaka2015} in such a way that $\mathcal{C}$ is a threshold between the existence and nonexistence of positive solutions of \eqref{eq:D}. 
Namely, it holds 
$$
\beta_*(\alpha)\le \beta_{ps}(\alpha):=
\sup_{u \in {\rm int}\,\C_+} 
\inf_{\varphi \in C^1_0(\overline{\Omega})_+ \setminus\{0\}}
\mathcal{L}_\alpha(u; \varphi)
\quad 
\text{ for } 
\alpha\ge \lambda_1(p),
$$
where $\mathcal{L}_\alpha(u; \varphi)$ is the extended functional (see \cite{ilfunc,BobkovTanaka2015}) defined as
$$
\mathcal{L}_\alpha(u; \varphi) := \frac{\intO |\nabla u|^{p-2}\nabla u\nabla\varphi \,dx + 
	\intO |\nabla u|^{q-2}\nabla u\nabla\varphi \,dx - 
	\alpha\intO |u|^{p-2} u \varphi\,dx}{\intO |u|^{q-2} u \varphi\,dx},
$$
and ${\rm int}\,\C_+$ denotes the interior of the positive cone of $C^1_0(\overline{\Omega})$, that is,
$$
{\rm int}\,\C_+ := 
\left\{
u \in C^1_0(\overline{\Omega})\,:\, 
u(x)>0 \text{ for all } x \in \Omega, \,
\frac{\partial u}{\partial\nu}(x) < 0 \text{ for all } x \in \partial\Omega 
\right\}.
$$
We do not know if $\mathcal{C}_*$ and $\mathcal{C}$ coincide. However, recent results of \cite{ilyasovkaye} for a related problem with indefinite nonlinearities may indicate that $\mathcal{C}_*$ and $\mathcal{C}$ are different.
If $\beta_*(\alpha)<\beta_{ps}(\alpha)$ for some $\lambda_1(p) \le \alpha<\alpha_*$, then for any $\beta_*(\alpha)<\beta\le \beta_{ps}(\alpha)$ our equation has a positive solution which is not a ground state of $\E$. 

On the other hand, in the bounded open set $\{(\alpha,\beta) \in \mathbb{R}^2:~ \lambda_1(p)<\alpha<\alpha_*,~ \lambda_1(q)<\beta< \beta_*(\alpha)\}$
we can find two positive solutions of \eqref{eq:D}, where one of them is a ground state of $\E$ and another one has the least energy among all solutions $w$ of \eqref{eq:D} such that $\E(w) > 0$, see Theorem~\ref{thm:multi} below. 
\end{remark}

Let us study the behavior of ground states of $\E$ when $(\alpha,\beta)$ approaches the boundary of 
$\{(\alpha,\beta) \in \mathbb{R}^2:~ \lambda_1(p)<\alpha<\alpha_*,~ \lambda_1(q)<\beta< \beta_*(\alpha)\}$.
\begin{proposition}\label{prop:behavior-gs-2} 
	Let $\{\alpha_n\}_{n \in \mathbb{N}}$ and $\{\beta_n\}_{n \in \mathbb{N}}$ be such that $\lambda_1(p) < \alpha_n < \alpha_*$ and $\lambda_1(q) < \beta_n < \beta_*(\alpha_n)$ for $n \in \mathbb{N}$.
	Let $\alpha, \beta \in \mathbb{R}$ be such that $\lim\limits_{n \to \infty} \alpha_n = \alpha$ and $\lim\limits_{n \to \infty} \beta_n = \beta$. 
	Let $u_n$ be a ground state of $\En$ for $n \in \mathbb{N}$. 
	Then the following assertions are satisfied:
	\begin{enumerate}[label={\rm(\roman*)}]
		\item\label{prop:behavior-gs-2:1} if $\beta=\lambda_1(q)$, then  $\lim\limits_{n\to\infty}\|\nabla u_n\|_p = 0$ and $|u_n|/\|\nabla u_n\|_q$ converges to $\varphi_q/\|\nabla \varphi_q\|_q$ weakly in $\W$ and strongly in $W_0^{1,q}$ as $n \to \infty$;
		\item\label{prop:behavior-gs-2:2} if $\alpha=\lambda_1(p)$ and $\lambda_1(q)<\beta<\beta_*$, then $\{u_n\}_{n \in \mathbb{N}}$ is bounded in $\W$ and any subsequence of $\{u_n\}_{n \in \mathbb{N}}$ has a subsequence strongly convergent in $\W$ to a global minimizer of $\E$ as $n \to \infty$; 
		\item\label{prop:behavior-gs-2:3} if $\lambda_1(p)<\alpha<\alpha_*$ and $\beta=\beta_*(\alpha)$, then $\{u_n\}_{n \in \mathbb{N}}$ is bounded in $\W$ and any subsequence of $\{u_n\}_{n \in \mathbb{N}}$ has a subsequence strongly convergent in $\W$ to a ground state of $\E$ as $n \to \infty$;
		\item\label{prop:behavior-gs-2:4} if $\alpha=\lambda_1(p)$, $\beta=\beta_*$ and $p<2q$, then $\lim\limits_{n\to\infty}\|\nabla u_n\|_p = \infty$, $\lim\limits_{n\to\infty}\En(u_n) = -\infty$ and $|u_n|/\|u_n\|_p$ converges to $\varphi_p/\|\varphi_p\|_p$ strongly in $\W$ as $n \to \infty$.
	\end{enumerate}
\end{proposition} 

Thanks to the assertions \ref{prop:behavior-gs-2:3} and \ref{prop:behavior-gs-2:4} of Proposition \ref{prop:behavior-gs-2}, we can complement Theorem \ref{thm:negative-gs} as follows. 
\begin{theorem}\label{thm:negative-gs-2} 
	Let $\lambda_1(p) \leq \alpha < \alpha_*$ and $\beta = \beta_*(\alpha)$. Then $d(\alpha, \beta) < 0$. Moreover, the following assertions are satisfied:
	\begin{enumerate}[label={\rm(\roman*)}]
		\item\label{thm:negative-gs-2:1} if $\lambda_1(p)<\alpha<\alpha_*$, then $d(\alpha,\beta)$ is attained by a positive solution of \eqref{eq:D}; 
		\item\label{thm:negative-gs-2:2} if $\alpha=\lambda_1(p)$, then $d(\alpha,\beta) > -\infty$ if and only $p \geq 2q$. Moreover, if $p > 2q$, then $d(\alpha,\beta)$ is attained by a global minimizer of $\E$.
	\end{enumerate}
\end{theorem} 

The behavior of energy levels described in Propositions \ref{prop:behavior-gs} and \ref{prop:behavior-gs-2} indicates that \eqref{eq:D} possesses the multiplicity of positive solutions for some $\alpha > \lambda_1(p)$ and $\beta > \lambda_1(q)$. We formulate the following result in this direction. 
\begin{theorem}\label{thm:multi}
	Let $\lambda_1(p)<\alpha<\alpha_*$ and $\lambda_1(q)<\beta\le \beta_*(\alpha)$. 
	Then \eqref{eq:D} has at least two positive solutions $u_1$ and $u_2$ such that $\E(u_1)=d(\alpha,\beta)<0$, $\E(u_2)>0$ if $\beta<\beta_*(\alpha)$, and $\E(u_2)=0$ if $\beta=\beta_*(\alpha)$. 
	In particular, in the case of $\beta<\beta_*(\alpha)$, 
$u_2$ has the least energy among all solutions $w$ of \eqref{eq:D} such that $\E(w) > 0$.
\end{theorem}

We conclude this subsection by collecting some general properties of the least energy on $\N$.
Recall that $\beta_*(\alpha) = \infty$ for $\alpha<\lambda_1(p)$ and we consider the least energy $d$ as an extended function, i.e., $d: \mathbb{R}^2 \to \mathbb{R} \cup \{\pm \infty\}$.
\begin{proposition}\label{prop:property-least-energy} 
	The following assertions are satisfied:
	\begin{enumerate}[label={\rm(\roman*)}]
		\item\label{prop:property-least-energy:1} if $\alpha\le \alpha^\prime$, $\beta\le \beta^\prime$, and $(\alpha,\beta)\neq(\alpha^\prime,\beta^\prime)$, then $d(\alpha,\beta) \ge d(\alpha^\prime,\beta^\prime)$;
		\item\label{prop:property-least-energy:2} if $\alpha\le \alpha^\prime$, $\lambda_1(q) < \beta \le \beta^\prime < \beta_*(\alpha^\prime)$, and $(\alpha,\beta)\neq(\alpha^\prime,\beta^\prime)$, then $d(\alpha,\beta)>d(\alpha^\prime,\beta^\prime)$; 
		\item\label{prop:property-least-energy:3} if $\lambda_1(p)<\alpha\le\alpha^\prime$, $\beta \le \beta^\prime < \beta_*(\alpha^\prime)$ and $(\alpha,\beta)\neq(\alpha^\prime,\beta^\prime)$, then $d(\alpha,\beta)>d(\alpha^\prime,\beta^\prime)$; 
		\item\label{prop:property-least-energy:4} $d(\alpha,\beta)$ is upper semicontinuous on $\mathbb{R}^2$; 
		\item\label{prop:property-least-energy:5} $d(\alpha,\beta)$ is continuous on the following set: 
		$$
		\mathbb{R}^2 \setminus \left(
		(\mathbb{R} \times \{\lambda_1(q)\})
		\cup 
		\mathcal{C}_*
		\cup 
		(\{\lambda_1(p)\} \times (-\infty, \lambda_1(q)])
		\right).
		$$
	\end{enumerate}
\end{proposition}


\section{Preliminaries}\label{sec:preliminaries}
We start by noting that 
\begin{equation}\label{eq:Nehari}
\E(u)=-\frac{p-q}{pq}H_\alpha(u)=\frac{p-q}{pq}G_\beta(u) 
\qquad {\rm for\ any}\ u \in \N. 
\end{equation}
Thus, for any $u \in \N$ we see that $\E(u) \le 0$ (resp.\ $\E(u) \ge 0$) if and only if $G_\beta(u) \le 0 \le H_\alpha(u)$ (resp.\ $G_\beta(u) \ge 0 \ge H_\alpha(u)$). 

\begin{proposition}[\protect{\cite[Proposition~6]{BobkovTanaka2015}} and \protect{\cite[Lemma~2.1]{BobkovTanaka2016}}]\label{prop:minpoint}
Let $u \in \W$. If $H_\alpha(u) \cdot G_\beta(u) < 0$, then there exists a unique extrema point $t(u) > 0$ of $\E(t u)$ with respect to $t > 0$, and $t(u) u \in \N$.
In particular, if
\begin{equation*}
G_\beta(u) < 0 < H_\alpha(u) \quad ({\it resp.}\ G_\beta(u) > 0 > H_\alpha(u)), 
\end{equation*}
then $t(u)$ is the unique minimum (resp.\ maximum) point of  $\E(t u)$ with respect to $t > 0$, and $\E(t(u) u) < 0$ (resp.\ $\E(t(u) u) > 0$).
\end{proposition}

Let us now prove Lemma~\ref{lem:nonempty-Nehari}.
\begin{lemma}\label{lem:nonempty-Nehari2} 
	$\N \not = \emptyset$ if and only if $(\alpha,\beta) \in \mathbb{R}^2 \setminus (-\infty,\lambda_1(p)] \times (-\infty,\lambda_1(q)]$.
\end{lemma}
\begin{proof} 
	Assume first that $\N \not = \emptyset$. If $u \in \N$, then we apply the Poincar\'e inequality to get
	\begin{equation*}
	(\lambda_1(p) - \alpha) \|u\|_p^p \leq H_\alpha(u) = -G_\beta(u) \leq (\beta - \lambda_1(q)) \|u\|_q^q,
	\end{equation*}
	which implies that either $(\alpha, \beta) \in \mathbb{R}^2 \setminus (-\infty,\lambda_1(p)] \times (-\infty,\lambda_1(q)]$ or $(\alpha, \beta) = (\lambda_1(p), \lambda_1(q))$. 
	In the second case, we derive that $H_{\lambda_1(p)}(u) = G_{\lambda_1(q)} = 0$, and hence $u$ is a first eigenfunction of the $p$-Laplacian and $q$-Laplacian, simultaneously.
	However, it contradicts Lemma~\ref{lem:LID}, and hence the first case is the only possible. 
		
	Assume now that $(\alpha,\beta) \in \mathbb{R}^2 \setminus (-\infty,\lambda_1(p)] \times (-\infty,\lambda_1(q)]$. 
	We distinguish two cases:
	
	1. $\alpha \leq \lambda_1(p)$ and $\beta > \lambda_1(q)$. 
	In view of Lemma~\ref{lem:LID}, we have $G_\beta(\varphi_q) < 0 < H_\alpha(\varphi_q)$. Hence, Proposition~\ref{prop:minpoint} ensures the nonemptiness of $\N$. 
	
	2. $\alpha > \lambda_1(p)$. Take any $u \in \W \setminus \{0\}$ satisfying $H_\alpha(u)=0$. (The existence of such $u$ can be shown by applying the intermediate value theorem to a continuous path connecting $\E^{-1}(-\infty,0)$ and $\E^{-1}(0, \infty)$ in $\W\setminus\{0\}$.) Moreover, taking $|u|$ if necessary, we may assume that $u\ge 0$. 
	Consider three cases:
	
	(i) $G_\beta(u)=0$. In this case, we have $u \in \N$, that is, $\N \neq \emptyset$. 
	
	(ii) $G_\beta(u)<0$. Note that $u$ is a regular point of $H_\alpha$ since $\alpha>\lambda_1(p)$ and $u \ge 0$. Thus, there exists $\theta \in \W$ such that $\langle H^\prime_\alpha(u), \theta \rangle>0$, and hence $\langle H^\prime_\alpha(\cdot), \theta \rangle>0$ in a neighborhood of $u$. Therefore, we have $H_\alpha(u+t\theta)=\int_0^t \langle H_\alpha^\prime(u+s\theta), \theta \rangle \, ds > 0$ for sufficiently small $t>0$. Moreover, since $G_\beta(u)<0$, we can choose $t>0$ smaller, if necessary, to get $G_\beta(u+t\theta) < 0 < H_\alpha(u+t\theta)$. Hence, applying Proposition \ref{prop:minpoint}, we see that $\N \neq \emptyset$.
	
	(iii) $G_\beta(u)>0$. Arguing as above, we can find $\theta \in \W$ satisfying $\langle H_\alpha^\prime(u),\theta \rangle < 0$, and hence $G_\beta(u+t\theta)>0>H_\alpha(u+t\theta)$ for $t>0$ small enough. Therefore, Proposition \ref{prop:minpoint} leads to the desired conclusion. 
\end{proof}

\subsection{Behavior of sequences} 

The following two lemmas are similar to \cite[Lemma~3.3]{BobkovTanaka2016} and will be needed for further arguments. 
\begin{lemma}\label{lem:bdd-PS}
Let $\{\alpha_n\}_{n \in \mathbb{N}}$, $\{\beta_n\}_{n \in \mathbb{N}} \subset \mathbb{R}$, and $\{u_n\}_{n \in \mathbb{N}}\subset \W \setminus \{0\}$ be sequences satisfying 
\begin{equation*}\label{eq:bdd-PS-0} 
\alpha_n\to\alpha,
\quad \beta_n\to\beta,
\quad \|u_n\|_p \to \infty,
\quad {\rm and} \quad 
\frac{\|\En'(u_n)\|_{(\W)^*}}{\|u_n\|_p^{p-1}} \to 0
\end{equation*}
as $n \to \infty$. 
Then the sequence $\{v_n\}_{n \in \mathbb{N}}$, where $v_n := u_n/\|u_n\|_p$ for $n \in \mathbb{N}$, has a subsequence strongly convergent in $\W$ to some $v_0 \in ES(p;\alpha)\setminus\{0\}$, that is, $\alpha\in\sigma(-\Delta_p)$. 

In particular, if $u_n$ is nonnegative for $n \in \mathbb{N}$, then $v_0 = \varphi_p/\|\varphi_p\|_p$ and  $\alpha=\lambda_1(p)$. 
\end{lemma}
\begin{proof} 
Note first that $\{v_n\}_{n \in \mathbb{N}}$ is bounded in $\W$ due to the following inequalities: 
\begin{align*}
o(1)\|\nabla v_n\|_p \ge 
\left< \frac{\En^\prime(u_n)}{\|u_n\|_p^{p-1}}, v_n \right>
&\ge \|\nabla v_n\|_p^p-|\alpha_n|\|v_n\|_p^p 
-\frac{|\beta_n|}{\|u_n\|_p^{p-q}}\|v_n\|_q^q \\
&\ge \|\nabla v_n\|_p^p-|\alpha_n| 
-\frac{|\beta_n|}{\|u_n\|_p^{p-q}}|\Omega|^{1-q/p}, 
\end{align*}
where $o(1)\to 0$ as $n\to\infty$ and $|\Omega|$ denotes the Lebesgue measure of $\Omega$. (The last estimate is obtained by the H\"older inequality.)
Therefore, we may suppose that, up to a subsequence, $v_n \rightharpoonup v_0$ in $\W$ and $v_n \to v_0$ in $L^p(\Omega)$, where $v_0\in \W$ is such that $\|v_0\|_p = 1$.
Consequently, we get
\begin{align*}
o(1) &= \left< \En^\prime(u_n),\frac{v_n-v_0}{\|u_n\|_p^{p-1}} \right> \\
&=\intO |\nabla v_n|^{p-2}\nabla v_n\nabla (v_n-v_0)\,dx 
+\frac{1}{\|u_n\|_p^{p-q}}
\intO |\nabla v_n|^{q-2}\nabla v_n\nabla (v_n-v_0)\,dx
\\
&~\quad -\alpha_n\intO |v_n|^{p-2}v_n(v_n-v_0)\,dx 
-\frac{\beta_n}{\|u_n\|_p^{p-q}} \intO |v_n|^{q-2}v_n(v_n-v_0)\,dx 
\\
&=\intO |\nabla v_n|^{p-1}\nabla v_n\nabla (v_n-v_0)\,dx +o(1). 
\end{align*}
Thus, the $(S_+)$ property of $-\Delta_p$ on $\W$ yields that $v_n\to v_0$ strongly in $\W$ (cf.\ \cite[Definition 5.8.31 and Lemma 5.9.14]{drabekmilota}). 
Moreover, for any $\varphi\in \W$, by taking $\varphi/\|u_n\|_p^{p-1}$ as a test function, we have 
\begin{align*}
o(1) 
&=\intO |\nabla v_n|^{p-2}\nabla v_n\nabla \varphi\,dx 
+\frac{1}{\|u_n\|_p^{p-q}}
\intO |\nabla v_n|^{q-2}\nabla v_n\nabla \varphi\,dx
\\
&~\quad -\alpha_n \intO |v_n|^{p-2}v_n\varphi\,dx
-\frac{\beta_n}{\|u_n\|_p^{p-q}}\intO |v_n|^{q-2}v_n\varphi\,dx. 
\end{align*}
Letting $n\to\infty$ and recalling that $\|v_0\|_p=1$, we see that $v_0$ is a nontrivial solution of 
$$
-\Delta_p v_0 =\alpha |v_0|^{p-2}v_0 \quad {\rm in}\ \Omega,\quad 
v_0=0 \quad {\rm on}\ \partial\Omega. 
$$
Thus, $\alpha\in \sigma(-\Delta_p)$. 
If, additionally, $v_n\ge 0$ for all $n \in \mathbb{N}$, then $v_0 \ge 0$. Since any eigenfunction except the first one must be sign-changing (cf.\ \cite{anane1987}), we conclude that $\alpha=\lambda_1(p)$.
\end{proof} 


\begin{lemma}\label{lem:bdd-PS-2}
Let $\{\alpha_n\}_{n \in \mathbb{N}}$, 
$\{\beta_n\}_{n \in \mathbb{N}} \subset \mathbb{R}$, 
and $\{u_n\}_{n \in \mathbb{N}}\subset\W\setminus\{0\}$ be sequences satisfying
$$
\alpha_n\to\alpha,
\quad \beta_n\to\beta,
\quad 
\|\nabla u_n\|_p \to 0, 
\quad {\rm and}\quad 
\frac{\|\En'(u_n)\|_{(\W)^*}}{\|\nabla u_n\|_q^{q-1}} \to 0
$$
as $n\to \infty$, and $H_{\alpha_n}(u_n)<0$ for $n \in \mathbb{N}$.
Then the sequence $\{w_n\}_{n \in \mathbb{N}}$, where $w_n := u_n/\|\nabla u_n\|_q$ for $n \in \mathbb{N}$, has a subsequence convergent to some $w_0 \in ES(q;\beta)\setminus\{0\}$ weakly in $\W$ and strongly in $W^{1,q}_0$, that is, $\beta \in \sigma(-\Delta_q)$. 

In particular, if $u_n$ is nonnegative for $n \in \mathbb{N}$, then $w_0=\varphi_q/\|\nabla \varphi_q\|_q$ and $\beta=\lambda_1(q)$. 
\end{lemma}
\begin{proof} 
By the assumption $H_{\alpha_n}(u_n)<0$, we may assume that $\|\nabla u_n\|_p^p\le (\alpha+1)\|u_n\|_p^p$ for all $n \in \mathbb{N}$. 
Therefore, due to \cite[Lemma 9]{T-2014}, there exists a constant $C > 0$ such that $\|\nabla u_n\|_p\le C\|u_n\|_q$ for all $n \in \mathbb{N}$. 
At the same time, we know that $\lambda_1(q)\|u_n\|_q^q \le \|\nabla u_n\|_q^q$ for all $n \in \mathbb{N}$. The last two inequalities directly imply the boundedness of $\{w_n\}_{n \in \mathbb{N}}$ in $\W$. 
Then, choosing an appropriate subsequence, we may suppose that $w_n\to w_0$ weakly in $\W$ (and hence in $W_0^{1,q}$) and strongly in $L^p(\Omega)$, where $w_0\in\W$. 
Therefore, we deduce that 
\begin{align}
\notag
o(1)&=
\left< \En^\prime(u_n),\frac{w_n-w_0}{\|\nabla u_n\|_q^{q-1}} \right>
\\
\notag
&=\|\nabla u_n\|_q^{p-q}\intO |\nabla w_n|^{p-2}\nabla w_n\nabla (w_n-w_0)\,dx
+\intO |\nabla w_n|^{q-2}\nabla w_n\nabla (w_n-w_0)\,dx
\\
\notag
&\qquad -\alpha_n \|\nabla u_n\|_q^{p-q}\intO |w_n|^{p-2} w_n (w_n-w_0)\,dx
-\beta_n \intO |w_n|^{q-2} w_n (w_n-w_0)\,dx
\\
\label{eq:lemma3.4:1}
&=\intO |\nabla w_n|^{q-2}\nabla w_n\nabla (w_n-w_0)\,dx +o(1).
\end{align}
Using the $(S_+)$ property of $-\Delta_q$ on $W_0^{1,q}$, we conclude that $w_n\to w_0$ strongly in $W_0^{1,q}$. 
This implies that $\|\nabla w_0\|_q=1$ and hence $w_0 \not\equiv 0$. 
Moreover, considering $\langle \En^\prime(u_n),\varphi/\|\nabla u_n\|_q^{q-1} 
\rangle$ for any $\varphi\in C_0^\infty(\Omega)$ and using the density of $C_0^\infty(\Omega)$ in $W^{1,q}_0$, we proceed analogously to \eqref{eq:lemma3.4:1} to deduce that $w_0 \in ES(q;\beta) \setminus \{0\}$. 
The final assertion follows as in Lemma~\ref{lem:bdd-PS}.
\end{proof} 

\begin{lemma}\label{lem:bdd-minimizer-Nehar} 
Suppose that $\alpha>\lambda_1(p)$ and $\beta \le \lambda_1(q)$. 
Let $\{u_n\}_{n \in \mathbb{N}} \subset \N$ be a minimizing sequence of $d(\alpha,\beta)$ such that $u_n \ge 0$ for all $n \in \mathbb{N}$. 
If $\|\nabla u_n\|_p\to\infty$ as $n\to\infty$, then $\alpha\ge \alpha_*$, $\beta=\lambda_1(q)$, and the sequence $\{v_n\}_{n \in \mathbb{N}}$, where $v_n:=u_n/\|\nabla u_n\|_p$ for $n \in \mathbb{N}$, has a subsequence convergent to some $v_0\in\mathbb{R}\varphi_q\setminus\{0\}$ weakly in $\W$ and strongly in $L^p(\Omega)$. 
\end{lemma}
\begin{proof} 
	Suppose that the assumptions of the lemma are satisfied. 
	Since $\{v_n\}_{n \in \mathbb{N}}$ is bounded in $\W$, we may assume that, up to a subsequence, $v_n$ converges to some $v_0 \in \W$ weakly in $\W$ and strongly in $L^p(\Omega)$. 
	Since $\{u_n\}_{n \in \mathbb{N}} \subset \N$ is a minimizing sequence, for sufficiently large $n \in \mathbb{N}$ we have 
	$$
	\frac{p-q}{pq}G_\beta(v_n)=
	\frac{\E(u_n)}{\|\nabla u_n\|_p^q}\le \frac{d(\alpha,\beta)+1}{\|\nabla u_n\|_p^q}=o(1) 
	$$
	as $n\to\infty$, and hence 
	$$
	0 \le G_\beta(v_0) \le \liminf_{n\to\infty} G_\beta(v_n) \le 0, 
	$$
	where the first inequality follows from $\beta\le \lambda_1(q)$. 
	Thus, $G_\beta(v_0)=0$ occurs. On the other hand, $H_\alpha(v_n)=-G_\beta(v_n)\le 0$ implies that $\|v_n\|_p^p\ge 1/\alpha$ for all $n \in \mathbb{N}$. Therefore, $H_\alpha(v_0)\le 0$, $\|v_0\|_p^p \ge 1/\alpha$, and hence $v_0 \not\equiv 0$. 
	Consequently, we must have $\beta=\lambda_1(q)$ and $v_0\in\mathbb{R}\varphi_q\setminus\{0\}$, 
	and the definition of $\alpha_*$ (see \eqref{def:values}) yields $\alpha\ge \alpha_*$. 
\end{proof} 

\begin{lemma}\label{lem:conv-gs} 
	Let $\{\alpha_n\}_{n \in \mathbb{N}}$ and $\{\beta_n\}_{n \in \mathbb{N}}$ be such that $\lim\limits_{n\to\infty}\alpha_n=\alpha$ and $\lim\limits_{n\to\infty}\beta_n=\beta$ for some $\alpha, \beta \in \mathbb{R}$.
	Suppose that $u_n$ is a ground state of $\En$ with $H_{\alpha_n}(u_n)\neq0$ for $n \in \mathbb{N}$. If $\{u_n\}_{n \in \mathbb{N}}$ is bounded in $\W$, then $\{u_n\}_{n \in \mathbb{N}}$ has a subsequence strongly convergent in $\W$ to a solution of \eqref{eq:D}. Moreover, if $\liminf\limits_{n\to\infty}\En(u_n)<0$, then $\{u_n\}_{n \in \mathbb{N}}$ has a subsequence strongly convergent in $\W$ to a ground state of $\E$ and $d(\alpha,\beta)<0$. 
\end{lemma} 
\begin{proof} 
	Let $u_n$ be a ground state of $\En$ with $H_{\alpha_n}(u_n)\neq0$ for all $n \in \mathbb{N}$ and $\{u_n\}_{n \in \mathbb{N}}$ is bounded in $\W$. By passing to a subsequence, we may assume that $u_n$ converges to some $u_0 \in \W$ weakly in $\W$ and strongly in $L^p(\Omega)$. 
	Then, noting that each $u_n$ is a (nontrivial) solution of $(GEV;\alpha_n,\beta_n)$ (see Remark \ref{rem:positivity}), we get 
	\begin{align*}
	&\intO |\nabla u_n|^{p-2}\nabla u_n\nabla (u_n-u_0)\,dx
	+\intO |\nabla u_n|^{q-2}\nabla u_n\nabla (u_n-u_0)\,dx
	\\
	&=\alpha_n \intO |u_n|^{p-2} u_n (u_n-u_0)\,dx
	+\beta_n \intO |u_n|^{q-2} u_n (u_n-u_0)\,dx \to 0
	\end{align*}
	as $n\to\infty$. Using the $(S_+)$ property of $-\Delta_p-\Delta_q$ (cf. \cite[Remark~3.5]{BobkovTanaka2016}), we deduce that $u_n\to u_0$ strongly in $\W$. As a consequence, we easily see that $u_0$ is a solution of \eqref{eq:D}. Note that $u_0$ can be trivial.
	
	Assume, additionally, that $\E(u_0) = \lim\limits_{n\to\infty}\En(u_n)<0$. Let us prove that, in this case, $u_0$ is a ground state of $\E$. Note that the strong convergence in $\W$ implies that $u_0 \in \N$. Fix any $w \in \N$ such that $\E(w)<0$. Since $G_\beta(w)<0<H_\alpha(w)$, we see that $G_{\beta_n}(w)<0<H_{\alpha_n}(w)$ for sufficiently large $n \in \mathbb{N}$. 
	Therefore, Proposition \ref{prop:minpoint} implies that for any such $n \in \mathbb{N}$ we can find a unique $t_n>0$ such that 
$t_nw\in \mathcal{N}_{\alpha_n,\beta_n}$ and 
	$$
	\En (u_n)=d(\alpha_n,\beta_n)\le \En(t_nw)
	=\min_{s\ge 0}\En(sw) \le \En(w). 
	$$
	Letting $n\to\infty$, we get $\E(u_0)\le \E(w)$. At the same time, $\E(u_0)\le \E(v)$ is obviously satisfied for any $v\in\N$ such that $\E(v)\ge 0$. Consequently, $\E(u_0)\le \E(w)$ for all $w\in\N$, and hence $u_0$ is a ground state of $\E$. 
\end{proof}

\subsection{Fibered functional}\label{sec:fib} 
Take any $u \in \W$ such that $H_\alpha(u) \cdot G_\beta(u) < 0$. As it follows from Proposition~\ref{prop:minpoint}, there exists a unique $t(u) > 0$ such that $t(u) u \in \N$. 
Moreover, we easily see that 
\begin{equation}\label{tu}
t(u) = \left(\frac{-G_\beta(u)}{H_\alpha(u)}\right)^{\frac{1}{p-q}} = 
\frac{|G_\beta(u)|^{\frac{1}{p-q}}}{|H_\alpha(u)|^{\frac{1}{p-q}}}.
\end{equation}
Since $H_\alpha(u) \cdot G_\beta(u) < 0$, we have $\text{sign}(H_\alpha(u)) = -\text{sign}(G_\beta(u))$. Therefore, noting that $H_\alpha$ and $G_\beta$ are $p$- and $q$-homogeneous, respectively, we get 
$$
\J(u) := \E(t(u) u) = -\text{sign}(H_\alpha(u))\, \frac{p-q}{pq}\, \frac{|G_\beta(u)|^{\frac{p}{p-q}}}{|H_\alpha(u)|^{\frac{q}{p-q}}}. 
$$
The functional $\J$ is called \textit{fibered} functional \cite{pohozaev}. Evidently, $\J$ is $0$-homogeneous.

Let us introduce the following subsets of $\W$:
\begin{align*}
B_{\alpha, \beta}^- &:= \{u \in \W:\, H_\alpha(u)>0>G_\beta(u)\},\\
B_{\alpha, \beta}^+ &:= \{u \in \W:\, H_\alpha(u)<0< G_\beta(u)\}.
\end{align*}
Since $\J$ is $0$-homogeneous, we see that $\J(u) = \E(u)$ for any $u \in \N \cap (B_{\alpha, \beta}^-\cup B_{\alpha, \beta}^+ )$. 

The following proposition contains the results of Proposition~\ref{prop:GM3} \ref{prop:GM3:3} and Theorem~\ref{thm:negative-gs-2} \ref{thm:negative-gs-2:2}. 
We present it in this subsection for the better exposition. 
\begin{proposition}\label{Iinf}
	Let $1<q<p<\infty$ and $\alpha = \lambda_1(p)$, $\beta = \beta_*$. Then $\N\cap B_{\alpha,\beta}^-\neq\emptyset$ and 
	\begin{equation}\label{prop:Infif}
	m(\alpha,\beta) = d(\alpha,\beta) = \inf_{u \in \N \cap B_{\alpha,\beta}^-} \J(u) < 0.
	\end{equation}
	Moreover, $m(\alpha,\beta) > -\infty$ if and only if $p \geq 2q$. 
	Furthermore, if $p > 2q$, then $m(\alpha,\beta)$ is attained.
\end{proposition}
\begin{proof} 
	Let us show first that $\N\cap B_{\alpha,\beta}^-\neq\emptyset$. 
	In view of Lemma~\ref{lem:LID}, we see that $\varphi_p$ is a regular point of $G_\beta$, i.e., there exists $\theta \in C_0^{\infty}(\Omega)$ such that
	\begin{equation}\label{g<0}
	-D:=\langle G'_\beta(\varphi_p), \theta \rangle < 0.
	\end{equation}
	Note that $\theta\not\in\mathbb{R}\varphi_p$ since $\langle G_\beta^\prime(\varphi_p),\varphi_p\rangle = q G_\beta(\varphi_p) = 0$.	
	Therefore, the simplicity of $\alpha=\lambda_1(p)$ implies that $H_\alpha(\varphi_p+\varepsilon\theta) > 0$ for any $\varepsilon \neq 0$.
	Moreover, by \eqref{g<0}, there exists $\varepsilon_0>0$ such that 
	\begin{equation*}\label{g<0-2}
	\langle G'_\beta(\varphi_p+\varepsilon\theta), \theta \rangle \le -\frac{D}{2} < 0
	\qquad {\rm for\ all}\ 
	\varepsilon \in [-\varepsilon_0,\varepsilon_0]. 
	\end{equation*}
	Fix any $\varepsilon\in(0,\varepsilon_0]$ and denote  $u_\varepsilon := \varphi_p + \varepsilon \theta$. According to the mean value theorem, there exist $\varepsilon_1\in (0,\varepsilon)$ and $\varepsilon_2\in(0,\varepsilon)$ such that
	\begin{align} 
	\label{eq:prop3.7:1}
	0 < H_\alpha(u_\varepsilon) &= 
	H_\alpha(\varphi_p) + 
	\varepsilon\langle H_\alpha'(\varphi_p+\varepsilon_1\theta), \theta \rangle 
	=\varepsilon \langle H_\alpha'(\varphi_p+\varepsilon_1\theta), \theta \rangle, 
	\\
	\label{eq:prop3.7:2}
	G_\beta(u_\varepsilon) &= 
	G_\beta(\varphi_p) + 
	\varepsilon 
	\langle G_\beta'(\varphi_p+\varepsilon_2\theta), \theta \rangle
	=\varepsilon\langle G_\beta'(\varphi_p+\varepsilon_2\theta), \theta \rangle
	\le -\frac{\varepsilon D}{2} < 0. 
	\end{align}
	Hence, $u_\varepsilon\in B_{\alpha, \beta}^-$ and there exists $t(u_\varepsilon) > 0$ such that
	$t(u_\varepsilon)u_\varepsilon\in \N\cap B_{\alpha, \beta}^-$, see Proposition \ref{prop:minpoint}. 
	
	Let us now prove \eqref{prop:Infif}. It is easy to see that 
	\begin{equation*}
	m(\alpha,\beta) \leq d(\alpha,\beta) \leq \inf_{u \in \N \cap B_{\alpha,\beta}^-} \J(u) < 0.
	\end{equation*}
	(The last inequality follows by considering $t_q \varphi_q \in \N \cap B_{\alpha,\beta}^-$, where $t_q>0$ is obtained by Proposition \ref{prop:minpoint} since $H_\alpha(\varphi_q) > 0 > G_\beta(\varphi_q)$.) 
	On the other hand, if $\{u_n\}_{n \in \mathbb{N}} \subset \W$ is a minimizing sequence for $m(\alpha,\beta)$, then we easily see that $H_\alpha(u_n) > 0 > G_\beta(u_n)$ for all $n \in \mathbb{N}$, and hence Proposition \ref{prop:minpoint} implies the existence of a unique minimum point $t_n > 0$ of $\E(tu_n)$ on $[0,\infty)$ such that $t_nu_n \in \N \cap B_{\alpha,\beta}^-$ for all $n \in \mathbb{N}$. Therefore, we get $\inf\limits_{u \in \N \cap B_{\alpha,\beta}^-} \J(u) \leq m(\alpha,\beta)$, and hence \eqref{prop:Infif} follows. 
	
	Now, we study the behavior of $\J(u_\varepsilon)$, where $u_\varepsilon$ is defined as above.  Assume first that $p<2q$. 
	Let us recall that there exists a positive constant $C$ such that for all for $x,y\in\mathbb{R}^N$ the following inequalities are satisfied:
	\begin{equation*}
	0 \le \langle |x|^{p-2}x-|y|^{p-2}y, x-y\rangle 
	\le 
	\left\{
	\begin{aligned}
	&C|x-y|^p 					&&{\rm if}\ 1<p\le 2,\\ 
	&C|x-y|^2(|x|+|y|)^{p-2}	&&{\rm if}\ p\ge 2.
	\end{aligned}
	\right.
	\end{equation*}
	Therefore, recalling also that $\alpha = \lambda_1(p)$ and $0<\varepsilon_1<\varepsilon\le \varepsilon_0$, we obtain 
	\begin{align*} 
	\langle H_\alpha'(\varphi_p+\varepsilon_1\theta), \theta \rangle
	&= \langle H_\alpha'(\varphi_p+\varepsilon_1\theta), \theta \rangle 
	-\langle H_\alpha'(\varphi_p), \theta \rangle \\
	&=\frac{1}{\varepsilon_1}
	\langle H_\alpha'(\varphi_p+\varepsilon_1\theta)-H_\alpha^\prime(\varphi_p), 
	(\varphi_p+\varepsilon_1\theta) - \varphi_p \rangle 
	\\
	&\le 
	\left\{
	\begin{aligned}
	&\frac{C\varepsilon_1^p}{\varepsilon_1}\|\nabla \theta\|_p^p
	=C^\prime\varepsilon_1^{p-1} &&{\rm if}\ 1<p\le 2, 
	\\ 
	&\frac{C\varepsilon_1^2}{\varepsilon_1}
	\intO |\nabla\theta|^2
	(2 |\nabla \varphi_p|+\varepsilon_1|\nabla \theta|)^{p-2}  \, dx
	\le C^\prime\varepsilon_1 
	&&{\rm if}\ p\ge 2,
	\end{aligned}
	\right.
	\end{align*} 
	where $C^\prime > 0$ is independent of $\varepsilon$. 
	Consequently, we deduce from \eqref{eq:prop3.7:1} that 
	$$
	H_\alpha(u_\varepsilon)\le C^\prime \varepsilon^{p} \quad 
	{\rm if}\ 1<p\le 2, 
	\qquad H_\alpha(u_\varepsilon)\le C^\prime \varepsilon^{2} \quad {\rm if}\ p\ge 2. 
	$$
	Recalling now that $t(u_\varepsilon)u_\varepsilon\in \N\cap B_{\alpha, \beta}^-$ and using \eqref{eq:prop3.7:2}, we get
	\begin{align*} 
	\inf_{u \in \W} \E(u) 
	&= \inf_{u \in \N} \E(u) 
	= \inf_{u \in \N \cap B_{\alpha,\beta}^-} \J(u) 
	\leq \J(t(u_\varepsilon)u_\varepsilon) \\
	&=\J(u_\varepsilon)
	=- \frac{p-q}{pq}\,
	\frac{|G_\beta(u_\varepsilon)|^{\frac{p}{p-q}}}{|H_\alpha(u_\varepsilon)|^{\frac{q}{p-q}}} \le 
	\left\{
	\begin{aligned}
	&-C^{\prime\prime} \varepsilon^{\frac{p}{p-q} - \frac{pq}{p-q}}
	&&{\rm if}\ 1<p\le 2, 
	\\[3mm]
	&-C^{\prime\prime} \varepsilon^{\frac{p}{p-q} - \frac{2q}{p-q}}
	&&{\rm if}\ p\ge 2, 
	\end{aligned}
	\right.
	\end{align*}
	where $C^{\prime\prime}$ is a positive constant independent of $\varepsilon$. 
	Since $p<2q$, we obtain that $m(\alpha,\beta) = -\infty$ by tending $\varepsilon \to +0$.
	
	Assume now that $p \geq 2q$. (In particular, we always have $p>2$.) Suppose, by contradiction, that $m(\alpha,\beta) = -\infty$. Then there exists a sequence $\{u_n\}_{n \in \mathbb{N}} \subset \N\cap B_{\alpha, \beta}^-$ such that $\J(u_n) \to -\infty$ as $n \to \infty$. Since $\J$ is $0$-homogeneous, we can assume that $\|\nabla u_n\|_p = 1$ for all $n \in \mathbb{N}$. Therefore, we see that $u_n \to \varphi_p$ strongly in $\W$, where $\|\nabla \varphi_p\|_p=1$. 
	Indeed, by the boundedness of $\{u_n\}_{n \in \mathbb{N}}$, we may assume that $u_n$ converges to some $u_0$ weakly in $\W$ and strongly in $L^p(\Omega)$. 
	Then, $G_\beta(u_0)\le \liminf\limits_{n\to\infty} G_\beta(u_n)\le 0$ and $0\le H_\alpha(u_0)\le \liminf\limits_{n\to\infty} H_\alpha(u_n)$ (recall that $\alpha=\lambda_1(p)$). 
	On the other hand, the assumption $\lim\limits_{n\to\infty}\J(u_n)=-\infty$ implies that  $\liminf\limits_{n\to\infty} H_\alpha(u_n)=0$, whence $H_\alpha(u_0)=0$. 
	This means that $\|\nabla u_0\|_p=\lim\limits_{n\to\infty}\|\nabla u_n\|_p$ and $\|\nabla u_0\|_p^p=\lambda_1(p)\|u_0\|_p^p$, that is, $u_n$ converges to $\varphi_p$ strongly in $\W$ as $n \to \infty$.

	Let us make the $L^2$-orthogonal decomposition $u_n = \gamma_n \varphi_p + v_n$, where $\gamma_n \in \mathbb{R}$ and $v_n \in \W$ are chosen in such a way that $\gamma_n = \|\varphi_p\|_2^{-2} \intO u_n \varphi_p \, dx$ and $\intO v_n \varphi_p \, dx = 0$ for all $n \in \mathbb{N}$. 
	Since $u_n \to \varphi_p$ strongly in $\W$, we derive that $\gamma_n \to 1$ and $\|\nabla v_n\|_p \to 0$ as $n \to \infty$.
	Using now the improved Poincar\'e inequality of \cite{takac}, we get
	\begin{align}
	H_\alpha(u_n) &\geq C\left(|\gamma_n|^{p-2} \intO |\nabla \varphi_p|^{p-2} |\nabla v_n|^2 \, dx + \intO |\nabla v_n|^p \, dx \right) 
	\nonumber 
	\\ 
	\label{eq:estimate_for_H}
	&\geq \dfrac{C}{2}\left(\intO |\nabla \varphi_p|^{p-2} |\nabla v_n|^2 \, dx + \intO |\nabla v_n|^p \, dx \right) 
	\end{align} 
	for large $n \in \mathbb{N}$, 
	where $C>0$ does not depend on $n \in \mathbb{N}$. (Below in the proof we will always denote by $C$ a positive constant independent of $n \in \mathbb{N}$.) 
	Let us now estimate $|G_\beta(u_n)|$ from above. Using the mean value theorem, we can find $\varepsilon_n \in (0,1)$ for each $n \in \mathbb{N}$ such that
	\begin{align}
	\notag
	0 > G_\beta(u_n) 
	&= 
	|\gamma_n|^{q} G_\beta(\varphi_p) + \langle G_\beta'(\gamma_n \varphi_p+\varepsilon_n v_n), v_n \rangle \\
	\label{eq:estim:Inf:G}
	&\geq - \intO |\nabla(\gamma_n \varphi_p+\varepsilon_n v_n)|^{q-1}|\nabla v_n| \, dx - 
	\beta \intO |\gamma_n \varphi_p+\varepsilon_n v_n|^{q-1}|v_n| \, dx.
	\end{align}
	First, we estimate the second summand in \eqref{eq:estim:Inf:G} as follows. Since $p \geq 2q > 2(q-1)$ by assumption, we use the H\"older inequality and an embedding result of \cite[Lemma 4.2]{takac} or \cite[Lemma 4.2]{takac2} to obtain
	\begin{align*}
	\intO |\gamma_n \varphi_p + \varepsilon_n v_n|^{q-1}|v_n| \, dx 
	&\leq \left(\intO |\gamma_n \varphi_p + \varepsilon_n v_n|^{2(q-1)} \, dx\right)^\frac{1}{2} \left(\intO |v_n|^2 \, dx\right)^\frac{1}{2}
	\\
	&\leq 
	C \left(\intO |v_n|^2 \, dx\right)^\frac{1}{2} \leq 
	C \left(\intO |\nabla \varphi_p|^{p-2} |\nabla v_n|^2 \, dx\right)^\frac{1}{2}
	\quad \text{for all } n \in \mathbb{N}.
	\end{align*}
	Let us estimate the first summand in \eqref{eq:estim:Inf:G}. 
	Note first that
	$$
	\intO |\nabla(\gamma_n \varphi_p+\varepsilon_n v_n)|^{q-1}|\nabla v_n| \, dx \leq C \intO \left(|\nabla \varphi_p|+| \nabla v_n|\right)^{q-1}|\nabla v_n| \, dx.
	$$
	Now, using the H\"older inequality, we get
	\begin{align*}
	\intO \left(|\nabla \varphi_p|+| \nabla v_n|\right)^{q-1}& |\nabla v_n| \, dx
	=
	\intO \left(|\nabla \varphi_p|+| \nabla v_n|\right)^\frac{p-2}{2}|\nabla v_n| \cdot \left(|\nabla \varphi_p|+| \nabla v_n|\right)^\frac{2q-p}{2} \, dx
	\\
	&\leq 
	\left(\intO \left(|\nabla \varphi_p|+| \nabla v_n|\right)^{p-2}|\nabla v_n|^2 \, dx\right)^\frac{1}{2}
	\left(\intO \frac{dx}{\left(|\nabla \varphi_p|+| \nabla v_n|\right)^{p-2q}} \right)^\frac{1}{2} 
	\\
	&\leq
	C \left(\intO |\nabla \varphi_p|^{p-2} |\nabla v_n|^2 \, dx + \intO |\nabla v_n|^{p} \, dx \right)^\frac{1}{2}
	\left(\intO \frac{dx}{|\nabla \varphi_p|^{p-2q}}\right)^\frac{1}{2}.
	\end{align*}
	If $p=2q$, then we conclude that 
	\begin{equation}\label{eq:estimate_for_G}
	|G_\beta(u_n)| \leq C \left(\intO |\nabla \varphi_p|^{p-2} |\nabla v_n|^2 \, dx + \intO |\nabla v_n|^{p} \, dx\right)^\frac{1}{2} 
	\quad \text{for all } n \in \mathbb{N}.
	\end{equation}
	Assume now that $p>2q$. Taking $r=\frac{p-2q}{p-1}$, we see that $r<1$. Hence, we can apply the integrability result of \cite[Theorem 1.1]{damascelli} to derive that $\intO \frac{dx}{|\nabla \varphi_p|^{(p-1)r}} \leq C$ whenever $N \geq 2$. 
	If $N=1$, then we also have $\intO \frac{dx}{|\varphi_p^\prime|^{(p-1)r}} \leq C$. Indeed, in this case $\varphi_p$ is a generalized trigonometric function $\sin_p$ (cf.\ \cite{busheled}). Then, we deduce from \cite[(2.12) and (2.18)]{busheled} that $\cos_p x := \varphi_p^\prime \approx C |x-a|^\frac{1}{p-1}$, where $a$ is a (unique) zero of $\cos_p$ on $\Omega$, which implies the desired integrability. 
	Therefore, we conclude that \eqref{eq:estimate_for_G} is satisfied for all $N \geq 1$ and $p \geq 2q$.
	
	Finally, combining the obtained estimates \eqref{eq:estimate_for_H} and \eqref{eq:estimate_for_G} for $H_\alpha(u_n)$ and $G_\beta(u_n)$, we get
	\begin{align}
	\notag
	-\infty 
	&= \inf_{u \in \N \cap B_{\alpha,\beta}^-} \J(u) 
	= \liminf_{n \to \infty} \J(u_n) \\
	\label{eq:Inf:contradict}
	&\geq
	- C \, \limsup_{n \to \infty} \left(\intO |\nabla \varphi_p|^{p-2} |\nabla v_n|^2 \, dx + \intO |\nabla v_n|^{p} \, dx \right)^\frac{p-2q}{2(p-q)} > -\infty
	\end{align}
	as $n \to \infty$ since $p \geq 2q$. A contradiction. 
	
	Let us show that $m(\alpha,\beta)$ is attained when $p>2q$. Let $\{w_n\}_{n \in \mathbb{N}}$ be a corresponding minimizing sequence for $m(\alpha,\beta)$. 
	In view of \eqref{prop:Infif}, we can assume that each $w_n \in \N\cap B_{\alpha, \beta}^-$.
	Suppose now, by contradiction, that $\|\nabla w_n\|_p \to \infty$ as $n \to \infty$. Then, considering $u_n := w_n/\|\nabla w_n\|_p$ for $n \in \mathbb{N}$, we see from \eqref{tu} that $H_\alpha(u_n) \to 0$, which implies that $u_n \to \varphi_p$ strongly in $\W$. However, in this case \eqref{eq:Inf:contradict} is valid, and we see that $\liminf\limits_{n \to \infty} \J(u_n) = 0$, which is a contradiction to $m(\alpha,\beta) < 0$. 
	Therefore, minimizing sequence $\{w_n\}_{n \in \mathbb{N}}$ is bounded, and hence $\E$ possesses a global minimizer whenever $p>2q$.
\end{proof}

\begin{remark}\label{rem:special} 
	Whether a global minimum of $E_{\lambda_1(p),\beta_*}$ is attained in the case $p = 2q$ remains an open problem.
\end{remark}

\begin{lemma}\label{lem:seq-to-negative-infty} 
Let $\alpha \geq \lambda_1(p)$ and $\beta > \lambda_1(q)$. 
Assume that $u_0\in\W$ satisfies $H_\alpha(u_0)=0$ and $G_\beta(u_0)<0$. 
Then
$$
\N\cap B_{\alpha,\beta}^-\neq\emptyset
\quad \text{and} \quad 
\inf_{u \in \W} \E(u) = \inf_{u \in \N} \E(u) = \inf_{u \in \N\cap B_{\alpha,\beta}^-} \J(u) = -\infty.
$$ 
\end{lemma} 
\begin{proof} 
Assume first that $\alpha> \lambda_1(p)$ and $\beta > \lambda_1(q)$.
Let $u_0\in\W$ be such that $H_\alpha(u_0)=0$ and $G_\beta(u_0)<0$. Considering $|u_0|$ if necessary, we may assume that $u_0\ge 0$.
Therefore, $u_0$ is a regular point of $H_\alpha$, and hence we can find $\theta \in \W$ such that 
$\langle H^\prime_\alpha(u_0), \theta \rangle > 0$.
Note that $\theta\not\in\mathbb{R}u_0$ since $\langle H_\alpha^\prime(u_0),u_0\rangle = p H_\alpha(u_0)= 0$.

Let us consider  $u_\varepsilon := u_0 + \varepsilon \theta$ 
for $\varepsilon>0$. It is easy to see that 
$$
\int_0^\varepsilon \langle H_\alpha^\prime(u_0+t\theta),u_0\rangle\,dt 
=H_\alpha(u_\varepsilon) > 0 
\quad {\rm and} \quad 
G_\beta(u_\varepsilon)<0 
$$
for sufficiently small $\varepsilon>0$. 
Therefore, $t(u_\varepsilon)u_\varepsilon\in \N\cap B_{\alpha,\beta}^-$, where $t(u_\varepsilon) > 0$ is obtained by Proposition~\ref{prop:minpoint}, and we get 
\begin{align*} 
\inf_{u \in \W} \E(u) 
&\leq \inf_{u \in \N} \E(u) \leq \inf_{u \in \N\cap B_{\alpha,\beta}^-} \J(u) \\
&\leq \J(t(u_\varepsilon)u_\varepsilon) =\J(u_\varepsilon) 
=- \frac{p-q}{pq}\,
\frac{|G_\beta(u_\varepsilon)|^{\frac{p}{p-q}}}{|H_\alpha(u_\varepsilon)|^{\frac{q}{p-q}}} 
\to -\infty 
\end{align*}
as $\varepsilon\to+0$, since $|G_\beta(u_\varepsilon)|\to |G_\beta(u_0)|\neq0$ and $|H_\alpha(u_\varepsilon)|\to |H_\alpha(u_0)|=0$. 

Assume now that $\alpha = \lambda_1(p)$ and $\beta > \lambda_1(q)$. Since $H_\alpha(u_0)=0$ if and only if $u_0 \in \mathbb{R}\varphi_p$, we see that $G_\beta(u_0) < 0$ if and only if $\beta > \beta_*$, see \eqref{def:values}. Let $u_0 = \varphi_p$. Taking any $\theta \in \W \setminus \mathbb{R}\varphi_p$ and considering $u_\varepsilon := \varphi_p + \varepsilon \theta$ for $\varepsilon > 0$, we can apply the arguments from above to obtain the desired conclusion 
because $H_\alpha(u_\varepsilon)>0$ for any $\varepsilon\not=0$. 
\end{proof} 

\begin{lemma}\label{lem:seq-to-zero}
Let $\alpha=\alpha_*$ and $\beta=\lambda_1(q)$. Then 
\begin{equation}\label{eq:lemma3.10}
\N \cap B_{\alpha,\beta}^+\neq \emptyset
\quad \text{and} \quad 
\inf_{u \in \N} E(u) = \inf_{u \in \N \cap B_{\alpha,\beta}^+} \J(u) = 0.
\end{equation}
\end{lemma} 
\begin{proof} 
Due to Lemma~\ref{lem:LID}, $\varphi_q$ is a regular point of $H_\alpha$. Hence, we can find $\theta \in C_0^\infty(\Omega)$ satisfying 
$
\langle H^\prime_\alpha(\varphi_q), \theta \rangle < 0.
$
Note that $\theta\not\in\mathbb{R}\varphi_q$ since $\langle H^\prime_\alpha(\varphi_q),\varphi_q \rangle = p H_\alpha(\varphi_q) = 0$. 
Let us consider  $u_\varepsilon := \varphi_q + \varepsilon \theta$ for $\varepsilon>0$. 
Fix any sufficiently small $\varepsilon>0$ such that $\langle H^\prime_\alpha(\varphi_q+t\theta), \theta \rangle < 0$ for all $t \in (0,\varepsilon)$. 
According to the mean value theorem, there exist $\varepsilon_1\in (0,\varepsilon)$ and $\varepsilon_2 \in (0,\varepsilon)$ such that
\begin{align}
\label{eq:lemma:3.10:1}
H_\alpha(u_\varepsilon) &= 
	H_\alpha(\varphi_q) + 
	\varepsilon\langle H_\alpha'(\varphi_q+\varepsilon_1\theta), \theta \rangle 
=\varepsilon \langle H_\alpha'(\varphi_q+\varepsilon_1\theta), \theta \rangle 
<0, 
\\
\notag
0<G_\beta(u_\varepsilon) &= 
	G_\beta(\varphi_q) + 
	\varepsilon 
\langle G_\beta'(\varphi_q+\varepsilon_2\theta), \theta \rangle
=\varepsilon\langle G_\beta'(\varphi_q+\varepsilon_2\theta), \theta \rangle, 
\end{align}
where we used the assumption $\alpha=\alpha_*$ and the fact that $G_\beta(w)>0$ for all $w\not\in \mathbb{R}\varphi_q$ due to the simplicity of $\beta=\lambda_1(q)$. 
Hence, Proposition \ref{prop:minpoint} guarantees the existence of $t(u_\varepsilon) > 0$ such that $t(u_\varepsilon) u_\varepsilon\in \N\cap B_{\alpha, \beta}^+$, and we get
\begin{align*} 
0 &\le \inf_{u \in \N\cap B_{\alpha,\beta}^+} \J(u) \le \J(t(u_\varepsilon)u_\varepsilon) \\
&= \J(u_\varepsilon) 
=\frac{p-q}{pq}\,
\frac{|G_\beta(u_\varepsilon)|^{\frac{p}{p-q}}}{|H_\alpha(u_\varepsilon)|^{\frac{q}{p-q}}} =\varepsilon\, \frac{p-q}{pq}\, 
\frac{|\langle G_\beta'(\varphi_q+\varepsilon_2\theta), \theta \rangle|^{\frac{p}{p-q}}}
{|\langle H_\alpha'(\varphi_q+\varepsilon_1\theta), \theta \rangle|^{\frac{q}{p-q}}}
\to 0 
\end{align*}
as $\varepsilon \to +0$ since 
$$
|\langle H_\alpha'(\varphi_q+\varepsilon_1\theta), \theta \rangle| \to |\langle H_\alpha'(\varphi_q), \theta \rangle| \neq 0. 
$$
Furthermore, thanks to $\beta=\lambda_1(q)$, we know that $G_\beta(v)\ge 0$ for all $v \in \W$ and hence $\E(u) \geq 0$ for all $u\in\N$, see \eqref{eq:Nehari}. 
The latter fact implies the desired equalities in \eqref{eq:lemma3.10}.
\end{proof}

\begin{proposition}\label{prop:miniNehari+} 
	Let $\lambda_1(p)<\alpha<\alpha_*$ and $\beta<\beta_*(\alpha)$. 
	Then there exists $v\in \N \cap B_{\alpha,\beta}^+$ such that 
	$$
	\E(v) = \inf_{u \in \N\cap B_{\alpha,\beta}^+}\E(u) = \inf_{u \in \N\cap B_{\alpha,\beta}^+}\J(u)>0. 
	$$
	Moreover, $v$ is a positive solution of \eqref{eq:D}.
\end{proposition}
\begin{proof}
	Let $\lambda_1(p)<\alpha<\alpha_*$. 
	If $\beta < \lambda_1(q)$, then the assertion follows from Theorem \ref{thm:MP-Nehari}. If $\beta = \lambda_1(q)$, then the assertion follows from Theorem \ref{prop:MP-Nehari-resonant} \ref{prop:MP-Nehari-resonant:2}. Assume now that $\lambda_1(q) < \beta < \beta_*(\alpha)$. 
	It is not hard to see that $\varphi_p \in B_{\alpha,\beta}^+$ (since $\beta < \beta_*(\alpha) < \beta_*$ by Proposition \ref{prop:property-curve} \ref{prop:property-curve:6}), and hence $\N \cap B_{\alpha,\beta}^+ \neq \emptyset$, as it follows from Proposition \ref{prop:minpoint}. 
	Let $\{u_n\}_{n \in \mathbb{N}}$ be a minimizing sequence for $\E$ over $\N \cap B_{\alpha,\beta}^+$. 
	Let us show first that $\{u_n\}_{n \in \mathbb{N}}$ is bounded in $\W$. Suppose, by contradiction, that $\|\nabla u_n\|_p \to \infty$ as $n \to \infty$. Then, considering $w_n := u_n/\|\nabla u_n\|_p$ for $n \in \mathbb{N}$, we see that $w_n$ converges to some $w_0$ weakly in $\W$ and strongly in $L^p(\Omega)$ and $L^q(\Omega)$.
	Thus, since $H_\alpha(w_n)<0$, the weak lower semicontinuity implies that $H_\alpha(w_0) \leq 0$. Moreover, $H_\alpha(w_n)<0$ yields $1<\alpha\|w_n\|_p^p$, and hence $w_0 \not\equiv 0$.
	Furthermore, recalling that $\beta < \beta_*(\alpha)$, we conclude that $G_\beta(w_0) > 0$. Therefore,
	$$
	\E(u_n) = \frac{p-q}{pq} \|\nabla u_n\|_p^q\, G_\beta(w_n) \to \infty
	\quad \text{as } n \to \infty,
	$$
	which is impossible, since $\{u_n\}_{n \in \mathbb{N}}$ is a minimizing sequence. Thus, $\{u_n\}_{n \in \mathbb{N}}$ is bounded in $\W$. Suppose now that $\|\nabla u_n \|_p \to 0$ as $n \to \infty$. 
	Considering again $w_n := u_n/\|\nabla u_n\|_p$, we derive as above that $H_\alpha(w_0) \leq 0$ and $w_0 \not\equiv 0$. However, since $u_n \in \N$, we get
	$$
	\|\nabla u_n\|_p^{p-q} H_\alpha(w_n) = - G_\beta(w_n) \to 0 \quad \text{as } n \to \infty,
	$$
	and hence $G_\beta(w_0) \leq 0$, which contradicts the definition of $\beta_*(\alpha)$ since $\beta < \beta_*(\alpha)$. As a result, we derive that $\inf\limits_{n\in\mathbb{N}}\|\nabla u_n\|_p>0$. 

	The boundedness of $\{u_n\}_{n \in \mathbb{N}}$ implies the existence of $u_0$ such that $u_n$ converges to $u_0$ weakly in $\W$ and strongly in $L^p(\Omega)$ and $L^q(\Omega)$, up to a subsequence. Thanks to $\delta:=\inf\limits_{n\in\mathbb{N}}\|\nabla u_n\|_p>0$ and $H_\alpha(u_0) \leq 0$, we see that $\alpha \|u_0\|_p^p \geq \delta$, and hence $u_0\not\equiv 0$. Therefore, the definition of $\beta_*(\alpha)$ ensures that  $G_\beta(u_0) > 0$. 
	
	Now, let us show that $u_n$ converges to $u_0$ strongly in $\W$. If we suppose that $\|\nabla u_0\|_p < \liminf\limits_{n\to\infty}\|\nabla u_n\|_p$, then $H_\alpha(u_0) < 0 < G_\beta(u_0)$, and hence Proposition \ref{prop:minpoint} yields the existence of $t_0 > 0$ such that $t_0u_0 \in \N \cap B_{\alpha,\beta}^+$. This implies the following contradiction:
	\begin{align*}
	\inf_{u \in \N\cap B_{\alpha,\beta}^+}\E(u) \leq \E(t_0u_0) < \liminf_{n\to\infty} \E(t_0u_n) \leq \liminf_{n\to\infty} \E(u_n) = \inf_{u \in \N\cap B_{\alpha,\beta}^+}\E(u),
	\end{align*}
	where the third inequality follows from the fact that $t=1$ is the unique maximum point of $\E(tu_n)$ on $[0,\infty)$ for any $n \in \mathbb{N}$. Consequently, $\|\nabla u_0\|_p=\liminf\limits_{n\to\infty}\|\nabla u_n\|_p$, whence $u_n\to u_0$ strongly in $\W$. Therefore, noting that $u_0\in\N$ and $H_\alpha(u_0) = -G_\beta(u_0) < 0$, we see that $u_0 \in \N \cap B_{\alpha,\beta}^+$ and it is a nonnegative minimizer of $\E$ over $\N \cap B_{\alpha,\beta}^+$ with $\E(u_0) > 0$. Consequently, $u_0$ is a positive solution of \eqref{eq:D}, see Remark \ref{rem:positivity}.
\end{proof}


\section{Proofs for global minimizers}\label{sec:proofs:global_minimizers}

\begin{proof*}{Proposition~\ref{prop:GM1}} 
\ref{prop:GM1:1} Let $\alpha \le \lambda_1(p)$ and $\beta \le \lambda_1(q)$. 
Then $H_\alpha(u)\ge 0$ and $G_\beta(u)\ge 0$ for all $u\in\W$, see Lemma \ref{lem:eigenvalue}. This implies that $\E(u) \geq 0$ for all $u\in\W$. On the other hand, we have $\E(0) = 0$, that is, $0$ is a global minimizer of $\E$. If $u \not\equiv 0$ is such that $\E(u) = 0$, then we get $H_\alpha(u) = 0$ and $G_\beta(u) = 0$. This is possible if and only if $\alpha=\lambda_1(p)$ and $\beta=\lambda_1(q)$. Consequently, $u = t\varphi_p$ and $u=s\varphi_q$ for some $t$, $s\in\mathbb{R} \setminus \{0\}$. However, it contradicts Lemma~\ref{lem:LID}, and hence $0$ is the unique global minimizer of $\E$. 

\ref{prop:GM1:2} Let $\alpha < \lambda_1(p)$ and $\beta > \lambda_1(q)$. The assertion was proved in \cite[Proposition~2]{BobkovTanaka2015}. 

\ref{prop:GM1:3} Let $\alpha>\lambda_1(p)$ and $\beta \in \mathbb{R}$. 
Since $H_\alpha(\varphi_p)<0$ and $p>q$, we have $\E(t\varphi_p) = t^pH_\alpha(\varphi_p)/p + t^qG_\beta(\varphi_p)/q \to -\infty$ as $t\to\infty$, which implies the desired result. 
\end{proof*}

\begin{proof*}{Proposition~\ref{prop:GM2}} 
Let $(\alpha_n,\beta_n) \in \mathbb{R}^2$ be such that $\alpha_n<\lambda_1(p)$ and $\beta_n>\lambda_1(q)$ for all $n \in \mathbb{N}$, and it converges to some $(\alpha,\beta) \in \mathbb{R}^2$ as $n \to \infty$. 
Let $u_n$ be a global minimizer of $\En$ given by Proposition \ref{prop:GM1} \ref{prop:GM1:2}. 
Since $\En(u_n)<0$, we have $G_{\beta_n}(u_n)<0<H_{\alpha_n}(u_n)$. Consequently, each $u_n$ is a solution of $(GEV;\alpha_n, \beta_n)$ (see Remark \ref{rem:positivity}), and we have 
\begin{equation}\label{eq:proof:prop:GM1:1}
\|\nabla u_n\|_p^p+\|\nabla u_n\|_q^q=\alpha_n\|u_n\|_p^p+\beta_n\|u_n\|_q^q.
\end{equation}
This implies that $\|u_n\|_p$ is bounded if and only if $\|\nabla u_n\|_p$ is bounded. 
Finally, since $\En$ is even, we may suppose that $u_n \ge 0$.

\ref{prop:GM2:1} Let $\alpha=\lambda_1(p)$ and $\beta>\beta_*$. 
Since $\alpha_n\to\lambda_1(p)$ and $u_n$ is a global minimizer of $\En$, 
we have 
\begin{equation*}\label{eq:GM2-1}
\limsup_{n\to\infty}\En(u_n)\le \limsup_{n\to\infty}\En(t\varphi_p)
=\frac{t^q}{q}G_\beta(\varphi_p)
\end{equation*}
for any $t>0$. 
Since $G_{\beta}(\varphi_p)<0$ for $\beta > \beta_*$, we get  $\lim\limits_{n\to\infty}\En(u_n)=-\infty$ by tending $t \to \infty$. 
Therefore, we see from \eqref{eq:proof:prop:GM1:1} that $\|u_n\|_p$ has no bounded subsequences, that is, $\lim\limits_{n\to\infty}\|u_n\|_p=\infty$ occurs.  
Finally, recalling that $\En^\prime(u_n)=0$ in $(\W)^*$ and $\alpha_n \to \lambda_1(p)$, 
Lemma \ref{lem:bdd-PS} guarantees that $u_n/\|u_n\|_p$ converges to $\varphi_p/\|\varphi_p\|_p$ 
strongly in $\W$ because any subsequence of $u_n/\|u_n\|_p$ has a strongly convergent subsequence to the same limit function. 

\ref{prop:GM2:2} Let $\alpha=\lambda_1(p)$ and $\lambda_1(q)<\beta<\beta_*$. 
First, we prove that $\{u_n\}_{n \in \mathbb{N}}$ is bounded in $\W$. 
In view of \eqref{eq:proof:prop:GM1:1}, it is sufficient to show the boundedness of $\{u_n\}_{n \in \mathbb{N}}$ in $L^p(\Omega)$.  
Suppose, by contradiction, that $\|u_n\|_p \to \infty$ as $n\to\infty$, up to a subsequence. 
Considering $v_{n} := u_{n}/\|u_{n}\|_p$ for $n \in \mathbb{N}$, Lemma \ref{lem:bdd-PS} implies the existence of a subsequence $\{v_{n_k}\}_{k \in \mathbb{N}}$ which converges strongly in $\W$ to $\varphi_p/\|\varphi_p\|_p$. Thus, since $\beta<\beta_*$, we have 
$\lim\limits_{k\to\infty}G_{\beta_{n_k}}(v_{n_k})=G_\beta(\varphi_p/\|\varphi_p\|_p)>0$. On the other hand, recalling that $u_{n} \in \mathcal{N}_{\alpha_n,\beta_n}$ 
and $m(\alpha_n,\beta_n)<0$ for all $n \in \mathbb{N}$, we get
\begin{align*}
0>\En(u_n) = \frac{p-q}{pq} G_{\beta_n}(u_n)
=\frac{p-q}{pq} G_{\beta}(u_n)
-\frac{p-q}{pq} (\beta_n-\beta)\|u_n\|_q^q. 
\end{align*}
This implies that $(\beta_{n_k}-\beta)\|v_{n_k}\|_q^q > G_{\beta}(v_{n_k})$ for all $k \in \mathbb{N}$.
Hence, letting $k\to\infty$, we obtain a contradiction. 
As a result, $\{\|u_n\|_p\}_{n \in \mathbb{N}}$ is bounded, and we conclude that $\{u_n\}_{n \in \mathbb{N}}$ is bounded in $\W$. 

Now, we prove that $\limsup\limits_{n\to\infty} \En(u_n)<0$. 
Since $u_n$ is a global minimizer of $\En$, we have for any $t>0$ that
$$
\limsup_{n\to\infty} \En(u_n)\le \limsup_{n\to\infty} \En(t\varphi_q)
= \frac{t^p}{p}H_\alpha(\varphi_q) + \frac{t^q}{q} G_\beta(\varphi_q).
$$
Then, recalling that $\beta>\lambda_1(q)$ and $q<p$, we take $t>0$ small enough to get the desired fact.

Finally, according to Lemma \ref{lem:conv-gs}, $\{u_n\}_{n \in \mathbb{N}}$ has a subsequence $\{u_{n_k}\}_{k \in \mathbb{N}}$ which converges strongly in $\W$ to a ground state $u_0$ of $\E$ and $\E(u_0) = d(\alpha,\beta)<0$. 
Moreover, $u_0$ is a global minimizer of $\E$. 
Indeed, taking any $w\in\W$ and passing to the limit in $E_{\alpha_{n_k},\beta_{n_k}}(u_{n_k})\le E_{\alpha_{n_k},\beta_{n_k}}(w)$, we conclude that $\E(u_0)\le \E(w)$, whence $u_0$ is a global minimizer of $\E$ and $m(\alpha,\beta)<0$.

\ref{prop:GM2:3} Let $\beta = \lambda_1(q)$. 
We begin by proving the boundedness of $\{u_n\}_{n \in \mathbb{N}}$ in $\W$. 
In view of \eqref{eq:proof:prop:GM1:1}, we suppose, by contradiction, that $\|u_n\|_p\to \infty$ as $n\to\infty$, up to a subsequence.
Since $u_n \geq 0$, it follows from Lemma~\ref{lem:bdd-PS} that $\{v_n\}_{n \in \mathbb{N}}$, where $v_n := u_n/\|u_n\|_p$ for $n \in \mathbb{N}$, has a subsequence $\{v_{n_k}\}_{k \in \mathbb{N}}$ which converges strongly  in $\W$ to $v_0 = \varphi_p / \|\varphi_p\|_p$, and $\alpha=\lambda_1(p)$. 
On the other hand, recalling that $G_{\beta_{n_k}}(u_{n_k}) < 0$, we get $G_{\beta_{n_k}}(v_{n_k})<0$. Since $\beta_n \to \lambda_1(q)$, we conclude that $G_\beta(v_0) = 0$ and hence  $v_0=\varphi_q/\|\varphi_q\|_p$. However, this contradicts Lemma~\ref{lem:LID}.
Therefore, $\{u_n\}_{n \in \mathbb{N}}$ is bounded in $\W$. 
This ensures that $\lim\limits_{n\to \infty}\En(u_n)=0$, since 
\begin{align*}
0>\En(u_n)&=\E(u_n)+\frac{\alpha-\alpha_n}{p}\|u_n\|_p^p
+\frac{\beta-\beta_n}{q}\|u_n\|_q^q
\\ 
&\ge \frac{\alpha-\alpha_n}{p}\|u_n\|_p^p
+\frac{\beta-\beta_n}{q}\|u_n\|_q^q=o(1),
\end{align*}
where we used the fact that $\E(u_n) \geq m(\alpha, \beta) = 0$, see Proposition~\ref{prop:GM1} \ref{prop:GM1:1}.

Since $\{u_n\}_{n \in \mathbb{N}}$ is bounded in $\W$ and $\En(u_n)<0$, Lemma~\ref{lem:conv-gs} implies that any subsequence of $\{u_n\}_{n \in \mathbb{N}}$ has a subsequence strongly convergent in $\W$ to a solution of \eqref{eq:D}. 
In view of Lemma~\ref{lem:nonempty-Nehari}, \eqref{eq:D} has no \textit{nontrivial} solutions for $\alpha \le \lambda_1(p)$ and $\beta=\lambda_1(q)$, and hence we conclude that $u_n$ converges to $0$ strongly in $\W$. 

Finally, consider $w_n := u_n/\|\nabla u_n\|_q$ for $n \in \mathbb{N}$. By choosing an appropriate subsequence of any subsequence of $\{w_n\}_{n \in \mathbb{N}}$, we may assume that $w_n$ converges to some $w_0$ weakly in $W_0^{1,q}$ and strongly in $L^q(\Omega)$. Since $G_{\beta_{n}}(w_{n})<0$, we get $1 \le \beta\|w_0\|_q^q$, whence $w_0 \not\equiv 0$. 
Moreover, by $\beta=\lambda_1(q)$, it is clear that $0\le G_\beta(w_0)\le \liminf\limits_{n\to\infty} G_{\beta_n}(w_n)\le 0$, that is,  $0 = G_\beta(w_0) = \lim\limits_{n\to\infty}G_{\beta_n}(w_n)$. 
This yields the strong convergence of $\{w_n\}_{n \in \mathbb{N}}$ in $W_0^{1,q}$ to  
$w_0=\varphi_q/\|\nabla \varphi_q\|_q$. 

\ref{prop:GM2:4} Let $\alpha=\lambda_1(p)$, $\beta=\beta_*$ and $p>2q$. 
First, we show that $\limsup\limits_{n\to\infty} \En(u_n)<0$. 
Taking any $v \in\N\cap B^-_{\alpha,\beta}$  (see Proposition~\ref{Iinf} for the existence), we see that 
$$
\En(u_n)=\inf_{u\in\W} \En(u)\le \En(v) = \E(v)+o(1) < 0
$$
for all large $n \in \mathbb{N}$, which implies the desired result. 

Now, let us show the boundedness of $\{u_n\}_{n \in \mathbb{N}}$ in $\W$. 
Suppose, by contradiction, that $\|\nabla u_n\|_p\to \infty$ as $n\to\infty$. 
Setting $w_n:=u_n/\|u_n\|_p$ for $n\in\mathbb{N}$, Lemma \ref{lem:bdd-PS} ensures that $w_n$ converges to $\varphi_p/\|\varphi_p\|_p$ strongly in $\W$, up to a subsequence. 
Therefore, considering the $L^2$-orthogonal decomposition $w_n=\gamma_n \varphi_p+v_n$ as in the proof of Proposition \ref{Iinf}, we see that $\gamma_n\to 1$ and $\|\nabla v_n\|_p\to0$ as $n\to\infty$. 
Recalling that $\alpha_n < \alpha=\lambda_1(p)$, we have $H_{\alpha_n}(w_n)>H_\alpha(w_n)>0$. 
Therefore, since $\{\beta_n\}_{n \in \mathbb{N}}$ is bounded, the same argument as in Proposition~\ref{Iinf} implies that 
$$
\En(u_n)=J_{\alpha_n,\beta_n}(w_n)\to 0 
$$
as $n\to \infty$, which contradicts $\limsup\limits_{n\to\infty} \En(u_n)<0$. 
Thus, $\{u_n\}_{n \in \mathbb{N}}$ is bounded in $\W$. 

Finally, Lemma \ref{lem:conv-gs} implies that $u_n$ converges strongly in $\W$, up to a subsequence, to a global minimizer of $\E$ as $n \to \infty$ (see the end of the proof of \ref{prop:GM2:2}).

\ref{prop:GM2:5} Let $\alpha=\lambda_1(p)$, $\beta=\beta_*$ and $p<2q$. 
Let us show $\lim\limits_{n\to\infty}\En(u_n)=-\infty$. 
Fix any $R>0$. According to Proposition \ref{Iinf}, we can choose $w \in \N \cap B^-_{\alpha,\beta}$ satisfying $\E(w)\le -R$. Then we get 
$$
\En(u_n)=\inf_{u\in\W} \En(u)\le \En(w)=\E(w)+o(1)
\le -R+o(1)
$$
for all large $n \in \mathbb{N}$, and hence $\limsup\limits_{n\to\infty}\En(u_n)\le -R$. 
Since $R>0$ is arbitrary, we get the desired result. 
The remaining claims can be proved as in \ref{prop:GM2:1}.
\end{proof*}

\begin{proof*}{Proposition~\ref{prop:GM3}} 
Let $\alpha=\lambda_1(p)$ and $\beta > \lambda_1(q)$. Then $H_\alpha(\varphi_q) > 0 > G_\beta(\varphi_q)$, and, considering $t\varphi_q$ for $t>0$ small enough, we see that $m(\alpha, \beta) \leq \E(t\varphi_q) < 0$.
	
\ref{prop:GM3:1} Let $\beta>\beta_*$.
Then $H_\alpha(\varphi_p)=0$ and $G_\beta(\varphi_p)<0$, and we get $\E(t\varphi_p) = t^qG_\beta(\varphi_p)/q\to -\infty$ as $t\to\infty$. 

\ref{prop:GM3:2} Let $\lambda_1(q)<\beta<\beta_*$. 
Set $\alpha_n=\alpha-1/n$, $n \in \mathbb{N}$. Since $\alpha_n<\lambda_1(p)$ and $\beta>\lambda_1(q)$, 
we can obtain a minimizer $u_n$ of $E_{\alpha_n,\beta}$ which satisfies $E_{\alpha_n,\beta}(u_n)<0$ for each $n \in \mathbb{N}$, see Proposition \ref{prop:GM1} \ref{prop:GM1:2}. 
According to Proposition \ref{prop:GM2} \ref{prop:GM2:2}, $u_n$ has a strongly convergent subsequence to a global minimizer $u_0$ with $\E(u_0)<0$. 

\ref{prop:GM3:3} Let $\beta=\beta_*$. 
The assertion follows from Proposition \ref{Iinf}. 
\end{proof*}

\begin{proof*}{Proposition~\ref{prop:conti-minimum-value}} 
First we prove that the extended function $m$ defined by \eqref{def:mini} is continuous at every $(\alpha,\beta)\in \mathbb{R}^2 \setminus \{\lambda_1(p)\}\times(-\infty,\beta_*]$. 
Let $\{(\alpha_n,\beta_n)\}_{n \in \mathbb{N}}$ be any sequence convergent to such  $(\alpha,\beta)$. 
We divide arguments for the following cases:

(a) Let $\alpha<\lambda_1(p)$ and $\beta<\lambda_1(q)$.
The assertion follows from Proposition \ref{prop:GM1} \ref{prop:GM1:1}.

(b) Let $\alpha<\lambda_1(p)$ and $\beta=\lambda_1(q)$. 
Then, $m(\alpha,\beta)=0$ holds by Proposition \ref{prop:GM1} \ref{prop:GM1:1}.
If there exists a subsequence of $\{(\alpha_n,\beta_n)\}_{n \in \mathbb{N}}$, denoted for simplicity by the same index $n$, such that $\beta_n>\lambda_1(q)$ for all $n \in \mathbb{N}$, then we can find a global minimizer $u_n$ of $\En$ for all $n \in \mathbb{N}$ large enough, see Proposition \ref{prop:GM1} \ref{prop:GM1:2}.
Namely, $m(\alpha_n,\beta_n)=\En(u_n)$, and Proposition \ref{prop:GM2} \ref{prop:GM2:3} shows that 
$m(\alpha_n,\beta_n) \to 0 = m(\alpha,\beta)$ as $n\to\infty$. 
On the other hand, if $\beta_n\le \lambda_1(q)$, then $m(\alpha_n,\beta_n)=0=m(\alpha,\beta)$ by Proposition \ref{prop:GM1} \ref{prop:GM1:1}, which completes the proof. 

(c) Let $\alpha<\lambda_1(p)$ and $\beta>\lambda_1(q)$. 
We may assume that $\alpha_n<\lambda_1(p)$ and $\beta_n>\lambda_1(q)$ for all sufficiently large $n \in \mathbb{N}$. 
By Proposition \ref{prop:GM1} \ref{prop:GM1:2}, we can choose a global minimizer $u_n$ of $\En$ and $m(\alpha_n,\beta_n)=\En(u_n)<0$. 
Recalling that $\alpha<\lambda_1(p)$, we deduce from Lemma~\ref{lem:bdd-PS} that $\{u_n\}_{n \in \mathbb{N}}$ is bounded in $\W$. Moreover, recalling that $q<p$, we get
\begin{align*}
m(\alpha_n,\beta_n)& \le \En(t\varphi_q) 
=\frac{t^p}{p}H_{\alpha_n}(\varphi_q) + \frac{t^q(\lambda_1(q)-\beta_n)}{q}\|\varphi_q\|_q^q 
\\
&=t^p\frac{\alpha_*-\alpha+o(1)}{p}\|\varphi_q\|_p^p-t^q\frac{\beta-\lambda_1(q)+o(1)}{q}\|\varphi_q\|_q^q <0
\end{align*}
for small $t>0$ and all $n \in \mathbb{N}$ large enough, which implies that $\limsup\limits_{n\to\infty}m(\alpha_n,\beta_n)<0$.
Therefore, Lemma \ref{lem:conv-gs} guarantees that $\{u_n\}_{n \in \mathbb{N}}$ has a subsequence strongly convergent in $\W$ to a global minimizer of $\E$. 
Thus, any subsequence of $\{m(\alpha_n,\beta_n)\}_{n \in \mathbb{N}}$ has a convergent subsequence to the same value $m(\alpha,\beta)$, i.e., $m(\alpha_n,\beta_n)\to m(\alpha,\beta)$ as $n\to \infty$. 

(d) Let $\alpha=\lambda_1(p)$ and $\beta>\beta_*$. 
Then $m(\alpha,\beta) = -\infty$ by Proposition~\ref{prop:GM3} \ref{prop:GM3:1}.
Taking a global minimizer $u_n$ of $\En$ provided $\alpha_n<\lambda_1(p)$ (see Proposition \ref{prop:GM1} \ref{prop:GM1:2}), we see that $m(\alpha_n,\beta_n)=\En(u_n) \to -\infty=m(\alpha,\beta)$ as $n\to\infty$ by Proposition \ref{prop:GM2} \ref{prop:GM2:1}. 
In the case of $\alpha_n\ge \lambda_1(p)$, the assertion obviously follows from Proposition \ref{prop:GM1} \ref{prop:GM1:3} or \ref{prop:GM3} \ref{prop:GM3:1} since $m(\alpha_n,\beta_n)=-\infty=m(\alpha,\beta)$.

(e) Let $\alpha>\lambda_1(p)$ and $\beta \in \mathbb{R}$. The assertion follows from Proposition~\ref{prop:GM1} \ref{prop:GM1:3}.

Let us now prove that $m$ is discontinuous on $(\alpha,\beta) \in  \{\lambda_1(p)\}\times(-\infty,\beta_*)$. On the one hand, $m(\alpha, \beta) = -\infty$ for any  $\alpha > \lambda_1(p)$ and $\beta \in \mathbb{R}$, see Proposition \ref{prop:GM1} \ref{prop:GM1:3}. On the other hand, if $\alpha = \lambda_1(p)$, then $m(\alpha, \beta) = 0$ for $\beta \leq \lambda_1(q)$, see Proposition \ref{prop:GM1} \ref{prop:GM1:1}, and $m(\alpha, \beta) > -\infty$ for $\lambda_1(q) < \beta < \beta_*$, see Proposition \ref{prop:GM3} \ref{prop:GM3:2}. These observations complete the proof.
\end{proof*}


\section{Proofs for ground states}\label{sec:proofs:ground_states}

\begin{proof*}{Proposition~\ref{prop:property-gs}} 
\ref{prop:property-gs:1} 
Let $u$ be a ground state of $\E$ with $\E(u)<0$. 
Suppose, by contradiction, that there exists a sequence $\{u_n\}_{n \in \mathbb{N}} \subset \W$ such that 
\begin{equation}\label{eq:prop:gs:1}
\E(u_n)<\E(u)\quad {\rm for\ all}\ n \in \mathbb{N} \quad {\rm and}\quad 
u_n \to u 
\quad {\rm strongly\ in}\ \W.
\end{equation}
Since $G_\beta(u)<0<H_\alpha(u)$, we may assume that $G_\beta(u_n)<0<H_\alpha(u_n)$ for all sufficiently large $n \in \mathbb{N}$. 
Thus, according to Proposition \ref{prop:minpoint}, there exists $s_n>0$ such that $s_nu_n\in\N$ and $\E(tu_n)$ attains the minimum value at $t=s_n$ on  $[0,\infty)$. 
Therefore, 
$$
\E(u)=\inf_{v\in\N}\E(v) \le \E(s_nu_n)=\min_{t\ge 0}\E(tu_n)\le \E(u_n),
$$
which contradicts \eqref{eq:prop:gs:1}.

\ref{prop:property-gs:2}
Let $u$ be a ground state of $\E$ with $\E(u)>0$. 
Proposition \ref{prop:minpoint} implies that $t=1$ is a unique maximum point of $\E(tu)$ on $[0, \infty)$, and hence $u$ is not a local minimum point of $\E$.
Let us now prove that $u$ is also not a local maximum point. 
Suppose, by contradiction, that there exists $\delta_0>0$ such that 
\begin{equation}\label{eq:unext-2} 
\E(v)\le \E(u) 
\quad {\rm for\ all}\ 
v 
{\rm \ with}\ 
\|\nabla v - \nabla u\|_p < \delta_0. 
\end{equation}
Let us take an arbitrary $\theta\in \W\setminus C^1_0(\overline{\Omega})$. Thus, $\theta\not\in\mathbb{R}u$ 
since $u\in  C^1_0(\overline{\Omega})$ (see Remark~\ref{rem:positivity}). Consider $u_\varepsilon := u + \varepsilon\theta$ for $\varepsilon\in\mathbb{R}$.
Recalling that $u \in \N$ and $G_\beta(u)>0>H_\alpha(u)$, there exists $\varepsilon_0 > 0$ such that 
$G_\beta(u_\varepsilon)>0>H_\alpha(u_\varepsilon)$ for any $\varepsilon\in (-\varepsilon_0,\varepsilon_0)$. 
Hence, in view of Proposition \ref{prop:minpoint}, for each $\varepsilon\in (-\varepsilon_0,\varepsilon_0)$ there exists a unique $t_\varepsilon > 0$ such that $t_\varepsilon u_\varepsilon\in\N$. 
Noting that $t_\varepsilon\to 1$ (see \eqref{tu}) and $u_\varepsilon\to u$ strongly in $\W$ as $\varepsilon\to 0$, we can choose $\varepsilon_1 \in (0, \varepsilon_0)$ such that $\|\nabla (t_\varepsilon u_\varepsilon) - \nabla u\|_p < \delta_0$ for any $\varepsilon\in (-\varepsilon_1,\varepsilon_1)$.
As a result, we deduce from \eqref{eq:unext-2} that
$$
0 < d(\alpha,\beta)\le \E(t_\varepsilon u_\varepsilon)\le \E(u)=d(\alpha,\beta), \quad 
{\rm and\ so}\quad d(\alpha,\beta)=\E(t_\varepsilon u_\varepsilon) 
$$
for all $\varepsilon\in (-\varepsilon_1,\varepsilon_1)$. 
Consequently, $t_\varepsilon u_\varepsilon$ must be a nontrivial solution of \eqref{eq:D} for all $\varepsilon\in (-\varepsilon_1,\varepsilon_1)$, and hence $t_\varepsilon u_\varepsilon \in \C$, see Remark~\ref{rem:positivity}.
Recalling that $u \in \C$, we get $\theta = \frac{1}{\varepsilon}\left(u_\varepsilon - u\right) \in \C$ for $\varepsilon \neq 0$,  which is impossible since $\theta \in \W\setminus\C$ by assumption. 
\end{proof*}

\begin{proof*}{Theorem~\ref{thm:MP-Nehari}} 
Let $\alpha>\lambda_1(p)$ and $\beta<\lambda_1(q)$. 
In \cite[Theorem 2.1]{BobkovTanaka2015}, it was proved that $c^+(\alpha,\beta)>0$ and it is attained by a positive solution $u$ of \eqref{eq:D}. Hence, $u\in \N$ and 
$$
d(\alpha,\beta) \le \E(u) = \E^+(u) = c^+(\alpha,\beta).
$$

On the other hand, $c^+(\alpha,\beta)\le c(\alpha,\beta)$. 
Indeed, fix any $\varepsilon > 0$ and take a path $\gamma_\varepsilon\in\Gamma(\alpha,\beta)$ such that $\max\limits_{s\in[0,1]}\E(\gamma_\varepsilon(s)) \le c(\alpha,\beta)+\varepsilon$. Noting that $\E(\gamma_\varepsilon(\cdot))=\E^+(|\gamma_\varepsilon(\cdot)|)$ and  $|\gamma_\varepsilon|\in\Gamma^+(\alpha,\beta)$, we obtain $$c^+(\alpha,\beta)\le \max\limits_{s\in[0,1]}\E^+(|\gamma_\varepsilon(s)|)= \max\limits_{s\in[0,1]}\E(\gamma_\varepsilon(s))\le c(\alpha,\beta)+\varepsilon. 
$$
Since $\varepsilon>0$ was chosen arbitrarily, we conclude that $c^+(\alpha,\beta)\le c(\alpha,\beta)$.

Finally, we show that $c(\alpha,\beta)\le d(\alpha,\beta)$. 
Fix any $\varepsilon>0$ and choose $w_\varepsilon\in \N$ such that $\E(w_\varepsilon) \le d(\alpha,\beta)+\varepsilon$. Since $\beta<\lambda_1(q)$, we see that $H_\alpha(w_\varepsilon)<0<G_\beta(w_\varepsilon)$. Therefore, $t=1$ is the maximum point of $\E(tw_\varepsilon)$ on $[0,\infty)$. Moreover, recalling that $q<p$, we can find sufficiently large $R>0$ such that $\E(Rw_\varepsilon)<0$. Hence, considering $\gamma(s):=sRw_\varepsilon$, we obtain that $\gamma \in \Gamma(\alpha,\beta)$ and 
$$
c(\alpha,\beta) \le 
\max_{s\in[0,1]} \E(\gamma(s))=\max_{s\ge 0} \E(sw_\varepsilon) 
=\E(w_\varepsilon)\le d(\alpha,\beta)+\varepsilon,
$$
which implies that $c(\alpha,\beta) \le d(\alpha,\beta)$. This leads to the desired conclusion.
\end{proof*}

\begin{proof*}{Theorem~\ref{prop:MP-Nehari-resonant}} 
Let $\alpha \in \mathbb{R}$ and $\beta = \lambda_1(q)$. 
Let $\{u_n\}_{n \in \mathbb{N}} \subset \N$ be a minimizing sequence for $d(\alpha,\beta)$. Since $\E$ is even, we may assume that $u_n\ge 0$. 
Note first that $d(\alpha,\beta) \ge 0$. Indeed, since $\beta=\lambda_1(q)$, we see that $G_\beta(u)\ge 0$ for any $u \in \W$ and hence $\E(u) \ge 0$ for any $u \in \N$ by \eqref{eq:Nehari}. 

\ref{prop:MP-Nehari-resonant:1} Let $\alpha \leq \lambda_1(p)$. The assertion follows from the emptiness of $\N$, see Lemma \ref{lem:nonempty-Nehari}.

\ref{prop:MP-Nehari-resonant:2} Let $\lambda_1(p) < \alpha < \alpha_*$. 
Suppose first that there exists a subsequence $\{u_{n_k}\}_{k \in \mathbb{N}}$ such that $\|\nabla u_{n_k}\|_p \to \infty$ as $k\to\infty$. Then Lemma~\ref{lem:bdd-minimizer-Nehar} yields $\alpha \ge \alpha_*$, a contradiction. Thus, $\{u_n\}_{n \in \mathbb{N}}$ is bounded in $\W$. 

Note now that $H_\alpha(u_n) < 0 < G_\beta(u_n)$ for all $n \in \mathbb{N}$, as it easily follows from Lemma \ref{lem:eigenvalue} \ref{lem:eigenvalue:2}. Setting $v_n:=u_{n}/\|\nabla u_{n}\|_p$, we may assume that, up to a subsequence, $v_n$ converges to some $v_0$ weakly in $\W$ and strongly in $L^p(\Omega)$. 
Moreover, since $H_\alpha(v_n)<0$, we obtain that $1 \le \alpha\|v_0\|_p^p$, and hence $v_0 \not\equiv 0$. Recall that $0 \le G_\beta(v_0)\le \liminf\limits_{n\to\infty}G_\beta(v_n)$ and $H_\alpha(v_0) \le \liminf\limits_{n\to\infty} H_\alpha(v_n) \le 0$. If $G_\beta(v_0)=0$, then $v_0=t \varphi_q$ for some $t>0$ and we get a contradiction to $\alpha < \alpha_*$. Thus, $G_\beta(v_0)>0$. 
This fact implies that $\inf\limits_{n \in \mathbb{N}}\|\nabla u_n\|_p>0$. 
Indeed, suppose that there exists a subsequence $\{u_{n_k}\}_{k \in \mathbb{N}}$ such that $\|\nabla u_{n_k}\|_p\to 0$ as $k\to\infty$. 
Since $G_\beta(v_0)>0$, we obtain the following contradiction:
$$
H_\alpha(v_0)\le \limsup_{k\to\infty}H_\alpha(v_{n_k})
=-\liminf_{k\to\infty}\frac{G_\beta(v_{n_k})}{\|\nabla u_{n_k}\|_p^{p-q}} 
\to -\infty
\quad \text{as } k\to\infty.
$$

Since $\{u_n\}_{n \in \mathbb{N}}$ is bounded in $\W$, we may assume that, up to a subsequence, $u_n$ converges to some $u_0$ weakly in $\W$ and strongly in $L^p(\Omega)$. Moreover, since $H_\alpha(u_n) < 0$ leads to $\alpha\|u_0\|_p\ge \inf\limits_{n \in \mathbb{N}}\|\nabla u_n\|_p>0$, we have $u_0 \not\equiv 0$. 
Let us show now that $u_n$ converges to $u_0$ strongly in $\W$. Suppose, contrary to our claim, that $\|\nabla u_0\|_p < \liminf\limits_{n\to\infty} \|\nabla u_n\|_p$. Then $H_\alpha(u_0) < 0$. Moreover, $G_\beta(u_0) > 0$ 
since otherwise $u_0\in\mathbb{R}\varphi_q\setminus\{0\}$ and so 
we get a contradiction to $\alpha < \alpha_*$. Therefore, Proposition \ref{prop:minpoint} yields the existence of a unique maximum point $t_0 > 0$ of $\E(tu_0)$ on $[0, \infty)$ such that $t_0 u_0 \in \N$, and hence 
\begin{align*}
d(\alpha, \beta) \leq \E(t_0 u_0) < \liminf\limits_{n\to\infty} \E(t_0 u_n) 
\leq \liminf\limits_{n\to\infty} \E(u_n) = d(\alpha, \beta),
\end{align*}
a contradiction. The last inequality was obtained by the fact that a unique maximum point of each $\E(tu_n)$ on $[0, \infty)$ is $t=1$.
Thus, $u_n \to u_0$ strongly in $\W$. This implies that $u_0 \in \N$ and $\E(u_0) = d(\alpha, \beta)$. Moreover, as above, we see that $H_\alpha(u_0) < 0 < G_\beta(u_0)$, which leads to $d(\alpha, \beta) > 0$ and to the fact that $u_0$ is a positive solution of \eqref{eq:D}, see Remark \ref{rem:positivity}. 

\ref{prop:MP-Nehari-resonant:3}
Let $\alpha = \alpha_*$. 
Then it follows from $H_\alpha(\varphi_q)=0=G_\beta(\varphi_q)$ that $t\varphi_q \in \N$ for any $t\neq0$ and $\E(t\varphi_q)=0$ for 
any $t$. Since we already know that $d(\alpha,\beta)\ge 0$, we conclude that $d(\alpha,\beta)=0$ and it is attained by $t\varphi_q$ for any $t \neq 0$. 
(Note that equality $d(\alpha,\beta)=0$ also follows from Lemma \ref{lem:seq-to-zero}.)
On the other hand, we see from \eqref{eq:Nehari} that any ground state $u_0$ of $\E$ must satisfy $G_\beta(u_0) = 0$. Recalling that $\beta = \lambda_1(q)$, we conclude that $u_0 \in \mathbb{R}\varphi_q \setminus \{0\}$.

\ref{prop:MP-Nehari-resonant:4}
Let $\alpha > \alpha_*$. 
We start by proving that $d(\alpha,\beta)=0$. 
Choose any $w_n\in\W \setminus \mathbb{R}\varphi_q$ such that $0<\|\nabla w_n-\nabla \varphi_q\|_p<1/n$ for $n \in \mathbb{N}$. 
Then, for sufficiently large $n \in \mathbb{N}$, we have $H_\alpha(w_n)<0$ because of $H_\alpha(\varphi_q)<0$, and $G_\beta(w_n) > 0$ because of $\beta = \lambda_1(q)$. 
Therefore, Proposition \ref{prop:minpoint} guarantees the existence of a unique maximum point $t_n>0$ of $\E(t w_n)$ on $[0,\infty)$ and $t_nw_n \in \N$. 
Moreover, we obtain (see \eqref{tu})
$$
t_n^{p-q}=-\frac{G_\beta(w_n)}{H_\alpha(w_n)}
=-\frac{G_\beta(\varphi_q)+o(1)}{H_\alpha(\varphi_q)+o(1)}
=-\frac{o(1)}{H_\alpha(\varphi_q)+o(1)}
=o(1)
\quad \text{as } n\to\infty.
$$
Thus, recalling that $w_n$ converges to $\varphi_q$ strongly in $\W$, we get $d(\alpha,\beta)=0$, since
$$
0\le d(\alpha,\beta)\le 
\E(t_nw_n)=\frac{p-q}{pq}G_\beta(t_nw_n)
\to 0
\quad \text{as} \quad 
n\to\infty.
$$

Suppose now that $d(\alpha,\beta)=0$ is attained by some $u_0\in \N$. 
This implies that $G_\beta(u_0)=0=H_\alpha(u_0)$, and hence $u_0 = t\varphi_q$ for some $t \neq 0$. However, this yields $\alpha = \alpha_*$, which is impossible by assumption.
\end{proof*}

\begin{proof*}{Proposition~\ref{prop:behavior-gs}} 
Let $\alpha_n > \lambda_1(p)$ and $\beta_n < \lambda_1(q)$ for all $n \in \mathbb{N}$, or $\lambda_1(p) < \alpha_n < \alpha_*$ and $\beta_n \leq \lambda_1(q)$ for all $n \in \mathbb{N}$, and let $u_n$ be a ground state of $\En$. Since $\En$ is even, we may assume that $u_n \ge 0$ for all $n \in \mathbb{N}$.
Recall that $u_n$ is a positive solution of $(GEV;\alpha_n,\beta_n)$ such that $H_{\alpha_n}(u_n)=-G_{\beta_n}(u_n)<0$, see Theorem \ref{thm:MP-Nehari}, Proposition \ref{prop:MP-Nehari-resonant}, and Remark \ref{rem:positivity}.
Note that $\|u_n\|_p$ is bounded if and only if $\|\nabla u_n\|_p$ is bounded, as it follows from the equality $\|\nabla u_n\|_p^p+\|\nabla u_n\|_q^q=\alpha_n\|u_n\|_p^p+\beta_n\|u_n\|_q^q$. 

\ref{prop:behavior-gs:1} Let $\alpha=\lambda_1(p)$ and $\beta\le \lambda_1(q)$. 
First, we show that $\lim\limits_{n\to\infty}\|u_n\|_p=\infty$. 
Suppose, by contradiction, that $\{u_n\}_{n \in \mathbb{N}}$ is bounded in $\W$, up to a subsequence. 
Then, Lemma~\ref{lem:conv-gs} ensures that $\{u_n\}_{n \in \mathbb{N}}$ has a subsequence $\{u_{n_k}\}_{k \in \mathbb{N}}$ which converges strongly in $\W$ to a solution $u_0$ of \eqref{eq:D}. 
Since \eqref{eq:D} has no nontrivial solutions (cf.\ Lemma \ref{lem:nonempty-Nehari}), 
we have $u_0 \equiv 0$, and hence $\|\nabla u_{n_k}\|_p\to 0$ as $k\to\infty$. 
Consider $\{w_k\}_{k \in \mathbb{N}}$, where $w_k:=u_{n_k}/\|\nabla u_{n_k}\|_q$ for $k \in \mathbb{N}$. Noting that $H_{\alpha_n}(u_n)<0$, we apply Lemma~\ref{lem:bdd-PS-2} to deduce that $\beta=\lambda_1(q)$ and that $\{w_k\}_{k \in \mathbb{N}}$ has a subsequence convergent to $w_0:=\varphi_q/\|\nabla \varphi_q\|_q$ weakly in $\W$ and strongly in $W_0^{1,q}$. 
However, since $\alpha=\lambda_1(p)$ and $H_{\alpha_n}(u_n)<0$, we get $H_\alpha(w_0) = 0$, which contradicts Lemma \ref{lem:LID}.

Now, in order to prove that $\lim\limits_{n\to\infty}\En(u_n)=\infty$, we suppose, by contradiction, that $\limsup\limits_{n\to\infty}\En(u_n)<\infty$. 
Since we already know that $\|u_n\|_p\to\infty$, it follows from Lemma~\ref{lem:bdd-PS} that 
$\{v_n\}_{n \in \mathbb{N}}$, where $v_n:=u_n/\|u_n\|_p$ for $n \in \mathbb{N}$, has a subsequence strongly convergent in $\W$ to $v_0 = \varphi_p/\|\varphi_p\|_p$. However, this yields the following contradiction:
$$
o(1)=\frac{\En(u_n)}{\|u_n\|_p^q}=\frac{p-q}{pq} G_{\beta_n}(v_n) = \frac{p-q}{pq} G_\beta(v_0) + o(1) > 0.
$$

\ref{prop:behavior-gs:2} Let $\lambda_1(p)<\alpha<\alpha_*$ and $\beta=\lambda_1(q)$. 
If $\{u_n\}_{n \in \mathbb{N}}$ has a subsequence which is unbounded in $L^p(\Omega)$, then Lemma~\ref{lem:bdd-PS} implies $\alpha=\lambda_1(p)$ since $u_n \ge 0$ for all $n \in \mathbb{N}$. However, this is a contradiction, and hence $\{u_n\}_{n \in \mathbb{N}}$ is bounded in $\W$. Moreover, since $H_{\alpha_n}(u_n)<0$ for all $n \in \mathbb{N}$, Lemma~\ref{lem:conv-gs} implies the existence of a subsequence $\{u_{n_k}\}_{k \in \mathbb{N}}$ strongly convergent in $\W$ to a solution $u_0$ of \eqref{eq:D}. 

If we suppose that $u_0 \equiv 0$, then Lemma~\ref{lem:bdd-PS-2} guarantees that $\{w_{n_k}\}_{k \in \mathbb{N}}$, where $w_k:=u_{n_k}/\|\nabla u_{n_k}\|_q$ for $k \in \mathbb{N}$, has a subsequence convergent to $w_0=\varphi_q/\|\nabla \varphi_q\|_q$ weakly in $\W$ and strongly in $W^{1,q}_0$. 
Hence $H_\alpha(w_0)>0$ by $\alpha<\alpha_*$, but this contradicts the fact that $H_{\alpha_n}(u_n)<0$ for all $n \in \mathbb{N}$. 
Therefore, $u_0 \not\equiv 0$ and, consequently, $u_0 \in \N$.

Finally, let us show that $u_0$ is a ground state of $\E$, that is, $\E(u_0)=d(\alpha,\beta)$. 
Recall that $d(\alpha,\beta)>0$ by Theorem~\ref{prop:MP-Nehari-resonant} \ref{prop:MP-Nehari-resonant:2}, and hence any $v\in\N$ satisfies $\E(v)>0$ and $G_\beta(v)>0>H_\alpha(v)$. 
Thus, for sufficiently large $n \in \mathbb{N}$, Proposition \ref{prop:minpoint} guarantees the existence of $t_n > 0$ such that $t_nv \in \mathcal{N}_{\alpha_n,\beta_n}$, and hence
\begin{align*} 
\E (v) &=\max_{t\ge 0} \E(tv) \ge \E (t_nv) \\
&=\En(t_n v)+o(1)\ge d(\alpha_n,\beta_n)+o(1)=\En (u_n)+o(1) 
\quad \text{as } n\to\infty.
\end{align*}
Consequently, $\E(v) \ge \E(u_0)$ for any $v\in\N$, which implies that $u_0$ is a ground state of $\E$. 

\ref{prop:behavior-gs:3} Let $\alpha\ge \alpha_*$ and $\beta=\lambda_1(q)$. 
By the same arguments as in case \ref{prop:behavior-gs:2}, we see that $\{u_n\}_{n \in \mathbb{N}}$ is bounded in $\W$ and any of its subsequence has a subsequence strongly convergent in $\W$ to a solution $u_0$ of \eqref{eq:D}. We can assume that $u_0 \ge 0$.

Suppose, by contradiction, that $u_0 \not\equiv 0$. Then we see that $\alpha=\alpha_*$ since it is proved in \cite[Proposition 4 (ii)]{BobkovTanaka2015} that \eqref{eq:D} has no positive solutions provided $\alpha>\alpha_*$ and $\beta = \lambda_1(q)$. 
Furthermore, since $t \varphi_q$ is not a solution of \eqref{eq:D} for $t\neq0$, we have $u_0 \not\in \mathbb{R}\varphi_q$, and hence $\E(u_0)>0=d(\alpha,\beta)$ because 
$d(\alpha,\beta)$ is attained only by $t\varphi_q$, see Theorem \ref{prop:MP-Nehari-resonant} \ref{prop:MP-Nehari-resonant:3}. 
By the same argument as in case \ref{prop:behavior-gs:2}, it can be shown that $\E(u_0)\le \E(v)$ for any $v\in\N\setminus \mathbb{R}\varphi_q$. However, this yields a contradiction since there exists a sequence $\{v_n\}_{n \in \mathbb{N}} \subset \N \setminus \mathbb{R} \varphi_q$ satisfying $\E(v_n) \to 0$ as $n \to\infty$, see Lemma~\ref{lem:seq-to-zero}. 
Consequently, any subsequence of $\{u_n\}_{n \in \mathbb{N}}$ has a subsequence strongly convergent in $\W$ to $0$, which implies that $\{u_n\}_{n \in \mathbb{N}}$ also converges strongly in $\W$ to $0$.

The second claim of the assertion \ref{prop:behavior-gs:3} directly follows from Lemma \ref{lem:bdd-PS-2}. 
\end{proof*}

\begin{proof*}{Proposition~\ref{prop:property-curve}}
\ref{prop:property-curve:1}, \ref{prop:property-curve:2} 
The assertions are obvious. 

\ref{prop:property-curve:3} Noting that the functional $H_\alpha$ in the constraint of $\beta_*(\alpha)$ is weakly lower semicontinuous, we apply the direct method of the calculus of variations to obtain that $\beta_*(\alpha)$ is attained for all $\alpha \geq \lambda_1(p)$.

\ref{prop:property-curve:4} Let $\alpha<\alpha_*$. If $\alpha < \lambda_1(p)$, then $\beta_*(\alpha) = \infty > \lambda_1(q)$. Assume that $\alpha \geq \lambda_1(p)$. Then $\beta_*(\alpha)$ is attained by \ref{prop:property-curve:3}. Therefore, recalling that $\|\nabla u\|_q^q/\|u\|_q^q=\lambda_1(q)$ if and only if $u \in \mathbb{R}\varphi_q$, and $H_\alpha(\varphi_q)>0$, we see that $\beta_*(\alpha)>\lambda_1(q)$. 

\ref{prop:property-curve:6} 
Let $\mathcal{B}(\alpha):=\{u\in\W\setminus\{0\}\,:\,H_\alpha(u)\le 0\}$ denotes the admissible set of $\beta_*(\alpha)$. 
Evidently, $\mathcal{B}(\alpha)$ satisfies $\mathcal{B}(\alpha)\subset \mathcal{B}(\alpha^\prime)$ provided $\alpha\le \alpha^\prime$, which implies that $\beta_*(\alpha)$ is nonincreasing for $\alpha \geq \lambda_1(p)$. 

Let us show that $\beta_*(\cdot)$ decreases on $[\lambda_1(p), \alpha_*]$. 
Suppose, by contradiction, that there exist $\alpha, \alpha^\prime$ such that $\lambda_1(p) \leq \alpha < \alpha^\prime \leq \alpha_*$ and $\beta_*(\alpha) = \beta^*(\alpha^\prime)$. 
By the assertion \ref{prop:property-curve:3}, $\beta_*(\alpha)$ and $\beta_*(\alpha^\prime)$ are attained. Let $u_0 \not\equiv 0$ be a minimizer for $\beta_*(\alpha)$, that is, $H_\alpha(u_0) \leq 0$ and $G_{\beta_*(\alpha)}(u_0) = 0$. 
Since $H_\alpha$ and $G_\beta$ are even, we may assume that $u_0 \geq 0$. 
Then we see that $H_\alpha(u_0) = 0$. 
Indeed, if we suppose that $H_\alpha(u_0) < 0$, then $u_0$ is an interior point of $\mathcal{B}(\alpha)$. Hence, we get $(|\nabla u_0\|_q^q/\|u_0\|_q^q)^\prime = 	0$ in $(\W)^*$, which implies that $G_{\beta_*(\alpha)}^{\prime}(u_0)=0$. 
This means that $u_0 \in ES(q;\beta_*(\alpha)) \setminus \{0\}$. 
Since there exist no constant sign eigenfunctions of $-\Delta_q$ except the first eigenfunctions $\mathbb{R}\varphi_q$, 
we must have $\beta_*(\alpha)=\lambda_1(q)$ and $u_0 \in \mathbb{R}\varphi_q$, which is a contradiction since $\beta_*(\alpha) > \lambda_1(q)$ by the assertion \ref{prop:property-curve:4}. 

As a result, we see that $u_0 \in \mathcal{B}(\alpha^\prime)$ 
with $H_{\alpha^\prime}(u_0) < H_{\alpha}(u_0)=0$ and $G_{\beta_*(\alpha^\prime)}(u_0) = 0$ because we are assuming $\alpha < \alpha^\prime$ and $\beta_*(\alpha) = \beta^*(\alpha^\prime)$. 
Applying the above argument to $\beta^*(\alpha^\prime)$, we again get a contradiction. 
Hence, $\beta_*(\cdot)$ is decreasing on $[\lambda_1(p),\alpha_*]$. 

\ref{prop:property-curve:5} 
Fix any $\alpha \geq \lambda_1(p)$ and take any sequence $\{\alpha_n\}_{n \in \mathbb{N}}$ which converges to $\alpha$. (If $\alpha = \lambda_1(p)$, then we assume that $\alpha_n > \lambda_1(p)$ for all $n \in \mathbb{N}$). 
By the assertion \ref{prop:property-curve:3}, for each $n \in \mathbb{N}$ we can find a minimizer  $u_n\in \mathcal{B}(\alpha_n)$ of $\beta_*(\alpha_n)$. We can assume that $\|u_n\|_p = 1$ and $u_n\ge 0$ for all $n \in \mathbb{N}$. 
Moreover, since $\|\nabla u_n\|_p^p \leq \alpha_n\|u_n\|_p^p=\alpha_n=\alpha+o(1)$ for all $n \in \mathbb{N}$, we may suppose that $u_n$ converges, up to a subsequence, weakly in $\W$ and strongly in $L^p(\Omega)$ to some $\tilde{u} \in \W$
 with $\|\tilde{u}\|_p = 1$. This readily implies that 
\begin{equation*}
\beta_*(\alpha)\le 
\dfrac{\|\nabla \tilde{u}\|_q^q}{\|\tilde{u}\|_q^q} \le 
\liminf_{n\to\infty} \dfrac{\|\nabla u_n\|_q^q}{\|u_n\|_q^q}=
\liminf\limits_{n\to\infty} \beta_*(\alpha_n),
\end{equation*}
that is, $\beta_*(\cdot)$ is lower semicontinuous. 

Let us show now the upper semicontinuity of $\beta_*$. By the assertion \ref{prop:property-curve:3}, we can find $u_0\in\mathcal{B}(\alpha)$ such that $u_0\ge 0$ and $\|\nabla u_0\|_q^q/\|u_0\|_q^q = \beta_*(\alpha)$. 

Assume first that $\alpha > \lambda_1(p)$. Thus, we have $H^\prime_\alpha(u_0) \neq 0$ in $(\W)^*$, and hence we can find $\theta \in C_0^\infty(\Omega)$ such that $\langle H^\prime_\alpha(u_0),\theta\rangle <0$. 
This ensures that $H_\alpha(u_0+t\theta)<0$ for all $t>0$ small enough (cf.\ \eqref{eq:lemma:3.10:1}). Thus, for any $\varepsilon>0$ there exists sufficiently small $t>0$ such that $\|\nabla (u_0+t\theta)\|_q^q/\|u_0+t \theta\|_q^q<\beta_*(\alpha)+\varepsilon$. 
Moreover, for sufficiently large $n \in \mathbb{N}$ it holds $H_{\alpha_n}(u_0+t\theta)<0$, that is, $u_0+t\theta\in\mathcal{B}(\alpha_n)$. Consequently, $\beta_*(\alpha_n)\le \|\nabla (u_0+t\theta)\|_q^q/\|u_0+t\theta\|_q^q <\beta_*(\alpha)+\varepsilon$ for sufficiently large $n \in \mathbb{N}$. 
Since $\varepsilon>0$ was taken arbitrarily, we get $\limsup\limits_{n\to\infty}\beta_*(\alpha_n)\le \beta_*(\alpha)$, and 
hence the upper semicontinuity follows. 

Assume now that $\alpha=\lambda_1(p)$. Since $\beta_*(\alpha)$ is nonincreasing (see \ref{prop:property-curve:6}) and $\alpha_n > \lambda_1(p)$, we have $\limsup\limits_{n\to\infty} \beta_*(\alpha_n)\le \beta_*(\lambda_1(p))$. Thus, $\beta_*(\alpha)$ is right upper semicontinuous at $\lambda_1(p)$. 
\end{proof*}

\begin{proof*}{Theorem~\ref{thm:negative-gs}} 
\ref{thm:negative-gs:1}
Assume first that $\alpha=\lambda_1(p)$ and $\lambda_1(q)<\beta<\beta_*$. Then Proposition~\ref{prop:GM3} \ref{prop:GM3:2} guarantees that $d(\alpha,\beta) <0$ and it is attained by a global minimizer of $\E$. 

Assume now that $\lambda_1(p)<\alpha<\alpha_*$ and $\lambda_1(q)<\beta<\beta_*(\alpha)$. 
First, we show that $d(\alpha,\beta)<0$. Since $G_\beta(\varphi_q)<0<H_\alpha(\varphi_q)$, there exists a unique $t_q>0$ such that $\E(t_q\varphi_q)=\min\limits_{t\ge 0} \E(t\varphi_q)<0$ and $t_q\varphi_q\in\N$, see Proposition \ref{prop:minpoint}.
Hence, $d(\alpha,\beta)\le \E(t_q\varphi_q)<0$. 

Let $\{u_n\}_{n \in \mathbb{N}}$ be a minimizing sequence for $d(\alpha,\beta)$, that is,
$u_n \in \N$ and $\E(u_n)\to d(\alpha,\beta)$ as $n\to\infty$. 
Since $d(\alpha,\beta)<0$, we have $\E(u_n)<0$ and $G_\beta(u_n)<0<H_\alpha(u_n)$ for sufficiently large $n \in \mathbb{N}$. 
We claim that $\{u_n\}_{n \in \mathbb{N}}$ is bounded in $\W$. 
Suppose, by contradiction, that $\|\nabla u_n\|_p\to \infty$, up to a subsequence. Setting $v_n:=u_n/\|\nabla u_n\|_p$ and choosing again an appropriate subsequence, we may suppose that $v_n$ converges to some $v_0$ weakly in $\W$ and strongly in $L^p(\Omega)$ as $n \to \infty$. Then, noting that $G_\beta$ is bounded on bounded sets, we obtain 
$$
H_\alpha(v_0)\le \liminf_{n\to\infty} H_\alpha(v_n)
=-\lim_{n\to\infty}\frac{G_\beta(v_n)}{\|\nabla u_n\|_p^{p-q}}= 0.
$$
Moreover, $H_\alpha(v_n)=1-\alpha\|v_n\|_p^p$ yields $\|v_0\|_p=1/\alpha$, and hence $v_0 \not\equiv 0$. Since $\beta<\beta_*(\alpha)$, we must have $G_\beta(v_0)>0$. 
However, since $G_\beta(v_n)<0$, we get $G_\beta(v_0) \le 0$. This contradiction implies that $\{u_n\}_{n \in \mathbb{N}}$ is bounded in $\W$. 

The boundedness of $\{u_n\}_{n \in \mathbb{N}}$ implies that $u_n$ converges to some $u_0$ weakly in $\W$ and strongly in $L^p(\Omega)$, up to a subsequence. 
It is clear that $u_0 \not\equiv 0$ since $\E(u_0) \le \liminf\limits_{n\to\infty}\E(u_n) = d(\alpha,\beta) < 0$. 
Moreover, $G_\beta(u_0) \le 0$. Since $\beta<\beta_*(\alpha)$, we have $H_\alpha(u_0)>0$. 
Since $H_\alpha(u_n)+G_\beta(u_n)=0$ leads to $H_\alpha(u_0)+G_\beta(u_0)\le 0$, we deduce that $G_\beta(u_0)<0<H_\alpha(u_0)$. 
As a result, there exists a unique $t_0>0$ such that $\E(t_0u_0)=\min\limits_{t\ge 0} \E(tu_0)$ and $t_0u_0\in\N$, see Proposition \ref{prop:minpoint}.
Therefore, we get
\begin{align*} 
d(\alpha,\beta) \le \E(t_0u_0)=\min_{t\ge 0} \E(tu_0) \le \E(u_0)\le \liminf_{n\to\infty}\E(u_n)=d(\alpha,\beta), 
\end{align*}
which implies that $t_0=1$, $u_0\in\N$, and $\E(u_0)=d(\alpha,\beta)$. 
Finally, since $d(\alpha, \beta) < 0$, we conclude that $u_0$ is a positive solution of \eqref{eq:D}, see Remark~\ref{rem:positivity}. 

\ref{thm:negative-gs:2} Let $\alpha \geq \lambda_1(p)$ and $\beta > \beta_*(\alpha)$. 
Let us show that we can find $v_0 \in \W$ such that $H_\alpha(v_0)=0$ and $G_\beta(v_0)<0$. Then Lemma~\ref{lem:seq-to-negative-infty} will imply the desired result. 

If $\alpha = \lambda_1(p)$, then we conclude by choosing $v_0 = \varphi_p$ since $\beta > \beta_*=\beta_*(\alpha)$. 

Assume that $\lambda_1(p) < \alpha < \alpha_*$. By Proposition \ref{prop:property-curve} \ref{prop:property-curve:3}, $\beta_*(\alpha)$ is attained. Let $u_0 \not\equiv 0$ be a corresponding minimizer, that is, $H_\alpha(u_0) \leq 0$ and $\|\nabla u_0\|_q^q/\|u_0\|_q^q = \beta_*(\alpha) < \beta$. The latter inequality yields $G_\beta(u_0) < 0$. 
If $H_\alpha(u_0) = 0$, then we are done. If we suppose that $H_\alpha(u_0) < 0$, then, arguing as in the proof of 
Proposition~\ref{prop:property-curve} \ref{prop:property-curve:6}, we obtain a contradiction.

If $\alpha = \alpha_*$, then we conclude by choosing $v_0 = \varphi_q$. 

Assume finally that $\alpha > \alpha_*$. 
Note that $\lambda_1(q)=\beta_*(\alpha)$, see Proposition \ref{prop:property-curve} \ref{prop:property-curve:2}. 
To prove the claim, we will show the existence of a sequence $\{v_n\}_{n \in \mathbb{N}} \subset \W \setminus \{0\}$ such that
\begin{equation}\label{eq:thm215:1}
H_\alpha(v_n) = 0
\quad 
\text{and}
\quad 
\frac{\|\nabla v_n\|_q^q}{\|v_n\|_q^q} \to \lambda_1(q)=\beta_*(\alpha)
\,(<\beta) 
\quad \text{as } n \to \infty.
\end{equation}
Recalling that $H_\alpha(\varphi_q) < 0$ by $\alpha > \alpha_*$, we see that $v_n$ cannot converge to $\varphi_q$ strongly in $\W$. Therefore, we must find a sequence $\{v_n\}$ which converges to $\varphi_q$ weakly in $\W$ and strongly in $W_0^{1,q}$.
Let us fix any function $\theta \in C_0^\infty(\Omega)$ such that $\|\nabla \theta\|_p = 1$, and consider 
$$
\theta_n(x) = n^{\frac{N}{p}-1}\, \theta(n\,x).
$$
By straightforward calculations, we have 
\begin{align}
\label{eq:thm:2.14:2}
\|\nabla \theta_n\|_p = \|\nabla \theta\|_p = 1&, 
\qquad 
\|\theta_n\|_p = \frac{1}{n} \|\theta\|_p \to 0,\\
\notag
\|\nabla \theta_n\|_q = \frac{1}{n^{\frac{N(p-q)}{pq}}}\|\nabla \theta\|_q \to 0&, 
\qquad 
\|\theta_n\|_q = \frac{1}{n^{1+\frac{N(p-q)}{pq}}}\|\theta\|_q \to 0,
\end{align}
as $n \to \infty$. Therefore, $\theta_n \to 0$ weakly in $\W$ and strongly in $W_0^{1,q}$ and $L^p(\Omega)$ and $L^q(\Omega)$.

Consider now the function $v_n := \varphi_q + \gamma_n \theta_n$ for $n \in \mathbb{N}$, where a positive constant $\gamma_n > 0$ is chosen such that $H_\alpha(v_n) = 0$, or, equivalently, 
\begin{equation}\label{eq:2}
\frac{\|\nabla \varphi_q + \gamma_n \nabla \theta_n\|_p^p}{\|\varphi_q + \gamma_n \theta_n\|_p^p} = \alpha, 
\end{equation}
for all $n \in \mathbb{N}$ large enough.
Note that such $\gamma_n >0$ exists, since
$$
\frac{\|\nabla \varphi_q\|_p^p}{\|\varphi_q\|_p^p} = \alpha_* < \alpha
\quad \text{and} \quad 
\frac{\|\nabla \varphi_q + C \nabla \theta_n\|_p^p}{\|\varphi_q + C\theta_n\|_p^p}=
\frac{\|\frac{1}{C}\nabla \varphi_q + \nabla \theta_n\|_p^p}{\|\frac{1}{C}\varphi_q + \theta_n\|_p^p} 
> \alpha
$$
for all sufficiently large $C>0$ and $n \in \mathbb{N}$, see \eqref{eq:thm:2.14:2}.
Moreover, $\{\gamma_n\}_{n \in \mathbb{N}}$ is bounded. 
Indeed, by the triangle inequality, we have
\begin{align*}
\|\nabla \varphi_q + \gamma_n \nabla \theta_n\|_p 
&\geq 
|\|\nabla \varphi_q\|_p - \gamma_n\|\nabla \theta_n\|_p| 
= 
|\|\nabla \varphi_q\|_p - \gamma_n|,\\
\|\varphi_q + \gamma_n\theta_n\|_p 
&\leq 
\|\varphi_q\|_p + \gamma_n\|\theta_n\|_p.
\end{align*}
Therefore, if we suppose that $\gamma_n \to \infty$ as $n \to \infty$, up to a subsequence, then 
$$
\frac{\|\nabla \varphi_q + \gamma_n \nabla \theta_n\|_p}{\|\varphi_q + \gamma_n \theta_n\|_p}
\geq 
\frac{|\|\nabla \varphi_q\|_p - \gamma_n|}{\|\varphi_q\|_p + \gamma_n\|\theta_n\|_p} = 
\frac{\left|\frac{1}{\gamma_n}\|\nabla \varphi_q\|_p - 1\right|}{\frac{1}{\gamma_n}\|\varphi_q\|_p + \|\theta_n\|_p} 
\to \infty
$$
as $n \to \infty$, which is impossible in view of \eqref{eq:2}.
Consequently, since
\begin{gather*}
|\|\nabla \varphi_q\|_q - \gamma_n \|\nabla \theta_n\|_q| 
\leq 
\|\nabla \varphi_q + \gamma_n \nabla \theta_n\|_q \leq  
\|\nabla \varphi_q\|_q + \gamma_n \|\nabla \theta_n\|_q, \\
|\|\varphi_q\|_q - \gamma_n \|\theta_n\|_q| 
\leq 
\|\varphi_q + \gamma_n \theta_n\|_q \leq 
\|\varphi_q\|_q + \gamma_n  \|\theta_n\|_q,
\end{gather*}
we get
$$
\|\nabla \varphi_q + \gamma_n \nabla \theta_n\|_q \to \|\nabla \varphi_q\|_q
\quad \text{and} \quad 
\|\varphi_q + \gamma_n \theta_n\|_q \to \|\varphi_q\|_q
\quad \text{as }
n \to \infty.
$$
Finally, we conclude that $v_n = \varphi_q + \gamma_n \theta_n$ satisfies \eqref{eq:thm215:1}.
Since $\beta > \beta_*(\alpha)$, we can choose $n \in \mathbb{N}$ large enough to get $H_\alpha(v_n) = 0$ and $G_\beta(v_n) < 0$, and then Lemma \ref{lem:seq-to-negative-infty} gives $d(\alpha, \beta) = -\infty$.
\end{proof*}

\begin{proof*}{Proposition~\ref{prop:behavior-gs-2} \ref{prop:behavior-gs-2:1} $\sim$ \ref{prop:behavior-gs-2:3}} 
Let $\lambda_1(p) < \alpha_n < \alpha_*$ and $\lambda_1(q) < \beta_n < \beta_*(\alpha_n)$ for all $n \in \mathbb{N}$, and let $u_n$ be a ground state of $\En$. Since $\En$ is even, we may assume that $u_n\ge 0$ for all $n \in \mathbb{N}$.
Recall that $\En(u_n) < 0$ and $u_n$ is a positive solution of $(GEV;\alpha_n,\beta_n)$, see Theorem \ref{thm:negative-gs}.

First, we claim that $\{u_n\}_{n \in \mathbb{N}}$ is bounded in $\W$. 
If $\alpha > \lambda_1(p)$, then the result follows from Lemma~\ref{lem:bdd-PS}. 
If $\alpha=\lambda_1(p)$, $\beta < \beta_*=\beta_*(\lambda_1(p))$, and we suppose that $\|u_n\|_p \to \infty$, then Lemma \ref{lem:bdd-PS} ensures that $v_n:=u_n/\|u_n\|_p$ converges strongly in $\W$, up to a subsequence, to $v_0=\varphi_p/\|\varphi_p\|_p$. 
Since $G_\beta(v_0)=\lim\limits_{n\to\infty}G_{\beta_n}(v_n)\le 0$, we get a contradiction 
to the assumption $\beta<\beta_*$. Hence, $\{u_n\}_{n \in \mathbb{N}}$ is bounded in $L^p(\Omega)$. This implies the boundedness of $\{u_n\}_{n \in \mathbb{N}}$ in $\W$ since $u_n \in \mathcal{N}_{\alpha_n,\beta_n}$. 
Therefore, Lemma \ref{lem:conv-gs} guarantees that $u_n$ converges strongly in $\W$, up to a subsequence, to a nonnegative solution $u_0$ of \eqref{eq:D}. 

\ref{prop:behavior-gs-2:1} Let $\beta=\lambda_1(q)$. 
Suppose that $u_0 \not\equiv 0$. Then $u_0 \in \N$ and $d(\alpha,\beta)\le \E(u_0)\le \liminf\limits_{n\to\infty}\En(u_n)\le 0$. If $\alpha=\lambda_1(p)$, then $\N$ is empty (see Lemma \ref{lem:nonempty-Nehari}), which is impossible. 
Thus, $u_0 \equiv 0$. 
If $\lambda_1(p)<\alpha<\alpha_*$, then we again get a contradiction since $d(\alpha,\beta)>0$ by Theorem \ref{prop:MP-Nehari-resonant} \ref{prop:MP-Nehari-resonant:2}. 
Finally, if $\alpha=\alpha_*$, then $d(\alpha, \beta)=0$ and $u_0=t\varphi_q$ for some $t>0$, see Theorem \ref{prop:MP-Nehari-resonant} \ref{prop:MP-Nehari-resonant:3}.
However, this is a contradiction because $t\varphi_q$ is not a solution of \eqref{eq:D}. 
Therefore, we conclude that $u_0 \equiv 0$. 
Applying the above arguments to any subsequence of $\{u_n\}_{n \in \mathbb{N}}$, we deduce that $u_n$ converges to $0$ strongly in $\W$. 

Recalling that $\beta=\lambda_1(q)$ and $G_{\beta_n}(u_n)<0$, the second claim of the assertion \ref{prop:behavior-gs-2:1} follows from the same arguments as in the the proof of Proposition~\ref{prop:GM2} \ref{prop:GM2:3}.

\ref{prop:behavior-gs-2:2} Let $\alpha=\lambda_1(p)$ and $\lambda_1(q)<\beta<\beta_*$. 
Let us show first that $u_0 \not\equiv 0$. 
Take a global minimizer $w_0\in\N$ of $\E$ (which exists by Proposition~\ref{prop:GM3} \ref{prop:GM3:2}). 
Since $\E(w_0)<0$ yields $G_\beta(w_0)<0<H_\alpha(w_0)$, we have $G_{\beta_n}(w_0)<0<H_{\alpha_n}(w_0)$ for all sufficiently large $n \in \mathbb{N}$. 
Therefore, using Proposition \ref{prop:minpoint}, we can find $t_n>0$ such that $t_n w_0 \in \mathcal{N}_{\alpha_n,\beta_n}$. 
Hence, passing to the limit as $n\to\infty$ 
in the following chain of inequalities:
$$
\En(u_n)=d(\alpha_n,\beta_n)\le \En(t_n w_0)
=\min_{t\ge 0} \En(t w_0)\le \En(w_0), 
$$
we get $\E(u_0)\le \E(w_0)<0$. Thus, $u_0 \not\equiv 0$ and $\liminf\limits_{n\to\infty}\E(u_n)<0$. 
As a result, Lemma \ref{lem:conv-gs} guarantees that $u_0$ is a ground state of $\E$, and we conclude that $u_0$ is a global minimizer of $\E$. 

\ref{prop:behavior-gs-2:3} Let $\lambda_1(p)<\alpha<\alpha_*$ and $\beta=\beta_*(\alpha)$. 
Choose any $\lambda_1(q)<\beta^\prime<\beta_*(\alpha)$ and take a ground state $w_0$ of $E_{\alpha,\beta^\prime}$ (the existence is shown by Theorem \ref{thm:negative-gs} \ref{thm:negative-gs:1}). 
Note that $E_{\alpha,\beta^\prime}(w_0) < 0$. 
Then, using the same arguments as in the proof of \ref{prop:behavior-gs-2:2}, we can show that $\liminf\limits_{n\to\infty}\E(u_n)<0$ and $u_0\not\equiv 0$. 
Indeed, by Proposition \ref{prop:minpoint}, there exists $t_n>0$ such that $t_nw_0 \in  \mathcal{N}_{\alpha_n,\beta_n}$ for all sufficiently large $n \in \mathbb{N}$, and 
\begin{align*}
\En(u_n) \le \En(t_n w_0) \le \En(w_0)
=\frac{1}{p}H_{\alpha_n}(w_0)+\frac{1}{q}G_{\beta_n}(w_0) 
\le \frac{1}{p}H_{\alpha_n}(w_0)+\frac{1}{q}G_{\beta^\prime}(w_0) 
\end{align*}
since $\beta^\prime<\beta$ and $\beta_n\to\beta$. 
Letting $n\to\infty$, we get $\E(u_0)\le E_{\alpha,\beta^\prime}(w_0)<0$ and $u_0 \not\equiv 0$.
Finally, Lemma~\ref{lem:conv-gs} ensures that $u_0$ is a ground state of $\E$. 
\end{proof*}

\begin{proof*}{Proposition~\ref{prop:behavior-gs-2} \ref{prop:behavior-gs-2:4}}
We may assume that $u_n\ge 0$ for all $n \in \mathbb{N}$. 
First, we show that $\lim\limits_{n\to\infty}\En(u_n)=-\infty$. Fix any $R>0$. 
Since $d(\alpha,\beta)=-\infty$ by Proposition \ref{Iinf}, we can choose $v\in\N$ such that $\E(v)\le -R$, and so $G_\beta(v)<0<H_\alpha(v)$. 
Then, for sufficiently large $n \in \mathbb{N}$ we have $G_{\beta_n}(v)<0<H_{\alpha_n}(v)$. 
Thus, there exists $t_n>0$ such that $t_n v \in \mathcal{N}_{\alpha_n,\beta_n}$ and $\En(t_n v)=\min\limits_{t \geq 0} \En (tv)$, see Proposition \ref{prop:minpoint}.
As a result, we see that $\En (u_n)\le \En(t_n v)\le \En(v)$, and hence $\limsup\limits_{n\to\infty}\En (u_n)\le \E(v)\le -R$. 
Since $R>0$ is arbitrary, we conclude that $\lim\limits_{n\to\infty}\En(u_n)=-\infty$. 
This implies that $\{u_n\}_{n \in \mathbb{N}}$ does not have subsequences bounded in $\W$. 
Hence, $\|\nabla u_n\|_p\to\infty$ (and also $\|u_n\|_p\to\infty$) as $n \to \infty$. Finally, Lemma \ref{lem:bdd-PS} implies that $u_n/\|u_n\|_p$ converges to $\varphi_p/\|\varphi_p\|_p$ strongly in $\W$. 
\end{proof*} 

\begin{proof*}{Theorem~\ref{thm:negative-gs-2}}
\ref{thm:negative-gs-2:1} Let $\lambda_1(p)<\alpha<\alpha_*$ and $\beta=\beta_*(\alpha)$. 
Choose a sequence $\{(\alpha_n,\beta_n)\}_{n \in \mathbb{N}}$ satisfying $\lambda_1(p)<\alpha_n<\alpha_*$ and $\lambda_1(q)<\beta_n<\beta_*(\alpha_n)$ 
for all $n \in \mathbb{N}$, and $\lim\limits_{n\to\infty}\alpha_n=\alpha$ and $\lim\limits_{n\to\infty}\beta_n=\beta$. Then, according to Theorem \ref{thm:negative-gs} \ref{thm:negative-gs:1}, we can find a ground state $u_n$ of $\En$ such that $\En(u_n)<0$ for each $n \in \mathbb{N}$.
Thanks to Proposition \ref{prop:behavior-gs-2} \ref{prop:behavior-gs-2:3}, $\{u_n\}_{n \in \mathbb{N}}$ has a subsequence strongly convergent in $\W$ to a ground state of $\E$ and $d(\alpha,\beta)<0$. 

\ref{thm:negative-gs-2:2} The assertion is proved in Proposition \ref{Iinf}.
\end{proof*}

In order to prove Theorem~\ref{thm:multi}, we prepare the following result. 

\begin{lemma}\label{lem:minimizer-sol} 
	Assume $\lambda_1(p)<\alpha<\alpha_*$ and $\beta=\beta_*(\alpha)$. 
	Let $u_0$ be a nonnegative minimizer of $\beta_*(\alpha)$, that is, $u_0\not\equiv 0$, $H_\alpha(u_0)\le 0$, and $\|\nabla u_0\|_q^q/\|u_0\|_q^q=\beta_*(\alpha)$. 
	Then there exists $t>0$ such that $tu_0$ is a positive solution of \eqref{eq:D} with $\E(tu_0)=0$. 
\end{lemma}
\begin{proof}
	First, we note that $H_\alpha(u_0)=0$. 
	Indeed, if $H_\alpha(u_0)<0$, then $u_0$ is an interior point of the admissible set of $\beta_*(\alpha)$. Therefore, $(\|\nabla u_0\|_q^q/\|u_0\|_q^q)^\prime = 0$ in $(\W)^*$, which implies that $G_\beta^\prime(u_0)=0$, and hence $u_0$ is a nontrivial and nonnegative eigenfunction of $-\Delta_q$ associated to $\beta$. However, this is a contradiction since $\beta_*(\alpha)>\lambda_1(q)$ for $\alpha < \alpha_*$, see Proposition \ref{prop:property-curve} \ref{prop:property-curve:4}. Thus, $H_\alpha(u_0)=0$. 
	
	According to the Lagrange multipliers rule, there exists $\lambda\in\mathbb{R}$ such that 
	\begin{equation}\label{eq:minimizer-sol:1}
	G_\beta^\prime(u_0)=\lambda H_\alpha^\prime(u_0) \quad{\rm in}\quad (\W)^*. 
	\end{equation}
	Since $u_0$ is a regular point of $G_\beta$, we have $\lambda \not= 0$. 
	In order to get $\lambda<0$, we suppose, by contradiction, that $\lambda > 0$. 
	Since $u_0 \geq 0$ and $\alpha > \lambda_1(p)$, $u_0$ is a regular point of $H_\alpha$, and hence we can find
	$\theta\in\W$ such that $\langle H_\alpha^\prime(u_0),\theta\rangle<0$, and so 
	$\langle G_\beta^\prime(u_0),\theta\rangle<0$ by \eqref{eq:minimizer-sol:1} 
	and our assumption $\lambda>0$. 
	Taking sufficiently small $\varepsilon_0>0$, we have
	$$
	\langle H_\alpha^\prime(u_0+\varepsilon\theta),\theta\rangle<0 
	\quad {\rm and }  \quad 
	\langle G_\beta^\prime(u_0+\varepsilon\theta),\theta\rangle<0
	\quad 
	{\rm for\ all}\ \varepsilon\in[0,\varepsilon_0). 
	$$
	Therefore, according to the mean value theorem, there exist $\varepsilon_1\in (0,\varepsilon)$ and $\varepsilon_2\in(0,\varepsilon)$ such that
	\begin{align*} 
	H_\alpha(u_0+\varepsilon\theta) &= 
	H_\alpha(u_0) + 
	\varepsilon\langle H_\alpha'(u_0+\varepsilon_1\theta), \theta \rangle = \varepsilon\langle H_\alpha'(u_0+\varepsilon_1\theta), \theta \rangle
	< 0, 
	\\
	G_\beta(u_0+\varepsilon\theta) &= 
	G_\beta(u_0) + 
	\varepsilon \langle G_\beta'(u_0+\varepsilon_2\theta), \theta \rangle = \varepsilon \langle G_\beta'(u_0+\varepsilon_2\theta), \theta \rangle
	< 0. 
	\end{align*}
	However, this implies that
	$$
	\frac{\|\nabla (u_0+\varepsilon\theta)\|_q^q}{\|u_0+\varepsilon\theta\|_q^q}
	<\beta=\beta_*(\alpha) 
	\quad {\rm and}\quad H_\alpha(u_0+\varepsilon\theta)<0, 
	$$
	which contradicts the definition of $\beta_*(\alpha)$. Therefore, $\lambda<0$. 
	Finally, taking $t=|\lambda|^\frac{1}{p-q}$, we see from \eqref{eq:minimizer-sol:1} that $tu_0$ is a positive solution of \eqref{eq:D}. 
\end{proof}

\begin{proof*}{Theorem~\ref{thm:multi}} 
The existence of the least energy solution $u_1$ is already shown in Theorem~\ref{thm:negative-gs}. 
In the case $\beta=\beta_*(\alpha)$, Lemma~\ref{lem:minimizer-sol} and Proposition~\ref{prop:property-curve} \ref{prop:property-curve:3} imply the existence of the second solution whose energy is zero.
Finally, if $\beta<\beta_*(\alpha)$, then Proposition~\ref{prop:miniNehari+} implies the desired result. 
\end{proof*}

\subsection{Properties of the least energy}
In this subsection, we prove Proposition~\ref{prop:property-least-energy}. First we prepare two auxiliary facts.

\begin{lemma}\label{lem:usc} 
Let $(\alpha, \beta) \in \mathbb{R}^2$. 
Assume that 
$d(\alpha, \beta)$ is attained and $d(\alpha, \beta) \neq 0$. 
Then $d$ is upper semicontinuous at $(\alpha, \beta)$.
\end{lemma}
\begin{proof} 
Let $\{\alpha_n\}_{n \in \mathbb{N}}$ and $\{\beta_n\}_{n \in \mathbb{N}}$ be arbitrary sequences satisfying $\alpha_n \to \alpha$ and $\beta_n \to \beta$ as $n\to\infty$. 
Since $d(\alpha, \beta) \neq 0$ and it is attained by some $w\in\N$, we see that for all $n \in \mathbb{N}$ large enough we have either $H_{\alpha_n}(w) < 0<G_{\beta_n}(w)$ or $H_{\alpha_n}(w)>0>G_{\beta_n}(w)$. 
Thus, by Proposition \ref{prop:minpoint}, there exists a unique $t_n > 0$ such that $t_n w \in \mathcal{N}_{\alpha_n, \beta_n}$, and $t_n \to 1$ as $n \to \infty$ (see \eqref{tu}). 
Therefore, we get $d(\alpha_n, \beta_n) \leq \En(t_n w)$ for all sufficiently large $n \in \mathbb{N}$, which finally yields the desired conclusion:
\begin{equation*}\label{E<d}
\limsup_{n \to \infty} d(\alpha_n, \beta_n) \leq \limsup_{n \to \infty} \En(t_n w) = \E(w) = d(\alpha, \beta).
\end{equation*}
\end{proof}

\begin{lemma}\label{lem:donti-d}
Let $U \subset \mathbb{R}^2$ be an open set such that 
$$
\{\lambda_1(p)\} \times 
\big((-\infty,\lambda_1(q)]\cup[\beta_*, \infty)\big) \cap U=\emptyset
\quad \text{and}\quad \mathbb{R} \times \{\lambda_1(q)\} \cap U=\emptyset. 
$$
Let $d(\alpha, \beta)$ be attained for any $(\alpha,\beta) \in U$. Moreover, let one of the following assumptions be satisfied:
\begin{enumerate}[label={\rm(\roman*)}]
	\item\label{lem:donti-d:1} $d(\alpha, \beta) > 0$ for any $(\alpha,\beta) \in U$;
	\item\label{lem:donti-d:2} $d(\alpha, \beta) < 0$ for any $(\alpha,\beta) \in U$.
\end{enumerate}
Then $d(\alpha, \beta)$ is continuous on $U$.
\end{lemma}
\begin{proof}
Take any $(\alpha,\beta) \in U$ and let $\{\alpha_n\}_{n \in \mathbb{N}}$ and $\{\beta_n\}_{n \in \mathbb{N}}$ be arbitrary sequences satisfying $\alpha_n \to \alpha$, $\beta_n \to \beta$ and $(\alpha_n,\beta_n) \in U$ (note that $U$ is open). 
Moreover, let $\{u_n\}_{n \in \mathbb{N}}$ be a sequence of minimizers of $d(\alpha_n, \beta_n)$ (the existence follows from the assumption). Since either case \ref{lem:donti-d:1} or case \ref{lem:donti-d:2} holds for all $n \in \mathbb{N}$ and we can assume that each $u_n\ge 0$, we see that $u_n$ is a positive solution of $(GEV;\alpha_n,\beta_n)$ for all $n \in \mathbb{N}$ (see Remark \ref{rem:positivity}).
Let us prove that $\{u_n\}_{n \in \mathbb{N}}$ has a subsequence strongly convergent in $\W$ to some $u_0$ and $d(\alpha,\beta)$ is attained by $u_0$. This will be the desired continuity of $d$.
	
First we show that $u_n$ is bounded in $\W$. 
Note that the boundedness of $\|\nabla u_n\|_p$ is equivalent to the boundedness of $\|u_n\|_p$ since $u_n \in \mathcal{N}_{\alpha_n,\beta_n}$.
Suppose, by contradiction, that $\|u_n\|_p \to \infty$ as $n \to \infty$. 
Applying Lemma \ref{lem:bdd-PS}, we see that $v_n:=u_n/\|u_n\|_p$ converges, up to a subsequence, to $\varphi_p/\|\varphi_p\|_p$ strongly in $\W$, and $\alpha=\lambda_1(p)$. Because we already know that $d(\alpha,\beta) < 0$ in a neighborhood of $\{\lambda_1(p)\} \times (\lambda_1(q),\beta_*)$ (see Propositions \ref{prop:GM2} \ref{prop:GM2:2} and \ref{prop:behavior-gs-2} \ref{prop:behavior-gs-2:2}) our case \ref{lem:donti-d:1} cannot occur. 
However, case \ref{lem:donti-d:2} implies that $G_{\beta_n}(v_n)<0$ and so $G_\beta(\varphi_p)\le 0$, which contradicts the assumption $\lambda_1(q) < \beta < \beta_*$. 
Consequently, $\{u_n\}_{n \in \mathbb{N}}$ is bounded in $\W$. 
Thus, due to Lemma \ref{lem:conv-gs}, there exists a solution $u_0$ of \eqref{eq:D} such that $u_n \to u_0$ strongly in $W_0^{1,p}$, up to a subsequence, and hence  $\liminf\limits_{n\to\infty}\En(u_n)=\liminf\limits_{n\to\infty}d(\alpha_n,\beta_n)$. 

Let us show that $u_0 \not\equiv 0$ in order to get $u_0\in\N$. 
Suppose, by contradiction, that $\|\nabla u_n\|_p \to 0$ as $n \to \infty$. 
Consider case \ref{lem:donti-d:1}. Since $H_{\alpha_n}(u_n)<0$, Lemma~\ref{lem:bdd-PS-2} implies $\beta=\lambda_1(q)$. 
However, this contradicts the assumption $\mathbb{R}\times\{\lambda_1(q)\}\cap U=\emptyset$. 
Consider case \ref{lem:donti-d:2}. Thanks to Lemma~\ref{lem:usc}, we know that $\limsup\limits_{n\to\infty} d(\alpha_n,\beta_n) \le d(\alpha,\beta) < 0$.
Therefore, Lemma~\ref{lem:conv-gs} ensures that $u_0$ is a ground state of $\E$, whence $u_0 \not\equiv 0$ and so $u_0\in\N$. 

Finally, let us show that $d(\alpha,\beta)$ is attained by $u_0$. 
Since $u_n \to u_0$ strongly in $W_0^{1,p}$, we get
\begin{equation*}\label{d<E}
d(\alpha, \beta) \leq \E(u_0) = \liminf\limits_{n \to \infty} \En(u_n) = \liminf\limits_{n \to \infty} d(\alpha_n, \beta_n).
\end{equation*}
As a result, $d$ is lower semicontinuous at $(\alpha,\beta)$. Combining this fact with Lemma \ref{lem:usc}, we deduce that $d$ is continuous at $(\alpha,\beta)$. 
\end{proof}

Now, we are ready to prove Proposition~\ref{prop:property-least-energy}. 
\begin{proof*}{Proposition~\ref{prop:property-least-energy}}
\ref{prop:property-least-energy:2}
Let $\alpha \le \alpha^\prime$ and $\lambda_1(q) < \beta \le \beta^\prime < \beta_*(\alpha^\prime)$. 
Note that the assumption $\lambda_1(q) < \beta_*(\alpha^\prime)$ implies that $\alpha^\prime < \alpha_*$, see Proposition \ref{prop:property-curve} \ref{prop:property-curve:2}. 
It follows from Proposition \ref{prop:GM1} \ref{prop:GM1:2} and Theorem \ref{thm:negative-gs} \ref{thm:negative-gs:1} that $d(\alpha,\beta)<0$ and it is attained by some $u\in\N$.
Since $\E(u) < 0$, we get $G_{\beta^\prime}(u) \le G_\beta(u) < 0$ and so $\|\nabla u\|_q^q/\|u\|_q^q<\beta\le \beta^\prime < \beta_*(\alpha^\prime)$. 
Hence, by the definition of $\beta_*(\alpha^\prime)$, we see that $0 < H_{\alpha^\prime}(u) \le H_{\alpha}(u)$. 
Therefore, according to Proposition \ref{prop:minpoint}, there exists a unique $t^\prime>0$ such that 
$t^\prime u\in \mathcal{N}_{\alpha^\prime,\beta^\prime}$ and 
$E_{\alpha^\prime,\beta^\prime}(t^\prime u) =\min\limits_{s\ge 0} E_{\alpha^\prime,\beta^\prime}(s u)$. As a result, our assertion follows from the following inequalities:
$$
d(\alpha^\prime,\beta^\prime)\le 
E_{\alpha^\prime,\beta^\prime}(t^\prime u)
=\min_{s\ge 0} E_{\alpha^\prime,\beta^\prime}(s u)
\le E_{\alpha^\prime,\beta^\prime}(u)
<\E(u)=d(\alpha,\beta), 
$$
where the last inequality is strict by $(\alpha,\beta)\neq(\alpha^\prime,\beta^\prime)$. 

Note that the method of the proof carries over to the case where $\lambda_1(p) \le \alpha^\prime < \alpha_*$ and $\beta^\prime = \beta_*(\alpha^\prime)$. Indeed, noting that $\beta_*(\alpha^\prime)$ decreases for $\lambda_1(p) \leq \alpha^\prime \leq \alpha_*$ (see Proposition \ref{prop:property-curve} \ref{prop:property-curve:6}), we see that $d(\alpha,\beta) < 0$ and it is attained.

\ref{prop:property-least-energy:3}
Let $\lambda_1(p)<\alpha\le \alpha^\prime$ and $\beta\le \beta^\prime<\beta_*(\alpha^\prime)$. 
Due to Theorem~\ref{thm:MP-Nehari}, Theorem \ref{prop:MP-Nehari-resonant} \ref{prop:MP-Nehari-resonant:2} and Theorem \ref{thm:negative-gs} \ref{thm:negative-gs:1}, $d(\alpha,\beta)$ is attained. If $\lambda_1(q) < \beta \leq \beta^\prime < \beta_*(\alpha^\prime)$, then our conclusion follows from the assertion \ref{prop:property-least-energy:2} proved above.  If $\beta \le \lambda_1(q) < \beta^\prime < \beta_*(\alpha^\prime)$, then $d(\alpha,\beta)>0$ and $d(\alpha^\prime,\beta^\prime)<0$, which yields the desired monotonicity.
Therefore, it remains to consider two cases: either $\lambda_1(p)<\alpha\le \alpha^\prime < \alpha_*$ and $\beta \le \beta^\prime \le \lambda_1(q)$ or $\alpha_* \leq \alpha \le \alpha^\prime$ and $\beta \le \beta^\prime < \lambda_1(q)$. In both cases, $d(\alpha,\beta)>0$ and $d(\alpha^\prime,\beta^\prime) > 0$. Let $u \in \N$ be a ground state of $\E$. It is easy to see that $u \not\in \mathbb{R}\varphi_q$. This yields
$G_{\beta}(u)\ge G_{\beta^\prime}(u)>0>H_\alpha(u)\ge H_{\alpha^\prime}(u)$.
Hence, Proposition \ref{prop:minpoint} implies the existence of a unique $t^\prime>0$ such that $t^\prime u\in \mathcal{N}_{\alpha^\prime,\beta^\prime}$.
Moreover, noting that $\E(u)=\max\limits_{s\ge 0} \E(s u)$, we obtain 
$$
d(\alpha^\prime,\beta^\prime)\le 
E_{\alpha^\prime,\beta^\prime}(t^\prime u)<
\E(t^\prime u) \le \max_{s\ge 0} \E(s u)=
\E(u) = d(\alpha,\beta), 
$$
where the second inequality is strict by $(\alpha,\beta)\neq(\alpha^\prime,\beta^\prime)$. 

\ref{prop:property-least-energy:1} 
Let us divide the $(\alpha,\beta)$-plane into the following four sets (see Fig.\ \ref{fig2}):
\begin{align*}
&A:=\{(\alpha,\beta)\,:\,\alpha<\alpha_*,\ 
\lambda_1(q)<\beta\le \beta_*(\alpha)\,\}, 
&&D:=\{(\alpha,\beta)\,:\, \lambda_1(p)\le \alpha,\ 
\beta_*(\alpha)<\beta\, 
\}, \\
&B:=\{(\alpha,\beta)\,:\, \alpha\le \lambda_1(p),
\ \beta\le \lambda_1(q)\,\}, 
&&C:=\{(\alpha,\beta)\,:\, 
\lambda_1(p)<\alpha,\ \beta\le \lambda_1(q)\, 
\}, 
\end{align*}
where in the set $A$, we denote $\beta_*(\alpha)=+\infty$ 
for $\alpha < \lambda_1(p)$. 
Recall that 
\begin{itemize}
	\item[1.] $d(\alpha,\beta) < 0$ for $(\alpha,\beta) \in A$, see Propositions \ref{prop:GM1} \ref{prop:GM1:2}, \ref{prop:GM3}, and Theorems \ref{thm:negative-gs} \ref{thm:negative-gs:1}, \ref{thm:negative-gs-2}
	(see also Remark \ref{rem:special} for the case $\alpha=\lambda_1(p)$ and $\beta=\beta_*$). 
	\item[2.] $d(\alpha,\beta) = \infty$ for $(\alpha,\beta) \in B$, see Lemma \ref{lem:nonempty-Nehari}.
	\item[3.] $d(\alpha,\beta) \geq 0$ for $(\alpha,\beta) \in C$, see Theorems \ref{thm:MP-Nehari} and \ref{prop:MP-Nehari-resonant}.
	\item[4.] $d(\alpha,\beta) = -\infty$ for $(\alpha,\beta) \in D$, see Theorem \ref{thm:negative-gs} \ref{thm:negative-gs:2}.
\end{itemize}
Let $(\alpha, \beta), (\alpha^\prime, \beta^\prime) \in \mathbb{R}^2$ be such that $\alpha \leq \alpha^\prime$ and $\beta \leq \beta^\prime$, and $(\alpha, \beta) \neq (\alpha^\prime, \beta^\prime)$.
If $(\alpha, \beta) \in B$ or $(\alpha^\prime,\beta^\prime) \in D$, then the assertion is trivial. 
Moreover, if $(\alpha, \beta) \in C$ and $(\alpha^\prime, \beta^\prime) \in A$, then the assertion is trivial, too. Therefore, there are only two remaining cases:
\begin{enumerate}[label={\rm(\alph*)}]
	\item\label{prop:ple:1} $(\alpha, \beta) \in A$ and $(\alpha^\prime, \beta^\prime) \in A$. This case is covered by the assertion \ref{prop:property-least-energy:2} (see the proof of \ref{prop:property-least-energy:2} for the borderline cases).
	\item\label{prop:ple:2} $(\alpha, \beta) \in C$ and $(\alpha^\prime, \beta^\prime) \in C$. This case follows from the assertion \ref{prop:property-least-energy:3} by noting that $d(\alpha,\beta) = 0$ whenever $\alpha \geq \alpha_*$ and $\beta = \lambda_1(q)$, see 	Theorem \ref{prop:MP-Nehari-resonant} \ref{prop:MP-Nehari-resonant:3}-\ref{prop:MP-Nehari-resonant:4}.
\end{enumerate}

\ref{prop:property-least-energy:4} 
Let $\{(\alpha_n,\beta_n)\}_{n \in \mathbb{N}}$ be any sequence convergent to $(\alpha,\beta)$. 
According to Lemma~\ref{lem:usc}, it is sufficient to handle the following three cases: 
\begin{enumerate}[label={\rm(\alph*)}]
	\item\label{prop:ple2:1} $d(\alpha,\beta)=-\infty$;
	\item\label{prop:ple2:2} $\alpha=\lambda_1(p)$, $\beta= \beta_*$ and $d(\alpha,\beta)>-\infty$;
	\item\label{prop:ple2:3} $\alpha\ge \alpha_*$ and $\beta=\lambda_1(q)$.
\end{enumerate}

Case \ref{prop:ple2:1}:
Fix any $R>0$. Since $d(\alpha,\beta)=-\infty$, there exists $u_R\in\N$ such that $\E(u_R)<-R<0$, and hence $H_\alpha(u_R)>0>G_\beta(u_R)$. 
Thus, we may assume that $H_{\alpha_n}(u_R)>0>G_{\beta_n}(u_R)$ for all sufficiently large $n \in \mathbb{N}$. This leads to 
\begin{align*} 
d(\alpha_n,\beta_n)\le \min_{s\ge 0} \En (s u_R)
\le \En(u_R) =\E(u_R)+o(1)<-R+o(1), 
\end{align*}
which implies that $\limsup\limits_{n\to\infty}d(\alpha_n,\beta_n)\le -R$. 
Since $R>0$ is arbitrary, we conclude that  $\lim\limits_{n\to\infty}d(\alpha_n,\beta_n) = -\infty = d(\alpha,\beta)$. 

Case \ref{prop:ple2:2}:
Recall that $d(\alpha,\beta)<0$. Fix any sufficiently small $\varepsilon>0$. Then, we can find $u\in\N$ such that $\E(u)<d(\alpha,\beta)+\varepsilon<0$. Arguing as in case \ref{prop:ple2:1}, we have 
\begin{align*} 
d(\alpha_n,\beta_n)\le \min_{s\ge 0} \En (s u)
\le \En(u) =\E(u)+o(1)<d(\alpha,\beta)+\varepsilon+o(1)
\end{align*}
for sufficiently large $n \in \mathbb{N}$, which implies that $\limsup\limits_{n\to\infty}d(\alpha_n,\beta_n)\le d(\alpha,\beta)$.

Case \ref{prop:ple2:3}:
If $\{(\alpha_n,\beta_n)\}_{n \in \mathbb{N}}$ has a subsequence contained in 
$(\lambda_1(p),+\infty)\times (-\infty,\lambda_1(q))$ or 
in $(\lambda_1(p),\alpha_*)\times\{\lambda_1(q)\}$, then our assertion follows from Proposition~\ref{prop:behavior-gs} \ref{prop:behavior-gs:3}. Otherwise, the claim is trivial because $d(\alpha_n,\beta_n)\le 0=d(\alpha,\beta)$. 

\ref{prop:property-least-energy:5}
The assertion follows from Lemma~\ref{lem:donti-d}.
\end{proof*}

\par
\bigskip
\noindent
{\bf Acknowledgements.} This work was supported by JSPS KAKENHI Grant Number 15K17577. 
The first author wishes to thank Tokyo University of Science, where this research was initiated, for the invitation and hospitality. The work of the first author was also supported by the project LO1506 of the Czech Ministry of Education, Youth and Sports.

\appendix
\section{Appendix}\label{sec:appendix}

\begin{proposition}\label{prop:LID}
Let $\Omega \subset \mathbb{R}^N$ be a bounded domain, $N \geq 1$, and $1<q<p<\infty$. Let $\varphi$ and $\psi$ be (nontrivial) eigenfunctions of the $p$-Laplacian and $q$-Laplacian in $\Omega$ under zero Dirichlet boundary condition, respectively. 
Then $\varphi$ and $\psi$ are linearly independent, i.e., $\varphi \not\in\mathbb{R}\psi$ 
(equivalently, $\psi \not\in\mathbb{R}\varphi$). 
\end{proposition}
\begin{proof}
In the case $N=1$, the result follows from \cite[Lemma A.1]{BobkovTanaka2016} or \cite[Lemma 4.3]{KTT}.
Let $N \geq 2$. Suppose, by contradiction, that $\varphi \in\mathbb{R}\psi$. Evidently, we can assume that $\varphi \equiv \psi$.
Since $\varphi\in C^{1,\gamma}(\Omega)$ for some $\gamma \in (0,1)$ (cf.\ \cite{tolksdorf}, where a $L^\infty$-bound can be obtained by the bootstrap arguments, for instance, as in \cite[Lemma 3.2]{drabekkufner}), $\varphi = 0$ on $\partial \Omega$, and $\varphi \not\equiv 0$, there exists at least one point of global extrema of $\varphi$ and, in view of the translation invariance of $p$-Laplacian, we can assume that this point is $0$. Moreover, considering $-\varphi$ instead of $\varphi$, if necessary, we can assume that $\varphi(0) > 0$. 
Let us denote by $\phi$ a restriction of $\varphi$ to a component $\Omega'$ of the set $\{x \in \Omega: \varphi(x) > 0\}$ such that $0 \in \Omega'$. Then $\phi \in W_0^{1,p}(\Omega')$, see \cite[Lemma 5.6]{cuesta}. Finally, for simplicity of notation, let us assume that $\Omega' \equiv \Omega$.
	
The following blow-up arguments are based on the article \cite{garcia}. 
Take any $\lambda > 0$ and consider the function
\begin{equation}\label{ul}
u_\lambda(x) := \frac{\phi(0) - \phi(\lambda x)}{\lambda^\frac{p}{p-1}}, 
\quad x \in \Omega_\lambda,
\end{equation}
where $\Omega_\lambda = \{x \in \mathbb{R}^N:~ 
\lambda x \in \Omega\}$. 
It was proved in \cite[Lemma~5]{garcia} that there exists a sequence $\lambda_n \to 0$ as $n \to \infty$ such that $u_{\lambda_n} \to \bar{u}$ in $C^1_{\text{loc}}(\mathbb{R}^N)$, where $\bar{u}$ is a nonnegative weak solution of
\begin{equation}\label{eq}
\Delta_p u = \lambda_1(p) \phi(0)^{p-1} ~~(=\text{const} > 0)
\quad \text{in }~ \mathbb{R}^N
\end{equation}
such that $\bar{u}(0) = 0$.
Let us show now that $\bar{u}$ is simultaneously a $q$-harmonic function in $\mathbb{R}^N$.
Indeed, noting that 
$$
\Delta_q u_\lambda = -\frac{1}{\lambda^\frac{p(q-1)}{p-1}} 
\Delta_q(\phi(\lambda x)) = 
\frac{\lambda^q}{\lambda^\frac{p(q-1)}{p-1}} \lambda_1(q) \phi(\lambda x)^{q-1} = \lambda^\frac{p-q}{p-1} \lambda_1(q) \phi(\lambda x)^{q-1},
$$
we see that each $u_\lambda$ defined by \eqref{ul} weakly satisfies the equation
$$
\Delta_q u = \lambda^\frac{p-q}{p-1} \lambda_1(q) \phi(\lambda x)^{q-1}
\quad \text{in }~ \Omega_\lambda.
$$
Since $\phi$ is uniformly bounded and $p>q$, we consider the sequence $\{\lambda_n\}_{n \in \mathbb{N}}$ as above and, passing to the limit as $n \to \infty$, deduce that $\bar{u}$ weakly satisfies
$$
\Delta_q \bar{u} = 0 \quad \text{in }~ \mathbb{R}^N,
$$
that is, $\bar{u}$ is a $q$-harmonic function in $\mathbb{R}^N$.
Recalling that $\bar{u}$ is nonnegative and $\bar{u}(0) = 0$, we derive from the strong maximum principle (cf.\ \cite[Theorem 5.3.1]{PS}) that $\bar{u} \equiv 0$ in $\mathbb{R}^N$.
However, it contradicts the equation \eqref{eq}.	
\end{proof}



\addcontentsline{toc}{section}{\refname}
\small
\begin{thebibliography}{99}
	
	\bibitem{papa2}
	\textsc{Aizicovici, S., Papageorgiou, N.~S., and Staicu, V.}
	Nodal solutions for $(p,2)$-equations.
	\textit{Transactions of the American Mathematical Society},
	367 (2015), 7343--7372.
	\href{http://dx.doi.org/10.1090/S0002-9947-2014-06324-1}{DOI:10.1090/S0002-9947-2014-06324-1}
	
	\bibitem{alves}
	\textsc{Alves, M.~J., Assun\c{c}\~{a}o, R.~B., and Miyagaki, O.~H.}
	Existence result for a class of quasilinear elliptic equations with $(p-q)$-Laplacian and vanishing potentials.
	\textit{Illinois Journal of Mathematics},
	59(3) (2015), 545--575.
	\url{http://projecteuclid.org/euclid.ijm/1475266397}
	
	\bibitem{anane1987}
	\textsc{Anane, A.}
	Simplicit\'e et isolation de la premiere valeur propre du $p$-laplacien avec poids.
	\textit{Comptes Rendus de l'Acad\'emie des Sciences-Series I-Mathematics},
	305(16) (1987), 725--728.
	\url{http://gallica.bnf.fr/ark:/12148/bpt6k57447681/f27}	
	
	\bibitem{bellazzini}
	\textsc{Bellazzini, J., and Visciglia, N.}
	Max-min characterization of the mountain pass energy level for a class of variational problems. 
	\textit{Proceedings of the American Mathematical Society}, 
	138(9) (2010), 3335--3343. 
	\href{http://dx.doi.org/10.1090/S0002-9939-10-10415-8}{DOI:10.1090/S0002-9939-10-10415-8}
	
	\bibitem{BenciDerrick}
	\textsc{Benci, V., D'Avenia, P., Fortunato, D., and Pisani, L.} 
	Solitons in Several Space Dimensions: Derrick's Problem and Infinitely Many Solutions.
	\textit{Archive for Rational Mechanics and Analysis},
	154(4) (2000), 297--324.
	\href{http://dx.doi.org/10.1007/s002050000101}{DOI:10.1007/s002050000101}
	
    \bibitem{BobkovTanaka2015}
	\textsc{Bobkov, V., and Tanaka, M.}
	On positive solutions for $(p,q)$-Laplace equations with two parameters.
	\textit{Calculus of Variations and Partial Differential Equations},
	54(3) (2015), 3277--3301. 
	\href{http://dx.doi.org/10.1007/s00526-015-0903-5}{DOI:10.1007/s00526-015-0903-5} 
	
	\bibitem{BobkovTanaka2016}
	\textsc{Bobkov, V., and Tanaka, M.}
	On sign-changing solutions for $(p,q)$-Laplace equations with two parameters.
	\textit{Advances in Nonlinear Analysis}.
	\href{http://dx.doi.org/10.1515/anona-2016-0172}{DOI:10.1515/anona-2016-0172}
	
	\bibitem{busheled}
	\textsc{Bushell, P. J., and Edmunds, D. E.}
	Remarks on generalised trigonometric functions. 
	\textit{Rocky Mountain Journal of Mathematics}
	42.1 (2012): 25--57.
	\href{http://dx.doi.org/10.1216/rmj-2012-42-1-25}{DOI:10.1216/rmj-2012-42-1-25}	
	
	\bibitem{cahnhill}
	\textsc{Cahn, J.~W., and Hilliard, J.~E.} 
	Free energy of a nonuniform system. I. Interfacial free energy.
	\textit{The Journal of chemical physics},
	28(2) (1958), 258--267.
	\href{http://dx.doi.org/10.1063/1.1744102}{DOI:10.1063/1.1744102}
	
	\bibitem{CherIl}
	\textsc{Cherfils, L., and Il'yasov, Y.}
	On the stationary solutions of generalized reaction diffusion equations with $p\&q$-laplacian.
	\textit{Communications on Pure and Applied Mathematics},
	4(1) (2005), 9--22.
	\href{http://dx.doi.org/10.3934/cpaa.2005.4.9}{DOI:10.3934/cpaa.2005.4.9}
	
	\bibitem{chueshov}
	\textsc{Chueshov, I., and Lasiecka, I.}
	Existence, uniqueness of weak solutions and global attractors for a class of nonlinear $2$D Kirchhoff-Boussinesq models. 
	\textit{Discrete and Continuous Dynamical Systems},
	15(3) (2006), 777--809.
	\href{http://dx.doi.org/10.3934/dcds.2006.15.777}{DOI:10.3934/dcds.2006.15.777}
	
	\bibitem{cuesta}
	\textsc{Cuesta, M., de Figueiredo, D., and Gossez, J.-P.}
	The beginning of the Fu\v{c}ik spectrum for the $p$-Laplacian. 
	\textit{Journal of Differential Equations},
	159(1) (1999), 212--238.
	\href{http://dx.doi.org/10.1006/jdeq.1999.3645}{DOI:10.1006/jdeq.1999.3645}
	
	\bibitem{damascelli}
	\textsc{Damascelli, L., and Sciunzi, B.} 
	Regularity, monotonicity and symmetry of positive solutions of $m$-Laplace equations. 
	\textit{Journal of Differential Equations}, 206(2) (2004), 483--515.
	\href{http://dx.doi.org/10.1016/j.jde.2004.05.012}{DOI:10.1016/j.jde.2004.05.012}
	
	\bibitem{drabek}
	\textsc{Dr\'abek, P.} 
	Geometry of the energy functional and the Fredholm alternative for the $p$-Laplacian in higher dimensions. 
	\textit{2001-Luminy Conference on Quasilinear Elliptic and Parabolic Equations and System},
	\textit{Electronic Journal of Differential Equations}, Conference 08 (2002), 103--120.
	\url{https://ejde.math.txstate.edu/conf-proc/08/d1/drabek.pdf}
	
	\bibitem{drabekkufner}
	\textsc{Dr{\'a}bek, P., Kufner, A., and Nicolosi, F.}
	\textit{Quasilinear elliptic equations with degenerations and singularities}.
	Walter de Gruyter, 1997.
		
	\bibitem{drabekmilota}
	\textsc{Dr{\'a}bek, P., and Milota, J.}
	\textit{Methods of nonlinear analysis: applications to differential equations}.
	Springer, 2013.

	\bibitem{faria}
	\textsc{Faria, L.~F., Miyagaki, O.~H., and Motreanu, D.} 
	Comparison and positive solutions for problems with the $(p,q)$-Laplacian and a convection term. 
	\textit{Proceedings of the Edinburgh Mathematical Society (Series 2)}, 
	57(03) (2014), 687--698.
	\href{http://dx.doi.org/10.1017/S0013091513000576}{DOI:10.1017/S0013091513000576}

	\bibitem{figueiredo}
	\textsc{Figueiredo, G.~M.}
	Existence of positive solutions for a class of $p\&q$ elliptic problems with critical growth on $\mathbb{R}^N$. 
	\textit{Journal of Mathematical Analysis and Applications}, 
	378(2) (2011), 507--518.
	\href{http://dx.doi.org/10.1016/j.jmaa.2011.02.017}{DOI:10.1016/j.jmaa.2011.02.017}	
	
	\bibitem{garcia}
	\textsc{Garc\'ia-Meli\'an, J.} 
	On the behavior of the first eigenfunction of the $p$-Laplacian near its critical points. 
	\textit{Bulletin of the London Mathematical Society},
	35(3) (2003), 391--400.
	\href{http://dx.doi.org/10.1112/S0024609303001966}{DOI:10.1112/S0024609303001966}
	
	\bibitem{takac}
	\textsc{Fleckinger-Pell\'e J., and Tak\'a\v{c}, P.}
	An improved Poincar\'e inequality and the $p$-Laplacian at resonance for $p>2$.
	\textit{Advances in Differential Equations},
	7(8) (2002), 951--971.
	\url{http://projecteuclid.org/euclid.ade/1356651685}
	
	\bibitem{ilfunc}
	\textsc{Il'yasov, Y.~S.}
	Bifurcation calculus by the extended functional method.
	\textit{Functional Analysis and Its Applications}, 
	41(1) (2007), 18--30.
	\href{http://dx.doi.org/10.1007/s10688-007-0002-2}{DOI:10.1007/s10688-007-0002-2}
	
	\bibitem{ilyasovkaye}
	\textsc{Il'yasov Y., Silva, K.}
	On branches of positive solutions for $p$-Laplacian problems at the extreme value of Nehari manifold method.
	\textit{arXiv:}\href{https://arxiv.org/abs/1704.02477}{1704.02477} (2017).
	
	\bibitem{jeanjean}
	\textsc{Jeanjean, L., and Tanaka, K.}
	A remark on least energy solutions in $\mathbb{R}^N$. 
	\textit{Proceedings of the American Mathematical Society}, 
	131(8) (2003), 2399--2408.
	\url{http://www.jstor.org/stable/1194267}
	
	\bibitem{KTT}
	\textsc{Kajikiya, R., Tanaka, M., and Tanaka, S.}
	Bifurcation of positive solutions for the one dimensional $(p, q)$-Laplace equation.
	\textit{Electronic Journal of Differential Equations}, 
	2017(107) (2017), 1--37.
	\url{https://ejde.math.txstate.edu/Volumes/2017/107/kajikija.pdf}
	
	\bibitem{Lieberman}
	\textsc{Lieberman, G.~M.}
	Boundary regularity for solutions of degenerate elliptic equations.
	\textit{Nonlinear Analysis: Theory, Methods \& Applications}, 
	12(11) (1988), 1203--1219.
	\href{http://dx.doi.org/10.1016/0362-546X(88)90053-3}{DOI:10.1016/0362-546X(88)90053-3}
	
	\bibitem{L}
	\textsc{Lieberman, G.~M.}
	The natural generalization of the natural conditions of Ladyzhenskaya and Ural'tseva for elliptic equations.
	\textit{Communications in Partial Differential Equations},
	16(2-3) (1991), 311--361.
	\href{http://dx.doi.org/10.1080/03605309108820761}{DOI:10.1080/03605309108820761}
	
	\bibitem{marcomasconi2017}
	\textsc{Marano, S., and Mosconi, S.}
	Some recent results on the Dirichlet problem for $(p, q)$-Laplace equations. 
	\textit{arXiv:}\href{https://arxiv.org/abs/1703.10831}{1703.10831} (2017).
	
	\bibitem{marano2013}
	\textsc{Marano, S.~A., and Papageorgiou N.~S.}
	Constant-sign and nodal solutions of coercive $(p, q)$-Laplacian problems. 
	\textit{Nonlinear Analysis: Theory, Methods \& Applications},
	77 (2013), 118--129.
	\href{http://dx.doi.org/10.1016/j.na.2012.09.007}{DOI:10.1016/j.na.2012.09.007}
	
	\bibitem{MT}
	\textsc{Motreanu, D., and Tanaka, M.}
	On a positive solution for $(p,q)$-Laplace equation with indefinite weight.
	\textit{Minimax Theory and its Applications},
	1(1) (2016), 1--20.
	\url{http://www.heldermann-verlag.de/mta/mta01/mta0001-b.pdf}
	
	\bibitem{pohozaev}
	\textsc{Pohozaev, S. I.}
	Nonlinear variational problems via the fibering method. 
	\textit{Handbook of differential equations: stationary partial differential equations}, 5  (2008), 49--209.
	\href{http://dx.doi.org/10.1016/S1874-5733(08)80009-5}{DOI:10.1016/S1874-5733(08)80009-5}
	
	\bibitem{PS}
	\textsc{Pucci, P., and Serrin, J.}
	\textit{The maximum principle},
	Springer, 2007.
	\href{http://dx.doi.org/10.1007/978-3-7643-8145-5}{DOI:10.1007/978-3-7643-8145-5}
	
	\bibitem{sun}
	\textsc{Sun, J., Chu, J., and Wu, T.~F.}
	Existence and multiplicity of nontrivial solutions for some biharmonic equations with $p$-Laplacian. 
	\textit{Journal of Differential Equations}, 
	262(2) (2017), 945-977.
	\href{http://dx.doi.org/10.1016/j.jde.2016.10.001}{DOI:10.1016/j.jde.2016.10.001}
	
	\bibitem{takac2}
	\textsc{Tak\'a\v{c}, P.}
	On the Fredholm alternative for the $p$-Laplacian at the first eigenvalue.
	\textit{Indiana University Mathematics Journal},
	51 (2002), 187--238.
	\href{http://dx.doi.org/10.1512/iumj.2002.51.2156}{DOI:10.1512/iumj.2002.51.2156}
	
	\bibitem{T-Uniq}
	\textsc{Tanaka, M.}
	Uniqueness of a positive solution and existence of a sign-changing solution for $(p,q)$-Laplace equation.
	\textit{Journal of Nonlinear Functional Analysis},
	2014 (2014), 1--15. 
	\url{http://jnfa.mathres.org/issues/JNFA201414.pdf}
	
	\bibitem{T-2014}
	\textsc{Tanaka, M.}
	Generalized eigenvalue problems for $(p,q)$-Laplacian with indefinite weight.
	\textit{Journal of Mathematical Analysis and Applications},
	419(2) (2014), 1181--1192.
	\href{http://dx.doi.org/10.1016/j.jmaa.2014.05.044}{DOI:10.1016/j.jmaa.2014.05.044}
	  
	\bibitem{tolksdorf}
	\textsc{Tolksdorf, P.}
	Regularity for a more general class of quasilinear elliptic equations. 
	\textit{Journal of Differential equations},
	51(1) (1984), 126--150.
	\href{http://dx.doi.org/10.1016/0022-0396(84)90105-0}{DOI:10.1016/0022-0396(84)90105-0}
	  
	\bibitem{yin}
	\textsc{Yin, H., and Yang, Z.} 
	A class of $p-q$-Laplacian type equation with concave-convex nonlinearities in bounded domain. 
	\textit{Journal of Mathematical Analysis and Applications}, 
	382(2) (2011), 843--855.
	\href{http://dx.doi.org/10.1016/j.jmaa.2011.04.090}{DOI:10.1016/j.jmaa.2011.04.090}
	  
	\bibitem{zakharov}
	\textsc{Zakharov, V.~E.}
	Collapse of Langmuir waves.
	\textit{Soviet Journal of Experimental and Theoretical Physics},
	35(5) (1972), 908-914.
	\url{http://jetp.ac.ru/cgi-bin/dn/e_035_05_0908.pdf}

\end{thebibliography}
\end{document}